\documentclass[11pt]{article}

\usepackage{comment,url,algorithm,algorithmic,graphicx,subcaption,relsize}
\usepackage{amssymb,amsfonts,amsmath,amsthm,amscd,dsfont,mathrsfs,mathtools,nicefrac}
\usepackage{float,psfrag,epsfig,color,xcolor,url,hyperref}
\usepackage{epstopdf,bbm,mathtools,enumitem}
\usepackage[toc,page]{appendix}
\usepackage[mathscr]{euscript}
\usepackage{xspace}

\usepackage[top=1in, bottom=1in, left=1in, right=1in]{geometry}

\def\balign#1\ealign{\begin{align}#1\end{align}}
\def\baligns#1\ealigns{\begin{align*}#1\end{align*}}
\def\balignat#1\ealign{\begin{alignat}#1\end{alignat}}
\def\balignats#1\ealigns{\begin{alignat*}#1\end{alignat*}}
\def\bitemize#1\eitemize{\begin{itemize}#1\end{itemize}}
\def\benumerate#1\eenumerate{\begin{enumerate}#1\end{enumerate}}

\newenvironment{talign*}
 {\csname align*\endcsname}
 {\endalign}
\newenvironment{talign}
 {\csname align\endcsname}
 {\endalign}

\def\balignst#1\ealignst{\begin{talign*}#1\end{talign*}}
\def\balignt#1\ealignt{\begin{talign}#1\end{talign}}

\let\originalleft\left
\let\originalright\right
\renewcommand{\left}{\mathopen{}\mathclose\bgroup\originalleft}
\renewcommand{\right}{\aftergroup\egroup\originalright}

\def\tinycitep*#1{{\tiny\citep*{#1}}}
\def\tinycitealt*#1{{\tiny\citealt*{#1}}}
\def\tinycite*#1{{\tiny\cite*{#1}}}
\def\smallcitep*#1{{\scriptsize\citep*{#1}}}
\def\smallcitealt*#1{{\scriptsize\citealt*{#1}}}
\def\smallcite*#1{{\scriptsize\cite*{#1}}}

\def\<{\left\langle} %
\def\>{\right\rangle}

\DeclareSymbolFont{rsfs}{U}{rsfs}{m}{n}
\DeclareSymbolFontAlphabet{\mathscrsfs}{rsfs}

\ifdefined\nonewproofenvironments\else
\ifdefined\ispres\else
\newtheorem{theorem}{Theorem}[section]
\newtheorem{lemma}[theorem]{Lemma}
\newtheorem{corollary}[theorem]{Corollary}

\renewenvironment{proof}{\noindent\textbf{Proof.}\hspace*{.3em}}{\qed\\}
\newenvironment{proof-sketch}{\noindent\textbf{Proof Sketch}
  \hspace*{1em}}{\qed\bigskip\\}
\newenvironment{proof-idea}{\noindent\textbf{Proof Idea}
  \hspace*{1em}}{\qed\bigskip\\}
\newenvironment{proof-of-lemma}[1][{}]{\noindent\textbf{Proof of Lemma {#1}}
  \hspace*{1em}}{\qed\\}
\newenvironment{proof-of-theorem}[1][{}]{\noindent\textbf{Proof of Theorem {#1}}
  \hspace*{1em}}{\qed\\}
\newenvironment{proof-attempt}{\noindent\textbf{Proof Attempt}
  \hspace*{1em}}{\qed\bigskip\\}
\newenvironment{proofof}[1]{\noindent\textbf{Proof of {#1}}
  \hspace*{1em}}{\qed\bigskip\\}

\fi

\newtheorem{proposition}[theorem]{Proposition}

\newtheorem{assumption}{Assumption}
\fi
\makeatletter
\@addtoreset{equation}{section}
\makeatother

\hypersetup{
  colorlinks,
  linkcolor={red!50!black},
  citecolor={blue!50!black},
  urlcolor={blue!80!black}
}

\mathtoolsset{showonlyrefs}
\usepackage{natbib}
\allowdisplaybreaks[4]

\newcommand{\rmd}{\mathrm d}

\def\wass{{\sf W}}

\usepackage{tikz}
\usetikzlibrary{arrows.meta}

\definecolor{darkmidnightblue}{rgb}{0.0, 0.2, 0.4}
\definecolor{darkpowderblue}{rgb}{0.0, 0.2, 0.6}
\definecolor{dukeblue}{rgb}{0.0, 0.0, 0.61}

\hypersetup{
    colorlinks = true,
    citecolor= midnightblue,
    urlcolor= black,
    breaklinks=true,
    linkcolor = midnightblue,
    linkbordercolor = {white},
}

\definecolor{darkmidnightblue}{HTML}{003366}    
\definecolor{midnightblue}{HTML}{0059b3}
\definecolor{chromered}{HTML}{f14233}

\begin{document}
\title{Robustness of Diffusion Models under Distribution Shift}
\author{Wei Luo\thanks{Department of Mathematical Sciences, Tsinghua University,  \texttt{luow25@mails.tsinghua.edu.cn
}}  \qquad Neil K. Chada\thanks{Department of Mathematics, City University of Hong Kong, Hong Kong SAR, \texttt{neilchada123@gmail.com}}
\qquad Shijie Zhang\thanks{Artificial Intelligence Research Institute, Shenzhen MSU-BIT University, \texttt{shijie.z@smbu.edu.cn}} 
\qquad Lu Yu\thanks{Corresponding author. Data Science Institute, Shandong University, \texttt{lu.yu@sdu.edu.cn}}}

\maketitle

\begin{abstract}

Score-based diffusion models are increasingly considered in settings where the underlying data distribution may differ from the training distribution, yet existing theoretical guarantees largely focus on the no-shift setting. In this work, we study robust score estimation under
Wasserstein perturbations of a reference distribution. For the
Ornstein--Uhlenbeck diffusion, we show that robust estimation decomposes into
two fundamental components: the statistical cost of learning the reference
distribution and the intrinsic cost of distribution shift. The latter scales
quadratically with the Wasserstein radius, and this dependence is minimax
optimal. We construct an explicit finite-sample estimator achieving the
resulting robust minimax rate without knowing the shift radius. When the
reference distribution lies on an unknown low-dimensional subspace, the
statistical term adapts to the intrinsic dimension while the shift cost remains
unchanged. Finally, we show that the same decomposition governs positive-time
reverse sampling and obtain matching minimax guarantees in KL divergence.
Together, these results characterize how finite data, intrinsic dimension, and
distribution shift affect the robustness of score-based diffusion models.

\end{abstract}

\begingroup
\color{midnightblue} 
\tableofcontents
\endgroup

\section{Introduction}

Diffusion models have become central to modern generative modeling
~\citep{dhariwal2021diffusion,ho2020denoising,song2020denoising,xu2022geodiff}.
A common formulation is the score-based diffusion model (SDM), which
constructs a generative process by reversing a stochastic differential
equation (SDE)~\citep{song2020score}. 
At each diffusion time, the reverse dynamics rely on the score function of
the corresponding noisy data distribution. Since the score is unknown, it
must be estimated from data, typically using a neural network in
practical implementations.

A substantial body of theory has studied the statistical properties of
learned scores and their impact on the reverse generative process. Existing results provide guarantees for score estimation and sampling under a fixed data distribution, where the estimator is evaluated on the same distribution used for training (see, e.g.,
~\citet{gao2023wasserstein,chen2022sampling,li2024sharp,yu2026diffusion}).
However, pretrained diffusion models are increasingly reused or adapted across changing data sources and new target domains~\citep{rombach2022high,ouyang2024transfer,gao2023back}. Such distribution shifts arise from domain changes, temporal drift, sampling bias, and variations in data collection~\citep{moreno2012unifying,koh2021wilds,ben2010theory,torralba2011unbiased,cheng2026generative}. This motivates a fundamental question: how robust are diffusion models when the deployment distribution differs from the reference distribution?

We study robustness through the score function, which governs the reverse diffusion dynamics and thus provides a natural entry point for analyzing generative robustness. Unlike standard score estimation, distribution shift introduces a mismatch between the distribution used for learning and the distribution encountered after deployment. The resulting robust error contains two components: the statistical error from learning the reference distribution and the additional error caused by the unknown shift. Our goal is to characterize these contributions and their optimal scaling.

To quantify distribution shifts, we model the allowed perturbations through Wasserstein neighborhoods of the reference distribution. Wasserstein distance has been widely used in diffusion sampling and stability analyses (see, e.g.,~\citet{gao2023wasserstein,yu2025advancing,beyler2025convergence}). At the same time, Wasserstein proximity does not directly imply proximity of
score functions, which depend on local variations of the underlying density. Forward diffusion regularizes this discrepancy through Gaussian smoothing, but the regularization becomes weaker at early diffusion times. Quantifying how a Wasserstein perturbation of the initial distribution propagates to the score along the diffusion trajectory is therefore a key step toward a statistical theory of diffusion robustness.

Specifically, let $\mathcal Q_\varepsilon(p_0)$ denote the Wasserstein-2 neighborhood of radius \(\varepsilon\) around the reference distribution $p_0,$ where each
 $q_0\in\mathcal Q_\varepsilon(p_0)$ represents a possible shifted
distribution.
Applying the same forward diffusion to $p_0$ and $q_0$ gives
the corresponding noisy marginals $p_t$ and $q_t$. We study score estimation
over diffusion times $t\in[t_0,T]$, where $t_0>0$ specifies the earliest diffusion time considered and avoids the singular zero-noise regime.
A score estimator $\hat s(x,t)$ is learned from samples of $p_0$, but under
distribution shift it is evaluated on $q_t$ and compared with the shifted
score $\nabla\log q_t$. We quantify robustness through the worst-case time-averaged score risk
\[
\mathcal R_{\varepsilon,p_0}(\hat s)
:=  \sup_{q_0\in\mathcal Q_\varepsilon(p_0)}
\mathbb{E}\left[
\frac{1}{T-t_0}
\int_{t_0}^T
\mathbb E_{X\sim q_t}
\left[
\|\hat s(X,t)-\nabla\log q_t(X)\|_2^2
\right]\,\rmd t
\right],
\]
where the outer expectation is over the training data and any randomness in
the estimator.

When \(\hat s\) is learned from finite samples of \(p_0\), this robust risk captures two sources of uncertainty: the statistical error from learning the reference distribution and the shift error arising from the unknown deployment distribution \(q_0\), which changes both the evaluation distribution and the target score. Understanding their interaction, the intrinsic cost of distribution shift, and whether such effects extend to generation remains beyond existing score-estimation theory.

These considerations lead to the central question of this work:
\begin{center}
\textit{
What are the statistical and intrinsic costs of distribution shift in
diffusion models, and how do they propagate from score estimation to
generation?
}
\end{center}
To our knowledge, this is the first work to develop a minimax theory for
robust SDMs under distribution shift.
We focus on the Ornstein--Uhlenbeck (OU)
process as a canonical setting in which these costs can be characterized
sharply.
We develop a minimax theory that reveals a fundamental decomposition
of the robust error into a statistical estimation cost and an intrinsic
distribution-shift cost. 
The latter
is quadratic in the Wasserstein radius and minimax optimal, while the former
adapts to low-dimensional structure. We further show that this decomposition
extends to positive-time reverse sampling with matching KL guarantees. Our main
contributions are summarized as follows.

\begin{itemize}

\item \textbf{Intrinsic cost of distribution shift.}
We establish a comparison principle for score estimators induced by estimated
initial laws under the OU semigroup. The resulting bound separates reference
distribution estimation error from an intrinsic shift cost that scales
quadratically with the Wasserstein radius. A matching minimax lower bound shows
that this dependence is unavoidable.

\item \textbf{Finite-sample robust score estimation.}
We construct an explicit estimator from $n$ i.i.d.\ samples from the reference distribution and prove, under compact support and for fixed dimension,
$$   \mathcal R_{\varepsilon,p_0}(\hat s_n)
    \lesssim
    \frac{1}{n(T-t_0)t_0^{d/2}}
    +
    \frac{\varepsilon^2}{t_0(T-t_0)}.
$$
The estimator does not require knowledge of the shift radius and achieves the
optimal robust rate.

\item \textbf{Adaptation to intrinsic dimension.}
We extend the theory to reference distributions supported on an unknown
$k$-dimensional linear subspace of an ambient $d$-dimensional space. In this
setting, 
we show that the statistical term depends on $k$ rather than the ambient
dimension $d$, while the shift cost remains unchanged.
Matching lower bounds establish
the minimax optimality of this intrinsic-dimensional rate.

\item \textbf{Robust generation.}
Finally, we extend the score estimation guarantees to positive-time reverse
sampling. The resulting KL error admits the same decomposition into a
finite-sample estimation term and a quadratic distribution shift term, with
matching minimax lower bounds. Thus, the statistical and intrinsic costs at the
score level persist in the generated distribution.

\end{itemize}

Together, these results provide a unified characterization of diffusion
robustness under Wasserstein shifts: more data and structural assumptions reduce
the statistical difficulty of learning the reference distribution, but cannot
remove the intrinsic uncertainty caused by the distribution shift.
A controlled numerical study in Appendix~\ref{app:three-dimensional-experiment} illustrates the predicted finite-sample decay and quadratic shift behavior for both the time-averaged score risk and the generation error in KL divergence
at time $t_0$.
All proofs are deferred to the appendix.

\paragraph{Additional Related Work.}
A substantial literature studies the learnability of score functions in diffusion models. Early work established $L^2$-estimation guarantees in terms of hypothesis-class complexity~\citep{block2020generative}, while subsequent studies derived sharper sample-complexity bounds using expressive neural networks, bounded-parameter local approximations for sub-Gaussian targets, and neural tangent kernel techniques~\citep{gupta2024improved,cole2024score,han2024neural}. Complementarily, \citet{chewi2025ddpm} relate DDPM score estimation to classical distribution learning. More refined guarantees have been obtained under additional structural assumptions, including Gaussian-mixture models~\citep{gatmiry2024learning,chen2024learning}, data supported on unknown linear subspaces~\citep{chen2023score}, and distributions concentrated on low-dimensional manifolds~\citep{azangulov2024convergence,yakovlev2025generalization,yakovlev2025implicit}. A related line of work establishes minimax rates for score estimation and sampling under assumptions such as sub-Gaussian tails~\citep{zhang2024minimax,stephanovitch2025generalization,cai2025minimax,dou2024optimal,fu2024unveil}, explicit density lower bounds~\citep{oko2023diffusion,fan2025optimal}, compact support~\citep{stephanovitch2025generalization,dou2024optimal,tang2024adaptivity}, and, more recently, heavy-tailed target distributions~\citep{yu2026diffusion}.

Another line of work studies how diffusion models and their learned representations can be transferred or adapted across related domains~\citep{cheng2025provable,ouyang2024transfer,song2025domain,kleutgens2025guided,moon2022fine,xie2023difffit,han2023svdiff}. For example, some methods adapt a pretrained source-domain score through additional guidance~\citep{ouyang2024transfer,kleutgens2025guided}, while others establish sample-efficient transfer guarantees for conditional diffusion models through representation learning~\citep{cheng2025provable}. Complementing these adaptation-oriented approaches, \citet{yu2026limits} characterize the reliability and limitations of reusing a frozen source-domain latent representation under distribution shift.

\paragraph{Notation.}
Let $\mathbb{R}^d$ be the $d$-dimensional Euclidean space and $I_d$ the identity matrix.
For a measurable $f:A\to\mathbb R^k$, write
$\|f\|_{L^2(A)}^2:=\int_A\|f(x)\|_2^2\,\rmd x$
and $\|f\|_2:=\|f\|_{L^2(\mathbb R^d)}$; for a measure $\mu$,
$\|f\|_{L^2(\mu)}^2:=\int\|f(x)\|_2^2\,\mu(\rmd x)$.
Let $\mathcal P_q(\mathbb R^d)$ denote the set of probability measures with finite $q$-th moment:
$\mathcal P_q(\mathbb R^d):=\{\mu:\int\|x\|_2^q\,\mu(\rmd x)<\infty\}$.
For $\mu,\nu\in\mathcal P_q(\mathbb R^d)$, define the Wasserstein distance
$
\wass_q(\mu,\nu):=
\left(\inf_{\varrho\in\Gamma(\mu,\nu)}
\int\|x-y\|_2^q\,\varrho(\rmd x,\rmd y)\right)^{1/q},
$
where $\Gamma(\mu,\nu)$ denotes the set of couplings of $\mu$ and $\nu$; define
$\mathrm{KL}(\mu\|\nu):=\int\log(\rmd\mu/\rmd\nu)\,\rmd\mu$
if $\mu\ll\nu$, and $\infty$ otherwise.
For $x\in\mathbb R^d$ and $r>0$, 
let $\mathcal B(x,r)$ denote the closed Euclidean ball centered at \(x\), and let \(\delta_x\) denote the Dirac measure at \(x\).

\section{Setting and Problem Formulation}
\label{sec:pre}
We introduce the framework for studying score estimation under distribution
shift, where $p_0$ is the reference distribution used to learn the score, and \(q_0\) is a possible shifted distribution at deployment.

\subsection{Diffusion Dynamics and Score Approximation}
We focus on the Ornstein-Uhlenbeck (OU) forward process,
\begin{equation}
\label{eq:intro-forward}
\rmd X_t = -\frac{1}{2}X_t \,\rmd t + \rmd B_t,
\qquad X_0\sim p_0,
\end{equation}
where $B_t$ is a standard $d$-dimensional Brownian motion. 
Its transition
kernel is Gaussian:
\[
X_t\mid X_0
=x_0\sim\mathcal N(a_tx_0,\sigma_t^2I_d),\qquad a_t=e^{-t/2},\qquad\sigma_t^2=1-e^{-t}.
\]
Thus, if $P_t$ denotes the OU semigroup, then $p_t:=P_t\, p_0$ is the law of
\[
a_tX_0+\sigma_t Z,
\qquad X_0\sim p_0, \qquad Z\sim\mathcal N(0,I_d),
\]
where $X_0$ and $Z$ are independent. More generally, for any initial law $\mu$, $P_t\mu$ denotes its OU-smoothed marginal at time $t$.
In SDMs, the forward process is reversed from a
large terminal time $T$, yielding
\begin{equation}
\label{eq:backward}
\rmd X_t^{\leftarrow}=\frac{1}{2}\left(X_t^{\leftarrow}+2\nabla\log p_{T-t}(X_t^{\leftarrow})\right)\rmd t+\rmd W_t,\qquad X_0^{\leftarrow}\sim p_T,
\end{equation}
where $W_t$ is another standard Brownian motion. The term $\nabla\log p_t$ is the score function of $p_t$.
Under mild conditions, the reverse process initialized at \(p_T\) coincides in distribution with the time-reversed forward process, and hence generates samples from \(p_0\)~\citep{anderson1982reverse,cattiaux2023time}.
Since the score $\nabla\log p_t$ is unknown, it must be
estimated from data.

We study score estimation over  $t\in[t_0,T]$ with $t_0>0$.
For later use, define
\[
\rho_t:=\sigma_t^2=1-e^{-t},
\qquad
\tau_0:=\frac{\rho_{t_0}}{a_{t_0}^2}=e^{t_0}-1,
\qquad
\tau_T:=\frac{\rho_T}{a_T^2}=e^T-1.
\]

\subsection{Wasserstein Distribution Shifts}
\label{sec:wass_shift}

We model deployment shifts through Wasserstein neighborhoods of the reference
distribution. Let $p_0\in\mathcal P_2(\mathbb R^d)$ be the reference
distribution. For $\varepsilon\geq0$, define the perturbation class
\begin{equation}
\mathcal Q_\varepsilon(p_0)
:=
\left\{
q_0\in\mathcal P_2(\mathbb R^d):
\wass_2(q_0,p_0)\leq\varepsilon
\right\}.
\label{eq:wasserstein-perturbation-class}
\end{equation}
Here, $\varepsilon$ quantifies the magnitude of the distribution shift.
Neither $p_0$ nor $q_0$ is required to admit a density. For every $t>0$,
Gaussian OU smoothing maps any initial law in $\mathcal P_2(\mathbb R^d)$ to
a strictly positive smooth density. Hence, for
$q_0\in\mathcal Q_\varepsilon(p_0)$, the marginals
$p_t=P_tp_0$ and $q_t=P_tq_0$ admit well-defined scores
$\nabla\log p_t$ and $\nabla\log q_t$.

Under distribution shift, an estimator learned from $p_0$ is evaluated on the
shifted marginal $q_t$ and compared with the shifted score
$\nabla\log q_t$, rather than the reference score $\nabla\log p_t$.

\subsection{Shifted Score Risk}

For a shifted distribution
$q_0\in\mathcal Q_\varepsilon(p_0)$, let $q_t=P_tq_0$. For a measurable score
function
$
s:\mathbb R^d\times[t_0,T]\rightarrow\mathbb R^d,
$
define the shifted score loss
\begin{equation}
\mathcal L_{q_0}(s)
:=
\frac{1}{T-t_0}
\int_{t_0}^{T}
\mathbb E_{X\sim q_t}
\left[
\|s(X,t)-\nabla\log q_t(X)\|_2^2
\right]
\,\rmd t .
\label{eq:shifted-score-loss}
\end{equation}
For a measurable, possibly random score estimator $\hat s$, define the robust
shifted score risk
\begin{equation}
\mathcal R_{\varepsilon,p_0}(\hat s)
:=
\sup_{q_0\in\mathcal Q_\varepsilon(p_0)}
\mathbb E
\left[
\mathcal L_{q_0}(\hat s)
\right],
\label{eq:robust-shifted-score-risk}
\end{equation}
where the expectation is taken over the reference training samples and any
internal randomness used to construct $\hat s$.
This risk measures worst-case score estimation error over all deployment
distributions within Wasserstein distance $\varepsilon$ of $p_0$. When
$\hat s$ is learned from finitely many reference samples, it incorporates both
the statistical difficulty of estimating the reference distribution and the
additional uncertainty induced by distribution shift. Our goal is to
characterize these two effects and their dependence on $\varepsilon$, $t_0$,
and the sample size.

\section{
The Intrinsic Cost of Distribution Shift}
\label{sec:structural-robustness}

We begin by isolating the effect of distribution shift from the statistical
error of learning the reference distribution. The key observation is that, if
an estimated initial law is propagated through the same OU semigroup as the
data distribution, then the entire estimated score trajectory is generated by
a single underlying density path. This allows the shifted score risk to be
controlled through a divergence between the shifted and estimated marginals at
one positive diffusion time.

\subsection{A Robust Comparison Principle}
\label{subsec:general-ou-oracle}

Let $\hat\mu_0$ be a possibly random estimate of $p_0$ with finite second
moment, and define
\begin{equation}
  \hat p_t:=P_t\hat\mu_0,
  \qquad
  \hat s(x,t):=\nabla\log\hat p_t(x).
  \label{eq:generic-ou-plug-in}
\end{equation}
The same estimated initial law therefore determines the entire score trajectory,
rather than estimating the score independently at each diffusion time. For strictly positive densities
$p$ and $r$, define the order-two R\'enyi divergence
\begin{equation}
  D_2(p\|r)
  :=
  \log\,\int_{\mathbb R^d}\frac{p(x)^2}{r(x)}\,\rmd x.
  \label{eq:renyi-two-definition}
\end{equation}

The following proposition gives the basic comparison principle underlying our
robustness analysis.
\begin{proposition}
\label{prop:ou-score-oracle}
Suppose $p_0\in\mathcal P_2(\mathbb R^d)$ and
$D_2(p_{t_0}\|\hat p_{t_0})<\infty$ almost surely. Then, for every
$q_0\in\mathcal P_2(\mathbb R^d)$, 
\begin{equation}
\mathrm{KL}\,\left(q_{t_0}\middle\|\hat p_{t_0}\right)
  \le
  \frac{a_{t_0}^2}{\rho_{t_0}}\,
  \wass_2^2(q_0,p_0)
  +
  D_2\,\left(p_{t_0}\middle\|\hat p_{t_0}\right)
  \label{eq:generic-shifted-kl}
\end{equation}
almost surely.
Consequently, for every $\varepsilon\ge0$,
\begin{equation}
  \mathcal R_{\varepsilon,p_0}(\hat s)
  \le
  \frac{2e^{-t_0}}{(T-t_0)(1-e^{-t_0})}\varepsilon^2
  +
  \frac{2}{T-t_0}
  \mathbb E \left[D_2\,\left(p_{t_0}\middle\|\hat p_{t_0}\right)\right].
  \label{eq:generic-ou-score-oracle}
\end{equation}
\end{proposition}
Proposition~\ref{prop:ou-score-oracle} separates two sources of error in robust score estimation. The first term depends only on the discrepancy between the shifted and reference initial distributions and is quadratic in $\wass_2(q_0,p_0)$, while the second measures the error in estimating the reference marginal $p_{t_0}$. The bound therefore controls the effect of distribution shift through an additive term of order $\varepsilon^2$. This reduces the finite-sample problem to controlling $D_2(p_{t_0}\|\hat p_{t_0})$, which we address in Section~\ref{sec:finite-sample}. Before turning to that construction, we show that the quadratic dependence on $\varepsilon$ is unavoidable.

\subsection{Minimax Optimality of the Quadratic Shift Cost}

Proposition~\ref{prop:ou-score-oracle} gives a shift contribution of order
$\varepsilon^2/[(T-t_0)(e^{t_0}-1)]$. We now show that this dependence is
unavoidable: when $\tau_0\asymp t_0$, every estimator incurs error of order
$\varepsilon^2/[t_0(T-t_0)]$ for some distribution in the Wasserstein
neighborhood.

The upper bound in Proposition~\ref{prop:ou-score-oracle} holds for arbitrary
initial laws in $\mathcal P_2(\mathbb R^d)$. For the lower bound, it suffices
to consider a compactly supported reference distribution. 
Let
\[
p_0^{(1)}(u)
:=
\frac12\mathbf 1_{[-1,1]}(u),
\qquad
p_0^{(d)}(x)
:=
\prod_{j=1}^d p_0^{(1)}(x_j).
\]
The following theorem gives a matching lower bound for the distribution shift
term.
\begin{theorem}
\label{thm:fixed-dimensional-lower-main}
Fix $d\ge1$. Let $P_t^{(d)}$ be the semigroup of the
$d$-dimensional Ornstein--Uhlenbeck process.
Suppose that
$0<t_0\le 1/(16\pi^2),
2t_0\le T,$ and $
0<\varepsilon\le\sqrt{t_0}.
$
Then,
\begin{equation}
\label{eq:fixed-dimensional-lower-main}
\inf_{\hat s}\mathcal R_{\varepsilon,p_0^{(d)}}(\hat s)
\ge
\frac{e^{-4/9}}{2592}
\frac{\varepsilon^2}{t_0(T-t_0)}\,,
\end{equation}
where the infimum is over all possibly randomized score estimators with
measurable realizations.
\end{theorem}

This theorem shows that the quadratic
dependence on the shift radius is intrinsic: even with unrestricted
choice of estimator, uncertainty over
$\mathcal Q_\varepsilon(p_0^{(d)})$ incurs a cost of order
$\varepsilon^2/[t_0(T-t_0)]$. Since
$\tau_0\asymp t_0$ in the regime above, this matches the shift term
in Proposition~\ref{prop:ou-score-oracle} up to universal constants.
Thus, both the quadratic dependence on $\varepsilon$ and its early-time
scaling are optimal.

\section{Finite-Sample Robust Score Estimation}
\label{sec:finite-sample}
Section~\ref{sec:structural-robustness} establishes the intrinsic quadratic cost of distribution shift. We now turn to the finite-sample setting, where \(p_0\) is observed through \(n\) i.i.d. samples. Proposition~\ref{prop:ou-score-oracle} reduces robust score estimation to controlling the positive-time estimation error \(D_2(p_{t_0}\|\hat p_{t_0})\). We construct an explicit estimator for which
this term can be controlled sharply and establish the resulting minimax rate.

\subsection{The Gaussian-Blanket Estimator}
\label{subsec:finite-sample-score}

For the finite-sample analysis, we impose the following compact support
assumption.
\begin{assumption}
\label{assum:support}
For some $D>0$,
$
    \operatorname{supp}(p_0)\subseteq \mathcal{B}(0,D).
$
\end{assumption}
Assumption~\ref{assum:support} is needed only for the finite-sample
construction below. The oracle inequality in
Proposition~\ref{prop:ou-score-oracle}, and hence the quadratic dependence on
the distribution shift established in Section~\ref{sec:structural-robustness},
continues to hold over the larger class $\mathcal P_2(\mathbb R^d)$.

A direct empirical approximation of \(p_0\) becomes strictly positive after OU smoothing, but positivity alone is insufficient for controlling the Rényi divergence required in our analysis, since the estimated density appears in the denominator and can be arbitrarily small in poorly sampled regions. We therefore introduce Gaussian blanket regularization, which provides a deterministic Gaussian lower envelope while preserving OU compatibility.

For $v>0$, let
\(\varphi_v(x):=(2\pi v)^{-d/2}\exp(-\|x\|_2^2/(2v))\) denote the density of $\mathcal{N}(0,vI_d)$.
Given
$X_1,\ldots,X_n\overset{\mathrm{i.i.d.}}{\sim}p_0$, fix $\vartheta>0$ and define $\nu_0^\vartheta:=\mathcal{N}(0,\vartheta^2I_d)$. Let $v_t:=\rho_t+a_t^2\vartheta^2$, and define the calibration factor
\begin{equation}
    M_{\vartheta,t_0}
    :=
    \exp\,\left(\frac{D^2}{2\vartheta^2}\right)
    \left(1+\frac{\vartheta^2}{\tau_0}\right)^{d/2}.
    \label{eq:blanket-calibration}
\end{equation}
The Gaussian-blanket estimator regularizes the empirical measure as
\begin{equation}
    \hat\mu_{n,0}^\vartheta
    :=
    \frac{1}{n+M_{\vartheta,t_0}}
    \sum_{i=1}^n\delta_{X_i}
    +
    \frac{M_{\vartheta,t_0}}{n+M_{\vartheta,t_0}}
    \nu_0^\vartheta\,.
    \label{eq:blanket-initial-measure}
\end{equation}
Propagating this measure through the same OU semigroup gives
\begin{equation}
    \hat p_{n,t}^\vartheta(x)
    :=
    P_t\hat\mu_{n,0}^\vartheta(x)
    =
    \frac{
        \sum_{i=1}^n\varphi_{\rho_t}(x-a_tX_i)
        +M_{\vartheta,t_0}\varphi_{v_t}(x)
    }{n+M_{\vartheta,t_0}}.
    \label{eq:blanket-density}
\end{equation}
The corresponding score estimator is
\begin{equation}
    \hat s_n^\vartheta(x,t)
    :=
    \nabla\log\hat p_{n,t}^\vartheta(x).
    \label{eq:blanket-score}
\end{equation}
The choice of $M_{\vartheta,t_0}$ ensures that the Gaussian component provides
uniform denominator control over all sample realizations allowed by
Assumption~\ref{assum:support}. This is the key ingredient for bounding the
R\'enyi-divergence term $
D_2(p_{t_0}\|\hat p_{n,t_0}^\vartheta).
$

\subsection{Finite-Sample Upper Bound}
\label{subsec}

We first quantify the approximation error of the Gaussian-blanket density.
\begin{proposition}
\label{prop:reverse-chi-square}
Under Assumption~\ref{assum:support}, for every $t\in[t_0,T]$,
\begin{equation}
    \mathbb E
    \left[
        D_2\,\left(
            p_t\middle\|\hat p_{n,t}^\vartheta
        \right)
    \right]
    \le
    \log\,\left(
        1+\frac{M_{\vartheta,t_0}-1}{n+1}
    \right).
    \label{eq:renyi-rate}
\end{equation}
\end{proposition}
Proposition~\ref{prop:reverse-chi-square} provides the finite-sample control
of the estimation term in Proposition~\ref{prop:ou-score-oracle}. Combining the
two propositions yields the following robust score estimation guarantee.
\begin{theorem}
\label{thm:finite-sample-robust-kde}
Fix $d\ge1$ and $D>0$, and suppose Assumption~\ref{assum:support} holds. Let
$X_1,\ldots,X_n\overset{\mathrm{i.i.d.}}{\sim}p_0$.
Set $\vartheta=D/\sqrt d$ and construct $\hat s_n^\vartheta$ according to
\eqref{eq:blanket-density}-\eqref{eq:blanket-score}.
If $\tau_0\le1$, then, for every $\varepsilon\ge0$,
\begin{equation}
  \mathcal R_{\varepsilon,p_0}(\hat s_n^\vartheta)
  \le
  \frac{C_{d,D}}
       {(n+1)(T-t_0)\tau_0^{d/2}}
  +
  \frac{2\varepsilon^2}
       {(T-t_0)\tau_0},
  \label{eq:rate-finite-sample-robust-bound}
\end{equation}
where $C_{d,D}=2e^{d/2}(1+D^2/d)^{d/2}$.
\end{theorem}

For fixed $d$ and $D$, since
$\tau_0\asymp t_0$ for bounded $t_0$, the two terms in
Theorem~\ref{thm:finite-sample-robust-kde} scale as
$
{1}/{[n(T-t_0)t_0^{d/2}]}
$ and $
{\varepsilon^2}/{[t_0(T-t_0)]},
$
respectively. The first term corresponds to the finite-sample estimation
error, while the second term captures the intrinsic cost of distribution
shift. Notably, the estimator does not depend on the shift radius
$\varepsilon$, and the coefficient of the shift term is independent of the
dimension, sample size, and support radius.

\paragraph{Neural realizability.}
The Gaussian-blanket score also admits a constructive ReLU-based approximation.
Under Assumption~\ref{assum:support},
Appendix~\ref{app:blanket-neural-approximation} constructs a
sample-dependent approximation that is uniform over
$\mathbb R^d\times[t_0,T]$, with explicit width and depth bounds (see
Proposition~\ref{gbnn:prop-approximation}). Moreover,
Subsection~\ref{gbnn:subsec-defect-proof} shows that the resulting neural
score inherits the robust finite-sample guarantee of
Theorem~\ref{thm:finite-sample-robust-kde}, up to the prescribed neural
approximation accuracy. In particular, choosing the approximation error
sufficiently small preserves the robust minimax rate.

\subsection{Minimax Optimality}
\label{subsec:finite-sample-lb}

We now show that the upper bound in
Theorem~\ref{thm:finite-sample-robust-kde} is minimax optimal. 
Define 
\[
\mathcal P_D
:=
\left\{
p_0\in\mathcal P(\mathbb R^d):
\operatorname{supp}(p_0)\subseteq \mathcal B(0,D)
\right\},
\]
and the minimax robust risk
\begin{equation}
\mathfrak R_{n,\varepsilon}^{\star}(D;t_0,T)
:=
\inf_{\hat s_n}
\sup_{p_0\in\mathcal P_D}
\mathcal R_{\varepsilon,p_0}(\hat s_n)\,,
\label{eq:finite-sample-minimax-risk}
\end{equation}
where the infimum is over all measurable, possibly randomized score estimators
based on \(n\) i.i.d. samples from \(p_0\).
When $\varepsilon=0$, this reduces to the nominal time-averaged score estimation risk.

Theorem~\ref{thm:fixed-dimensional-lower-main} already identifies the
unavoidable contribution of distribution shift. It remains to characterize
the statistical difficulty of estimating the reference score trajectory from
finite data. Because the risk is integrated over diffusion time, a fixed-time
lower bound does not immediately suffice: the hardest instance may depend on
the diffusion time. The next theorem constructs a single family whose score
separation persists over a nontrivial time interval.
\begin{theorem}
\label{thm:integrated-statistical-lower-bound}
Fix $d\ge1$ and $D>0$. Let $c_0,c_1>0$ be the constants given in
Appendix~\ref{app:statistical-minimax-lower-bound}, which depend only on $d$
and $D$. If
$
0<t_0\le c_0,
$ and $
2t_0\le T,
$
then, for every $n\ge1$ and $\varepsilon\ge0$,
\begin{equation}
\mathfrak R_{n,\varepsilon}^{\star}(D;t_0,T)
\ge
\frac{c_1}{T-t_0}
\left(
1\wedge\frac{1}{n\tau_0^{d/2}}
\right).
\label{eq:integrated-statistical-lower-bound}
\end{equation}
\end{theorem}
Combining this result with the
shift lower bound in Theorem~\ref{thm:fixed-dimensional-lower-main} gives the
full robust minimax rate.
\begin{theorem}
\label{thm:finite-sample-minimax-optimality}
Fix \(d\ge1\) and \(D\ge\sqrt d\), and let \(c_0,c_1\) be the constants in
Theorem~\ref{thm:integrated-statistical-lower-bound}.  If
\(0<t_0\le\min\{c_0,1/(16\pi^2)\}\), \(2t_0\le T\),
\(n\tau_0^{d/2}\ge1\), and \(0\le\varepsilon\le\sqrt{t_0}\), then
\begin{equation}
  c_{d,D}
  \left[
    \frac{1}{n(T-t_0)\tau_0^{d/2}}
    +
    \frac{\varepsilon^2}{t_0(T-t_0)}
  \right]
  \le
  \mathfrak R_{n,\varepsilon}^{\star}(D;t_0,T)
  \le
  C_{d,D}
  \left[
    \frac{1}{n(T-t_0)\tau_0^{d/2}}
    +
    \frac{\varepsilon^2}{t_0(T-t_0)}
  \right].
\label{eq:combined-finite-sample-minimax-rate}
\end{equation}
Here, \(C_{d,D}=2e^{d/2}(1+D^2/d)^{d/2}\), and 
\(c_{d,D}:=\frac12\min\{c_1,e^{-4/9}/2592\}\). 
\end{theorem}
This theorem shows that the upper bound in
Theorem~\ref{thm:finite-sample-robust-kde} is sharp, up to constants, in both
the finite-sample and distribution-shift terms. Thus, the Gaussian-blanket
estimator attains the robust minimax rate in the stated regime.

\paragraph{Relation to fixed-time score estimation.}
The statistical term agrees with existing fixed-time score-estimation rates. For fixed $d$ and $D$, prior work gives, up to logarithmic factors, the fixed-time rate
$
n^{-1}\tau^{-d/2-1}
$~\citep{zhang2024minimax,dou2024optimal,cai2025minimax}.
Reparameterizing OU time by $\tau=e^t-1$, together with the corresponding score rescaling, integration over $[t_0,T]$ yields
$
{1}/{[n(T-t_0)\tau_0^{d/2}]},
$
which matches the statistical term in
Theorem~\ref{thm:finite-sample-robust-kde}.

\paragraph{Adaptation to the shift radius.}
The Gaussian-blanket estimator does not depend on $\varepsilon$. Under the
conditions of Theorem~\ref{thm:finite-sample-minimax-optimality}, the same
estimator is minimax optimal simultaneously for all
$0\le\varepsilon\le\sqrt{t_0}$, without radius-dependent tuning.

\section{Adaptation to Intrinsic Dimension}
\label{sec:unknown-subspace}

The finite-sample rate in Section~\ref{sec:finite-sample} depends on the ambient dimension \(d\). We show that this dependence can be replaced by the intrinsic dimension when \(p_0\) is supported on an unknown low-dimensional linear subspace. The shifted distribution is allowed to leave the source subspace, and the resulting estimator adapts to the intrinsic dimension without requiring exact subspace recovery.

We impose the following additional structure.

\begin{assumption}
\label{assum:unknown-subspace}
Let $k\in\{1,\ldots,d\}$ be known. There exists an unknown matrix
$U\in\mathbb R^{d\times k}$ such that
\begin{equation}
  U^\top U=I_k,
  \qquad
  \operatorname{supp}(p_0)\subseteq\operatorname{range}(U).
  \label{eq:unknown-subspace-model}
\end{equation}
\end{assumption}

Let $X\sim p_0$ and write $X=UZ$ with $Z\in\mathbb R^k$. We do not impose
smoothness or non-degeneracy assumptions on the latent distribution. The
assumption applies only to $p_0$. The perturbation class
$\mathcal Q_\varepsilon(p_0)$ remains the full Wasserstein neighborhood in
$\mathbb R^d$.

To estimate the unknown source subspace, we split the $n$ i.i.d. samples from $p_0$ into
two independent subsets of sizes
$n_1=\lceil n/2\rceil$ and $n_2=n-n_1$, with $n\ge2$.
Using the first subset, we construct the empirical subspace and its orthogonal
projector
\[
\hat L:=\operatorname{span}(X_1,\ldots,X_{n_1}),
\qquad
\hat\Pi:=\operatorname{Proj}_{\hat L},
\]
and project the remaining samples as
\[
Y_i:=\hat\Pi X_{n_1+i},\qquad 1\le i\le n_2 .
\]
We then apply the Gaussian-blanket construction in the estimated subspace
with intrinsic scale $\vartheta_k^2=D^2/k$ {and the calibration
$M_{\vartheta_k,t_0}:=e^{k/2}(1+D^2/(k\tau_0))^{k/2}$.}
Specifically, define
\[
\hat\mu_{n,0}^{\mathrm{sub}}
:=
\frac{\sum_{i=1}^{n_2}\delta_{Y_i}
+M_{\vartheta_k,t_0}\mathcal N(0,\vartheta_k^2\hat\Pi)}
{n_2+M_{\vartheta_k,t_0}},
  \qquad 
\hat s_n^{\mathrm{sub}}(x,t)
:=
\nabla\log P_t\hat\mu_{n,0}^{\mathrm{sub}}(x).
\]
The Gaussian component provides the required regularization, while the
projection onto the estimated subspace enables adaptation from the ambient
dimension $d$ to the intrinsic dimension $k$.

The estimator does not require exact recovery of the source subspace. Instead,
it suffices to control the source mass missed by the empirical span, leading to
the following guarantee.
\begin{theorem}
\label{thm:unknown-subspace-upper}
Under Assumptions~\ref{assum:support} and~\ref{assum:unknown-subspace},
for every
$n\ge2$, $0<t_0<T<\infty$, and $\varepsilon\ge0$,
\begin{equation}
  \mathcal R_{\varepsilon,p_0}(\hat s_n^{\mathrm{sub}})
  \le
  \frac{2}{T-t_0}
  \log\left(1+\frac{M_{\vartheta_k,t_0}-1}{n_2+1}\right)
  +\frac{2\varepsilon^2}{(T-t_0)\tau_0}
  +\frac{2kD^2}{(n_1+k)(T-t_0)\tau_0}.
  \label{eq:unknown-subspace-finite}
\end{equation}
In particular, if $k\ge2$ and $\tau_0\le1$, then
\begin{equation}
  \mathcal R_{\varepsilon,p_0}(\hat s_n^{\mathrm{sub}})
  \le
  \frac{\tilde C_{k,D}}
       {n(T-t_0)\tau_0^{k/2}}
  +\frac{2\varepsilon^2}{(T-t_0)\tau_0}.
  \label{eq:unknown-subspace-rate}
\end{equation}
Here
\(\tilde C_{k,D}
:=4[e^{k/2}(1+D^2/k)^{k/2}+kD^2]\)
is independent of the ambient dimension $d$.
\end{theorem}

The bound decomposes into three terms corresponding to estimation within the
empirical subspace, distribution shift, and the missed source component. The
last term is controlled by
\[
\mathbb E\int\|(I_d-\hat\Pi)x\|_2^2\,p_0(\rmd x)
\le \frac{kD^2}{n_1+k}.
\]
For $k\ge2$ and $\tau_0\le1$, it is absorbed into the statistical term,
yielding the simplified rate.
Furthermore, the resulting rate replaces the ambient-dimensional dependence
$\tau_0^{-d/2}$ in Theorem~\ref{thm:finite-sample-robust-kde} by
$\tau_0^{-k/2}$, while the distribution shift term remains unchanged. Thus,
intrinsic structure reduces the statistical cost of learning the reference
distribution but does not mitigate the uncertainty induced by distribution
shift.

\subsection{Minimax Optimality}
\label{subsec:unknown-subspace-minimax}
The intrinsic-dimensional rate is also minimax optimal. 
Let $\mathcal P_{k,D}$ denote the subclass of $\mathcal P_D$ satisfying
Assumption~\ref{assum:unknown-subspace}.
Define the minimax robust score risk by
\begin{equation}
  \mathfrak R_{n,\varepsilon}^{\star}(k,D;t_0,T)
  :=\inf_{\hat s_n}\sup_{p_0\in\mathcal P_{k,D}}
  \mathcal R_{\varepsilon,p_0}(\hat s_n).
  \label{eq:unknown-subspace-minimax-risk}
\end{equation}
By Corollary~\ref{cor:unknown-subspace-minimax}\footnote{We defer the formal statement and proof to Appendix~\ref{app:unknown-subspace-minimax}.}, when $k\geq 2,$
\[
\mathfrak R_{n,\varepsilon}^{\star}(k,D;t_0,T)
\asymp
\frac{1}{n(T-t_0)\tau_0^{k/2}}
+
\frac{\varepsilon^2}{t_0(T-t_0)},
\]
with constants depending only on $k$ and $D$, independently of the ambient
dimension $d$.

\section{From Robust Estimation to Robust Generation}

We now translate the score estimation guarantees developed above into
distributional guarantees for reverse-time sampling. Since the shifted data
law $q_0$ may be singular, we evaluate the sampler at the positive forward
time $t_0>0$ and measure its error relative to $q_{t_0}$ in KL divergence.

\subsection{Minimax Optimal Reverse Sampling}
\label{subsec:positive-time-kl-sampling}

{
{Under Assumption~\ref{assum:support}, take $\hat s_n=\hat s_n^\vartheta$
and construct the approximate reverse process by replacing the true score
in \eqref{eq:backward} with $\hat s_n$.}
When initialized with the self-consistent terminal distribution
$\nu_T=\hat p_{n,T}^{\vartheta}$, the reverse process exactly recovers the
estimated OU trajectory, so its output distribution at time $t_0$ coincides
with $\hat p_{n,t_0}^{\vartheta}$. 
Therefore, setting $\vartheta=D/\sqrt d$ and assuming $\tau_0\le1$,
combining Propositions~\ref{prop:ou-score-oracle} and~\ref{prop:reverse-chi-square} yields
\begin{equation}
\sup_{q_0\in\mathcal Q_\varepsilon(p_0)}
\mathbb E\left[
\mathrm{KL}\bigl(q_{t_0}\|\hat\mu_{n,t_0}^{\,\nu}\bigr)
\right]
\le
\frac{C_{d,D}}{2(n+1)\tau_0^{d/2}}
+\frac{\varepsilon^2}{\tau_0},
\label{eq:self-consistent-kl-sampling-bound}
\end{equation}
where $\hat\mu_{n,t_0}^{\,\nu}$ denotes the sampler output distribution.
Hence, the reverse sampler inherits the same decomposition as score
estimation: a finite-sample statistical term and an intrinsic quadratic
distribution-shift term. {A corresponding bound under standard Gaussian initialization
is given in Corollary~\ref{cor:standard-gaussian-kl-sampling}.}

To characterize whether these two terms are unavoidable, define the minimax
positive-time KL risk
\[
\mathfrak K_{n,\varepsilon}^{\star}(D;t_0)
:=
\inf_{\hat\mu_n}
\sup_{p_0\in\mathcal P_D}
\sup_{q_0\in\mathcal Q_\varepsilon(p_0)}
\mathbb E \left[
\mathrm{KL}(q_{t_0}\|\hat\mu_n)
\right],
\]
where the infimum is over all measurable, possibly randomized estimators based on
$n$ i.i.d.\ samples from $p_0$.

The following theorem shows that both components are minimax optimal.
}

\begin{theorem}
\label{thm:robust-kl-minimax}
Fix $d\ge1$ and $D>0$. Let $C_{d,D}$ be as in
Theorem~\ref{thm:finite-sample-robust-kde}. There exist constants
$c_0,\tilde c_{d,D}>0$, depending only on $d$ and $D$, such that if
$
0<t_0\le c_0,
n\tau_0^{d/2}\ge1,
0\le\varepsilon\le\sqrt{\tau_0},
$
then
\begin{equation}
  \tilde c_{d,D}
  \left[
    \frac{1}{n\tau_0^{d/2}}
    +\frac{\varepsilon^2}{\tau_0}
  \right]
  \le
  \mathfrak K_{n,\varepsilon}^{\star}(D;t_0)
  \le
  C_{d,D}
  \left[
    \frac{1}{n\tau_0^{d/2}}
    +\frac{\varepsilon^2}{\tau_0}
  \right].
  \label{eq:robust-kl-minimax-rate}
\end{equation}
The lower-bound constants are specified in
Appendices~\ref{app:statistical-minimax-lower-bound}
and~\ref{app:intrinsic-lower-bound}.
\end{theorem}
Hence, the statistical and distribution shift costs identified at the score
level persist in the generated distribution and are both minimax optimal.

\subsection{Intrinsic-Dimensional Reverse Sampling}

The intrinsic-dimensional estimator in Section~\ref{sec:unknown-subspace}
immediately yields corresponding generation guarantees. With
$\hat s_n=\hat s_n^{\mathrm{sub}}$ and self-consistent initialization,
the output distribution satisfies the same KL decomposition, with the ambient
dimension replaced by the intrinsic dimension.\footnote{We defer the formal statement and proof to Appendix~\ref{app:unknown-subspace-kl}.}

For the class $\mathcal P_{k,D}$, define
\[
\mathfrak K_{n,\varepsilon}^{\star}(k,D;t_0)
:=
\inf_{\hat\mu_n}
\sup_{p_0\in\mathcal P_{k,D}}
\sup_{q_0\in\mathcal Q_\varepsilon(p_0)}
\mathbb E[
\mathrm{KL}(q_{t_0}\|\hat\mu_n)].
\]
Then, by Corollary~\ref{cor:unknown-subspace-kl-minimax}\footnote{We defer the formal statement and proof to Appendix~\ref{app:unknown-subspace-kl-minimax}.}, when $k\geq 2,$
\[
\mathfrak K_{n,\varepsilon}^{\star}(k,D;t_0)
\asymp
\frac{1}{n\tau_0^{k/2}}
+\frac{\varepsilon^2}{\tau_0},
\]
with constants depending only on $k$ and $D$, independently of the ambient
dimension $d$. Thus, intrinsic structure reduces the statistical cost to the
intrinsic dimension $k$, while the distribution-shift cost remains unchanged.

{
\color{red}

}

\section{Discussion}
\label{sec:conc}

This work develops a theory of robust score estimation and generation under
distribution shift. We show that the robust error decomposes into a statistical
estimation component and an intrinsic distribution-shift component. For the
OU diffusion, the shift cost is quadratic in the Wasserstein
radius and minimax optimal. We further construct an optimal finite-sample
estimator, extend the analysis to low-dimensional source structures, and
establish corresponding minimax guarantees for positive-time reverse sampling.

Several directions remain open. First, our analysis focuses on the OU process,
which provides a canonical setting where the statistical and intrinsic costs of
distribution shift can be characterized sharply. Extending this framework to
more general diffusion dynamics and noise schedules is an important direction
for future work. Second, our intrinsic-dimensional results rely on linear
subspace structure; extending them to more general notions of intrinsic
complexity, such as nonlinear manifolds, remains open. Third, while we
establish neural realizability of the proposed estimator, understanding when
practical score-network training achieves the predicted robustness guarantees
remains an important challenge. Finally, extending the theory to discretized
samplers and quantifying the interaction between distribution shift, score
estimation, and numerical discretization would provide a more complete
end-to-end theory of robust diffusion sampling.

\section*{Acknowledgement}

This work was partially supported by the City University of Hong Kong Startup Fund 9610707 and the Hong Kong Research Grants Council through ECS Grant No. 21306325 and GRF Grant No. 11309626. NKC is supported by a City University of Hong Kong Start-up Grant 7200809, and a Hong Kong ECS-RGC Grant 21300826.

\bibliographystyle{plainnat}
\bibliography{bib}

\newpage
\appendix

\section{Proofs for Robust Score Upper Bounds}
\label{app:ou-information-estimates}
\label{app:finite-sample-robust-score}

We first establish two OU information estimates, then prove the robust
oracle bound in Proposition~\ref{prop:ou-score-oracle} and the
Gaussian-blanket upper bounds in Proposition~\ref{prop:reverse-chi-square}
and Theorem~\ref{thm:finite-sample-robust-kde}. The corresponding statistical
and shift lower bounds are proved in
Appendices~\ref{app:statistical-minimax-lower-bound}
and~\ref{app:intrinsic-lower-bound}, respectively.

\subsection{Fixed-Time OU Gaussian-Channel Bound}
\label{app:fixed-time-ou-kl-proof}

The fixed-time bound below is a special case of a standard reverse transport
inequality for Langevin semigroups; see, e.g.,
\citet[Corollary~3.11]{altschuler2025shifted}. We include a direct proof
to make our OU normalization explicit.

\begin{lemma}[Fixed-time OU Gaussian-channel bound]
\label{lem:fixed-time-ou-kl}
Let \(p_0\) and \(q_0\) be Borel probability measures on \(\mathbb R^d\)
with finite second moments, and set \(p_t=P_tp_0\) and \(q_t=P_tq_0\).
Recall that \(a_t=e^{-t/2}\) and \(\sigma_t^2=1-e^{-t}\).
Then, for every \(t>0\),
\begin{equation}
\label{eq:fixed-time-ou-kl}
    \mathrm{KL}(q_t\|p_t)
    \le
    \frac{a_t^2}{2\sigma_t^2}
    \wass_2^2(p_0,q_0)
    =
    \frac{e^{-t}}{2(1-e^{-t})}
    \wass_2^2(p_0,q_0).
\end{equation}
\end{lemma}

\begin{proof}
Fix \(t>0\) and let \(\pi\in\Gamma(p_0,q_0)\) be an arbitrary coupling.
The \(\mathcal P_2\) assumption ensures that
\[
    \int\|y_0-x_0\|_2^2\,\pi(\rmd x_0,\rmd y_0)
    \le
    2\int\|x_0\|_2^2\,p_0(\rmd x_0)
    +
    2\int\|y_0\|_2^2\,q_0(\rmd y_0)
    <\infty.
\]
The common OU transition kernel is
\(p_{t\mid0}(\cdot\mid x_0):=\mathcal N(a_tx_0,\sigma_t^2I_d)\).
Since the first and second marginals of \(\pi\) are \(p_0\) and \(q_0\),
respectively, the two OU marginals admit the mixture representations
\[
    p_t
    =
    \int_{\mathbb R^d\times\mathbb R^d}
    p_{t\mid 0}(\cdot\mid x_0)\,
    \pi(\rmd x_0,\rmd y_0),
    \qquad
    q_t
    =
    \int_{\mathbb R^d\times\mathbb R^d}
    p_{t\mid 0}(\cdot\mid y_0)\,
    \pi(\rmd x_0,\rmd y_0).
\]

We use the following mixture form of the joint convexity of relative
entropy.  Let \(\lambda\) be a probability measure and let
\((\mu_z,\nu_z)\) be a measurable pair of probability kernels indexed by
\(z\).  Then
\[
    \mathrm{KL}\left(
        \int \mu_z\,\lambda(\rmd z)
        \,\middle\|\,
        \int \nu_z\,\lambda(\rmd z)
    \right)
    \le
    \int \mathrm{KL}(\mu_z\|\nu_z)\,\lambda(\rmd z).
\]
For completeness, define joint probability measures by
\[
    \overline\mu(\rmd z,\rmd x)
    :=
    \lambda(\rmd z)\mu_z(\rmd x),
    \qquad
    \overline\nu(\rmd z,\rmd x)
    :=
    \lambda(\rmd z)\nu_z(\rmd x).
\]
They have the same \(z\)-marginal \(\lambda\), and hence the chain rule for
relative entropy gives
\[
    \mathrm{KL}(\overline\mu\|\overline\nu)
    =
    \int \mathrm{KL}(\mu_z\|\nu_z)\,\lambda(\rmd z).
\]
Their \(x\)-marginals are the two mixtures in the displayed inequality.
Applying the data-processing inequality to the projection
\((z,x)\mapsto x\) therefore proves the claim.

Apply this inequality with \(z=(x_0,y_0)\), \(\lambda=\pi\),
\(\mu_z=p_{t\mid0}(\cdot\mid y_0)\), and
\(\nu_z=p_{t\mid0}(\cdot\mid x_0)\).
The preceding mixture representations yield
\[
    \mathrm{KL}(q_t\|p_t)
    \le
    \int_{\mathbb R^d\times\mathbb R^d}
    \mathrm{KL}\left(
        p_{t\mid 0}(\cdot\mid y_0)
        \middle\|
        p_{t\mid 0}(\cdot\mid x_0)
    \right)
    \pi(\rmd x_0,\rmd y_0).
\]
Since \(t>0\), we have \(\sigma_t^2>0\).  The two transition laws inside
the relative entropy are therefore nondegenerate Gaussians with the same
covariance matrix \(\sigma_t^2I_d\), so
\[
    \mathrm{KL}\left(
        p_{t\mid 0}(\cdot\mid y_0)
        \middle\|
        p_{t\mid 0}(\cdot\mid x_0)
    \right)
    =
    \mathrm{KL}\left(
        \mathcal N(a_ty_0,\sigma_t^2I_d)
        \middle\|
        \mathcal N(a_tx_0,\sigma_t^2I_d)
    \right)
    =
    \frac{a_t^2}{2\sigma_t^2}\|y_0-x_0\|_2^2.
\]
Consequently,
\[
    \mathrm{KL}(q_t\|p_t)
    \le
    \frac{a_t^2}{2\sigma_t^2}
    \int_{\mathbb R^d\times\mathbb R^d}
    \|y_0-x_0\|_2^2\,\pi(\rmd x_0,\rmd y_0).
\]
Taking the infimum over \(\pi\in\Gamma(p_0,q_0)\) proves
\eqref{eq:fixed-time-ou-kl}.
\end{proof}

\subsection{Relative de Bruijn Identity}
\label{app:relative-de-bruijn-ou-proof}

The relative de Bruijn identity with a moving reference under common
Gaussian smoothing goes back to \citet{verdu2010mismatched}.  A direct
heat-equation derivation and the corresponding exact integral representation
were given by \citet[Theorem~2.2]{hirata2012integral}.
See also \citet[Lemma~2 and Proposition~3]{yoshida2017dissipation} for the
Gaussian identity and its extension to common Fokker--Planck diffusion
flows under regularity assumptions.  We record the precise positive-time
OU formulation needed here.

\begin{lemma}[OU relative de Bruijn identity on \(\mathcal P_2\)]
\label{lem:relative-de-bruijn-ou}
Let \(p_0,q_0\in\mathcal P_2(\mathbb R^d)\), and set
\(p_t:=P_tp_0\) and \(q_t:=P_tq_0\).
For \(t>0\), regard \(q_t\) and \(p_t\) as their strictly positive smooth
densities, and define
\[
    I(q_t\|p_t)
    :=
    \int_{\mathbb R^d}
    q_t(x)
    \left\|
        \nabla\log q_t(x)-\nabla\log p_t(x)
    \right\|_2^2
    \,\rmd x.
\]
Then \(t\mapsto\mathrm{KL}(q_t\|p_t)\) is locally absolutely continuous
on \((0,\infty)\), and
\begin{equation}
\label{eq:relative-de-bruijn-differential}
    \frac{\rmd}{\rmd t}\mathrm{KL}(q_t\|p_t)
    =
    -\frac12 I(q_t\|p_t)
\end{equation}
for almost every \(t>0\).  Equivalently, for every
\(0<t_-<t_+<\infty\),
\begin{equation}
\label{eq:relative-de-bruijn-integrated}
    \mathrm{KL}(q_{t_+}\|p_{t_+})
    +
    \frac12\int_{t_-}^{t_+}I(q_t\|p_t)\,\rmd t
    =
    \mathrm{KL}(q_{t_-}\|p_{t_-}).
\end{equation}
\end{lemma}

\begin{proof}
Let \((\mathsf H_s)_{s\ge0}\) denote the heat semigroup
\(\mathsf H_s\eta:=\eta*\mathcal N(0,sI_d)\), and set
\(\overline p_s:=\mathsf H_sp_0\) and \(\overline q_s:=\mathsf H_sq_0\).
We first verify the hypotheses of the moving-reference heat-flow identity in
\citet[Theorem~2.2]{hirata2012integral}.
For every \(s>0\), both \(\overline q_s\) and \(\overline p_s\) are
strictly positive smooth densities; in particular,
\(\overline q_s\ll\overline p_s\).  For any
\(\eta\in\mathcal P_2(\mathbb R^d)\), if \(Y\sim\eta\) and
\(Z\sim\mathcal N(0,I_d)\) are independent and
\(U_s:=Y+\sqrt{s}Z\), the Gaussian score identity gives
\[
    \nabla\log(\mathsf H_s\eta)(U_s)
    =
    -s^{-1/2}\mathbb E[Z\mid U_s].
\]
Consequently, conditional Jensen's inequality yields
\begin{equation}
\label{eq:heat-smoothed-fisher-bound}
\begin{aligned}
    J(\mathsf H_s\eta)
    :=
    \int_{\mathbb R^d}
    \mathsf H_s\eta(x)\|\nabla\log(\mathsf H_s\eta)(x)\|_2^2\,\rmd x
    &=
    \frac1s\mathbb E\left[
        \left\|\mathbb E[Z\mid U_s]\right\|_2^2
    \right]
    \\
    &\le
    \frac1s\mathbb E\|Z\|_2^2
    =
    \frac ds.
\end{aligned}
\end{equation}

It remains to check finiteness of the relative entropy.  Given \(s>0\),
put \(u:=\log(1+s)\), let \(S_c(x):=cx\), and write
\((S_c)_\#\lambda\) for the pushforward of \(\lambda\) under \(S_c\).  Since
\[
    P_u\eta=(S_{a_u})_\#\mathsf H_s\eta,
\]
invariance of relative entropy under a common invertible dilation and
Lemma~\ref{lem:fixed-time-ou-kl} imply
\begin{equation}
\label{eq:heat-smoothed-relative-entropy-bound}
\begin{aligned}
    \mathrm{KL}(\overline q_s\|\overline p_s)
    &=
    \mathrm{KL}(P_uq_0\|P_up_0)
    \\
    &\le
    \frac{a_u^2}{2\sigma_u^2}\wass_2^2(p_0,q_0)
    =
    \frac{1}{2s}\wass_2^2(p_0,q_0)
    <\infty.
\end{aligned}
\end{equation}
For each fixed \(s>0\), we now apply
\citet[Theorem~2.2]{hirata2012integral} with
\(f=\overline q_s\) and \(g=\overline p_s\).  As noted above, both densities
are strictly positive, and hence \(\overline q_s\ll\overline p_s\).  Applying
\eqref{eq:heat-smoothed-fisher-bound} separately to \(\eta=q_0\) and
\(\eta=p_0\) gives
\[
    J(\overline q_s)<\infty,
    \qquad
    J(\overline p_s)<\infty,
\]
while \eqref{eq:heat-smoothed-relative-entropy-bound} gives
\[
    \mathrm{KL}(\overline q_s\|\overline p_s)<\infty.
\]
These are precisely the hypotheses of the cited theorem.

The integral representation in the cited theorem, together with the
heat-semigroup property, gives
\begin{equation}
\label{eq:heat-relative-entropy-tail}
    \mathrm{KL}(\overline q_s\|\overline p_s)
    =
    \frac12\int_s^\infty
    I(\overline q_v\|\overline p_v)\,\rmd v.
\end{equation}
It follows that
\(s\mapsto\mathrm{KL}(\overline q_s\|\overline p_s)\) is locally absolutely
continuous and, for \(0<s_-<s_+<\infty\),
\begin{equation}
\label{eq:heat-relative-de-bruijn-integrated}
    \mathrm{KL}(\overline q_{s_+}\|\overline p_{s_+})
    +
    \frac12\int_{s_-}^{s_+}
    I(\overline q_s\|\overline p_s)\,\rmd s
    =
    \mathrm{KL}(\overline q_{s_-}\|\overline p_{s_-}).
\end{equation}

We now transfer this identity to the OU clock. A common invertible dilation
leaves relative entropy invariant and rescales relative Fisher information.
In the present notation,
\begin{align*}
\tau_t&:=\frac{\sigma_t^2}{a_t^2}=e^t-1,
&q_t&=(S_{a_t})_\#\overline q_{\tau_t},
&p_t&=(S_{a_t})_\#\overline p_{\tau_t},\\
\mathrm{KL}(q_t\|p_t)
&=\mathrm{KL}(\overline q_{\tau_t}\|\overline p_{\tau_t}),
&I(q_t\|p_t)
&=a_t^{-2}I(\overline q_{\tau_t}\|\overline p_{\tau_t}),
&\tau_t'&=e^t=a_t^{-2}.
\end{align*}
Consequently, under the change of variables \(s=\tau_t\), we have
\(\rmd s=\tau_t'\,\rmd t=a_t^{-2}\,\rmd t\), and
\[
\int_{\tau_{t_-}}^{\tau_{t_+}}
I(\overline q_s\|\overline p_s)\,\rmd s
=
\int_{t_-}^{t_+}I(q_t\|p_t)\,\rmd t.
\]
Together with
\(\mathrm{KL}(\overline q_{\tau_{t_\pm}}\|\overline p_{\tau_{t_\pm}})
=\mathrm{KL}(q_{t_\pm}\|p_{t_\pm})\),
substituting \(s_\pm=\tau_{t_\pm}\) into
\eqref{eq:heat-relative-de-bruijn-integrated} yields
\eqref{eq:relative-de-bruijn-integrated}.  The local absolute continuity
and the almost-everywhere differential identity
\eqref{eq:relative-de-bruijn-differential} follow from the same change of
variables.
\end{proof}
\subsection{Proof of Proposition~\ref{prop:ou-score-oracle}}
\label{app:general-ou-oracle}

The proof uses the fixed-time Gaussian-channel estimate in
Lemma~\ref{lem:fixed-time-ou-kl} and the $\mathcal P_2$ relative de Bruijn identity in
Lemma~\ref{lem:relative-de-bruijn-ou}.  The only additional ingredient is the following elementary
change-of-reference inequality.

\begin{lemma}[Three-density KL inequality]
\label{lem:three-density-kl}
Let $q,p,r$ be strictly positive probability densities.  If
$\mathrm{KL}(q\|p)<\infty$ and $D_2(p\|r)<\infty$, then
\begin{equation}
  \mathrm{KL}(q\|r)
  \le
  2\mathrm{KL}(q\|p)
  +
  D_2(p\|r).
  \label{eq:three-density-kl}
\end{equation}
\end{lemma}

\begin{proof}
For any bounded measurable function $g$, the entropy variational inequality applied to $g/2$
gives
\begin{align*}
  \mathbb E_q g
  &\le
  2\mathrm{KL}(q\|p)
  +2\log\mathbb E_p e^{g/2},\\
  \mathbb E_p e^{g/2}
  &=
  \mathbb E_r\left[\frac{p}{r}e^{g/2}\right]
  \le
  \left(\int\frac{p^2}{r}\right)^{1/2}
  \left(\mathbb E_r e^g\right)^{1/2}.
\end{align*}
Here the second line rewrites the expectation under $r$ and applies
Cauchy--Schwarz.
Therefore,
\[
  \mathbb E_q g-\log\mathbb E_r e^g
  \le
  2\mathrm{KL}(q\|p)+D_2(p\|r).
\]
Taking the supremum over bounded measurable $g$ in the variational representation of
$\mathrm{KL}(q\|r)$ proves \eqref{eq:three-density-kl}.
\end{proof}

\begin{proofof}{Proposition~\ref{prop:ou-score-oracle}}
Fix a realization of $\hat\mu_0$ and let $q_0\in\mathcal P_2(\mathbb R^d)$.  At time
$t_0$, apply Lemma~\ref{lem:three-density-kl} with $q=q_{t_0}$,
$p=p_{t_0}$, and $r=\hat p_{t_0}$. All three densities are strictly
positive. Lemma~\ref{lem:fixed-time-ou-kl} gives
\(2\mathrm{KL}(q_{t_0}\|p_{t_0})
\le a_{t_0}^2\wass_2^2(q_0,p_0)/\rho_{t_0}\),
and $D_2(p_{t_0}\|\hat p_{t_0})<\infty$ by assumption.  Hence
\[
  \mathrm{KL}\left(q_{t_0}\middle\|\hat p_{t_0}\right)
  \le
  \frac{a_{t_0}^2}{\rho_{t_0}}\wass_2^2(q_0,p_0)
  +D_2\left(p_{t_0}\middle\|\hat p_{t_0}\right),
\]
which is \eqref{eq:generic-shifted-kl}.

Because $q_0,\hat\mu_0\in\mathcal P_2(\mathbb R^d)$,
Lemma~\ref{lem:relative-de-bruijn-ou} applies directly to the two OU trajectories
$q_t=P_tq_0$ and $\hat p_t=P_t\hat\mu_0$.  It yields
\begin{align*}
&\int_{t_0}^T
\mathbb E_{q_t}
\left[
  \|\hat s(X,t)-\nabla\log q_t(X)\|_2^2
\right]\rmd t
\\
&\qquad=
2\left[
  \mathrm{KL}\left(q_{t_0}\middle\|\hat p_{t_0}\right)
  -
  \mathrm{KL}\left(q_T\middle\|\hat p_T\right)
\right]
\\
&\qquad\le
  2\mathrm{KL}\left(q_{t_0}\middle\|\hat p_{t_0}\right).
\end{align*}
Combining the preceding display with \eqref{eq:generic-shifted-kl}, dividing by
$T-t_0$, and using $a_{t_0}^2/\rho_{t_0}=1/\tau_0$ gives, for every
$q_0\in\mathcal Q_\varepsilon(p_0)$,
\[
  \mathcal L_{q_0}(\hat s)
  \le
  \frac{2}{(T-t_0)\tau_0}\wass_2^2(q_0,p_0)
  +
  \frac{2}{T-t_0}
  D_2\left(p_{t_0}\middle\|\hat p_{t_0}\right),
\]
for almost every realization of $\hat\mu_0$. Fixing $q_0$ and taking expectation over
the randomness of $\hat\mu_0$ yields
\[
  \mathbb E\mathcal L_{q_0}(\hat s)
  \le
  \frac{2\varepsilon^2}{(T-t_0)\tau_0}
  +
  \frac{2}{T-t_0}
  \mathbb E D_2\left(p_{t_0}\middle\|\hat p_{t_0}\right).
\]
The right-hand side is independent of $q_0$. Taking the supremum over
$q_0\in\mathcal Q_\varepsilon(p_0)$ and using
\eqref{eq:robust-shifted-score-risk} proves \eqref{eq:generic-ou-score-oracle}.
\end{proofof}

\subsection{Proof of Proposition~\ref{prop:reverse-chi-square}}
\label{app:gaussian-domination}

The proof uses the following two ingredients.  The first establishes the
Gaussian domination property underlying the blanket construction, while the
second is a general add-one bound for dominated kernels.  Their proofs are
deferred until after the short deduction of
Proposition~\ref{prop:reverse-chi-square}.

\begin{lemma}[Gaussian blanket domination]
\label{lem:gaussian-blanket-domination}
For \(t\ge t_0\), \(y\in\mathcal B(0,D)\), and \(x\in\mathbb R^d\), let
\(k_{t,y}(x):=\varphi_{\rho_t}(x-a_ty)\). Then
\begin{equation}
  k_{t,y}(x)
  \le
  M_{\vartheta,t_0}\varphi_{v_t}(x).
  \label{eq:kernel-blanket-domination}
\end{equation}
More precisely,
\begin{equation}
  \sup_{x\in\mathbb R^d}
  \frac{k_{t,y}(x)}{\varphi_{v_t}(x)}
  =
  \left(1+\frac{\vartheta^2}{e^t-1}\right)^{d/2}
  \exp\left(\frac{\|y\|_2^2}{2\vartheta^2}\right).
  \label{eq:exact-gaussian-ratio}
\end{equation}
\end{lemma}

\begin{lemma}[Dominated-kernel add-one bound]
\label{lem:dominated-kernel-add-one}
Let $(\mathsf Y,\mathcal A)$ be a measurable space and let $n\ge1$. Suppose
$Y_1,\ldots,Y_n$ are i.i.d. copies of a random element $Y$. Let
$k:\mathsf Y\times\mathbb R^d\to[0,\infty)$ be jointly measurable, write
$k_y(x):=k(y,x)$, and suppose that $k_y$ is a probability density on
$\mathbb R^d$ for every $y\in\mathsf Y$. Let $r$ be a strictly positive
probability density. Assume that, for some $M\ge1$,
\begin{equation}
  k_y(x)\le Mr(x)
  \qquad
  \text{for every }x\in\mathbb R^d\text{ and every }y\in\mathsf Y.
  \label{eq:abstract-domination}
\end{equation}
Define
\[
  p(x):=\mathbb E[k_Y(x)],
  \qquad
  \hat p_n(x):=
  \frac{\sum_{i=1}^n k_{Y_i}(x)+Mr(x)}{n+M}.
\]
Then
\begin{equation}
  \mathbb E\int_{\mathbb R^d}\frac{p(x)^2}{\hat p_n(x)}\,\rmd x
  \le
  \frac{n+M}{n+1}.
  \label{eq:abstract-add-one}
\end{equation}
\end{lemma}

\begin{proofof}{Proposition~\ref{prop:reverse-chi-square}}
Fix $t\in[t_0,T]$.  Take $\mathsf Y=\mathcal B(0,D)$, equipped with its Borel
$\sigma$-field, and apply Lemma~\ref{lem:dominated-kernel-add-one} with
$Y_i=X_i$, $k_y(x)=k_{t,y}(x)=\varphi_{\rho_t}(x-a_ty)$,
$r(x)=\varphi_{v_t}(x)$, and $M=M_{\vartheta,t_0}$.
The domination assumption follows from Lemma~\ref{lem:gaussian-blanket-domination}, while
$p(x)=\mathbb E[k_{t,X}(x)]=p_t(x)$ for a generic draw $X\sim p_0$, and the regularized density
in the abstract lemma is exactly
$\hat p_{n,t}^\vartheta$. Therefore,
\begin{equation}
  \mathbb E_{X_{1:n}}
  \int_{\mathbb R^d}
  \frac{p_t(x)^2}{\hat p_{n,t}^\vartheta(x)}\,\rmd x
  \le
  \frac{n+M_{\vartheta,t_0}}{n+1}.
  \label{eq:add-one-second-moment}
\end{equation}

Finally, the concavity of the logarithm implies
\begin{align*}
\mathbb E_{X_{1:n}}
D_2\left(p_t\middle\|\hat p_{n,t}^\vartheta\right)
&=
\mathbb E_{X_{1:n}}\log\int\frac{p_t^2}{\hat p_{n,t}^\vartheta}
\\
&\le
\log\mathbb E_{X_{1:n}}\int\frac{p_t^2}{\hat p_{n,t}^\vartheta}
\le
\log\frac{n+M_{\vartheta,t_0}}{n+1},
\end{align*}
which is \eqref{eq:renyi-rate}.
\end{proofof}

\begin{proofof}{Lemma~\ref{lem:gaussian-blanket-domination}}
Fix $t\ge t_0$ and $y\in\mathcal B(0,D)$.  Recall that
\(k_{t,y}(x)=\varphi_{\rho_t}(x-a_ty)\) and
\(v_t=\rho_t+a_t^2\vartheta^2\).
For every $x\in\mathbb R^d$,
\begin{align*}
\log\frac{k_{t,y}(x)}{\varphi_{v_t}(x)}
&=
\frac d2\log\frac{v_t}{\rho_t}
-
\frac{\|x-a_ty\|_2^2}{2\rho_t}
+
\frac{\|x\|_2^2}{2v_t}.
\end{align*}
The last two terms form a strictly concave quadratic function of $x$.  Since
$v_t-\rho_t=a_t^2\vartheta^2$, completing the square gives
\begin{equation}
-
\frac{\|x-a_ty\|_2^2}{2\rho_t}
+
\frac{\|x\|_2^2}{2v_t}
=
\frac{\|y\|_2^2}{2\vartheta^2}
-
\frac{a_t^2\vartheta^2}{2\rho_t v_t}
\left\|x-\frac{v_t}{a_t\vartheta^2}y\right\|_2^2.
\label{eq:gaussian-blanket-completed-square}
\end{equation}
Taking the supremum over $x$ in the preceding identity and simplifying the
variance ratio gives
\begin{align*}
\sup_{x\in\mathbb R^d}
  \frac{k_{t,y}(x)}{\varphi_{v_t}(x)}
&=
\left(\frac{v_t}{\rho_t}\right)^{d/2}
\exp\left(\frac{\|y\|_2^2}{2\vartheta^2}\right),\\
  \frac{v_t}{\rho_t}
  &=1+\frac{a_t^2\vartheta^2}{\rho_t}
  =1+\frac{\vartheta^2}{e^t-1}.
\end{align*}
These identities give \eqref{eq:exact-gaussian-ratio}. Since
$t\mapsto(e^t-1)^{-1}$ is decreasing and
$\|y\|_2\le D$, the right-hand side is bounded by $M_{\vartheta,t_0}$ from
\eqref{eq:blanket-calibration}.  This proves \eqref{eq:kernel-blanket-domination}.
\end{proofof}

\begin{proofof}{Lemma~\ref{lem:dominated-kernel-add-one}}
Fix $x\in\mathbb R^d$ and write
\begin{align*}
U_i(x)&:=\frac{k_{Y_i}(x)}{Mr(x)}\in[0,1],
&\theta(x)&:=\mathbb E[U_i(x)]=\frac{p(x)}{Mr(x)}\in[0,1].
\end{align*}
Then
\begin{equation}
  \frac{p(x)^2}{\hat p_n(x)}
  =
  Mr(x)\theta(x)^2(n+M)
  \frac{1}{1+\sum_{i=1}^nU_i(x)}.
  \label{eq:pointwise-add-one-start}
\end{equation}
For every $c>0$, the map $u\mapsto(c+u)^{-1}$ is convex on $[0,1]$.
Hence, if $U\in[0,1]$ has mean $\theta$, the chord inequality gives
\[
  \mathbb E\left[\frac{1}{c+U}\right]
  \le
  \frac{1-\theta}{c}+\frac{\theta}{c+1}.
\]
Let $V_1,\ldots,V_n$ be i.i.d.\
$\operatorname{Bernoulli}(\theta(x))$ random variables, independent of
$U_1(x),\ldots,U_n(x)$. Conditional on the remaining summands, the
preceding inequality shows that replacing $U_i(x)$ by $V_i$ can only
increase the expectation. Repeating this replacement successively for
$i=1,\ldots,n$ yields
\begin{equation}
  \mathbb E\frac{1}{1+\sum_{i=1}^nU_i(x)}
  \le
  \mathbb E\frac{1}{1+\sum_{i=1}^nV_i}
  =
  \mathbb E\frac{1}{1+B},
  \qquad
  B:=\sum_{i=1}^nV_i
  \sim\operatorname{Bin}(n,\theta(x)).
  \label{eq:bernoulli-extremal}
\end{equation}
For $\theta\in(0,1]$,
\begin{align*}
  \mathbb E\frac{1}{1+B}
  &=
  \int_0^1\mathbb E[z^B]\,\rmd z
  =
  \int_0^1(1-\theta+\theta z)^n\,\rmd z
  \\
  &=
  \frac{1-(1-\theta)^{n+1}}{(n+1)\theta}
  \le
  \frac{1}{(n+1)\theta}.
\end{align*}
If $\theta(x)=0$, then $p(x)=0$ and the left-hand side of
\eqref{eq:pointwise-add-one-start} vanishes.  If $\theta(x)>0$, substituting the preceding estimate
into \eqref{eq:pointwise-add-one-start} yields
\[
  \mathbb E\frac{p(x)^2}{\hat p_n(x)}
  \le
  Mr(x)\theta(x)\frac{n+M}{n+1}
  =
  p(x)\frac{n+M}{n+1}.
\]
Tonelli's theorem and $\int p=1$ now give \eqref{eq:abstract-add-one}.
\end{proofof}

We record the consequences of Lemma~\ref{lem:gaussian-blanket-domination}
needed below. For every probability measure $\mu$ supported in
$\mathcal B(0,D)$,
\begin{equation}
  P_t\mu(x)
  =\int k_{t,y}(x)\,\mu(\rmd y)
  \le M_{\vartheta,t_0}\varphi_{v_t}(x),
  \qquad t\in[t_0,T].
  \label{eq:any-mixture-dominated}
\end{equation}
Moreover, by construction,
\begin{equation}
  \hat p_{n,t}^\vartheta(x)
  \ge
  \frac{M_{\vartheta,t_0}}{n+M_{\vartheta,t_0}}\varphi_{v_t}(x).
  \label{eq:estimator-lower-blanket}
\end{equation}
Taking $\mu=p_0$ in \eqref{eq:any-mixture-dominated} and combining the two displays gives
\begin{equation}
  \frac{p_t(x)}{\hat p_{n,t}^\vartheta(x)}
  \le n+M_{\vartheta,t_0},
  \qquad t\in[t_0,T],\ x\in\mathbb R^d.
  \label{eq:p-over-phat-bound}
\end{equation}
Multiplying \eqref{eq:p-over-phat-bound} by $p_t$ and integrating gives
\[
  \int_{\mathbb R^d}
  \frac{p_t(x)^2}{\hat p_{n,t}^\vartheta(x)}\,\rmd x
  \le
  \bigl(n+M_{\vartheta,t_0}\bigr)\int_{\mathbb R^d}p_t(x)\,\rmd x
  =n+M_{\vartheta,t_0}<\infty.
\]
Thus $D_2(p_t\|\hat p_{n,t}^\vartheta)<\infty$ whenever
$X_1,\ldots,X_n\in\mathcal B(0,D)$, and hence $p_0^{\otimes n}$-almost surely. This
verifies the finiteness assumption in Proposition~\ref{prop:ou-score-oracle}.

\subsection{Proof of Theorem~\ref{thm:finite-sample-robust-kde}}
\label{app:finite-sample-assembly}

\begin{proof}
Assumption~\ref{assum:support} implies $p_0\in\mathcal P_2(\mathbb R^d)$. Since
$X_i\in\mathcal B(0,D)$ almost surely, for $p_0^{\otimes n}$-almost every
realization of the training sample, the blanket initial measure belongs to
\(\mathcal P_2(\mathbb R^d)\), since
\[
  \int_{\mathbb R^d}\|x\|_2^2\,
  \hat\mu_{n,0}^\vartheta(\rmd x)
  =
  \frac{
    \sum_{i=1}^n\|X_i\|_2^2
    +
    M_{\vartheta,t_0}d\vartheta^2
  }{
    n+M_{\vartheta,t_0}
  }
  <\infty.
\]
Moreover, the pointwise bound \eqref{eq:p-over-phat-bound} guarantees
$D_2(p_{t_0}\|\hat p_{n,t_0}^\vartheta)<\infty$ for every such realization.
Therefore, Proposition~\ref{prop:ou-score-oracle} applies to the random measure
$\hat\mu_0=\hat\mu_{n,0}^\vartheta$ and gives
\[
  \mathcal R_{\varepsilon,p_0}(\hat s_n^\vartheta)
  \le
  \frac{2\varepsilon^2}{(T-t_0)\tau_0}
  +
  \frac{2}{T-t_0}
  \mathbb E_{X_{1:n}}
  D_2\left(
    p_{t_0}\middle\|\hat p_{n,t_0}^\vartheta
  \right).
\]
Applying Proposition~\ref{prop:reverse-chi-square} at \(t=t_0\), and then using
\(\log(1+u)\le u\) for \(u\ge0\), yields
\begin{align}
  \mathcal R_{\varepsilon,p_0}(\hat s_n^\vartheta)
  &\le
  \frac{2\varepsilon^2}{(T-t_0)\tau_0}
  +
  \frac{2}{T-t_0}
  \log\left(1+\frac{M_{\vartheta,t_0}-1}{n+1}\right)
  \label{eq:exact-finite-sample-robust-bound}
  \\
  &\le
  \frac{2\varepsilon^2}{(T-t_0)\tau_0}
  +
  \frac{2(M_{\vartheta,t_0}-1)}{(n+1)(T-t_0)}.
  \label{eq:linearized-finite-sample-bound}
\end{align}

It remains to prove the explicit rate specialization. Set
$\vartheta^2=D^2/d$. Since $\tau_0=e^{t_0}-1$, we have
\begin{equation}
  M_{\vartheta,t_0}
  =
  e^{d/2}
  \left(
    1+\frac{D^2}{d\tau_0}
  \right)^{d/2}.
  \label{eq:M-simple-blanket-scale}
\end{equation}
If $\tau_0\le1$, then
\[
  1+\frac{D^2}{d\tau_0}
  \le
  \tau_0^{-1}\left(1+\frac{D^2}{d}\right),
  \qquad
  M_{\vartheta,t_0}
  \le
  e^{d/2}
  \left(1+\frac{D^2}{d}\right)^{d/2}\tau_0^{-d/2}.
\]
Substituting this into \eqref{eq:linearized-finite-sample-bound} and using
$M_{\vartheta,t_0}-1\le M_{\vartheta,t_0}$ proves
\eqref{eq:rate-finite-sample-robust-bound} with the constant $C_{d,D}$
specified in Theorem~\ref{thm:finite-sample-robust-kde}. The order comparison
following that theorem follows from \(n+1\asymp n\) and
\(\tau_0=e^{t_0}-1\asymp t_0\) whenever \(\tau_0\le1\).
\end{proof}

\section{Proofs for Statistical Lower Bounds}
\label{app:statistical-minimax-lower-bound}

This appendix proves the statistical lower bounds for time-integrated score
estimation and positive-time KL estimation. The argument is self-contained
apart from the standard Gilbert--Varshamov packing lemma and Fano's
inequality, and its construction is adapted from the high-noise
score-estimation lower bound of \citet{dou2024optimal}. The essential
additional ingredient for the integrated problem is that a single hard family
remains score-separated throughout the heat-time window
\([\tau_0,2\tau_0]\). Subsections~\ref{app:ou-to-heat-reduction}--
\ref{app:fano-proof-statistical-lower} prove
Theorem~\ref{thm:integrated-statistical-lower-bound};
Subsections~\ref{app:fixed-time-kl-separation}
and~\ref{app:nominal-kl-minimax-proof} reuse the same hard family to prove the
nominal positive-time KL lower bound. The robust minimax conclusions are
completed in Appendix~\ref{app:intrinsic-lower-bound} by combining these
statistical bounds with the corresponding shift lower bounds.

Throughout the appendix, \(\phi_\tau\) denotes the density of \(\mathcal N(0,\tau I_d)\), and
\[
  (\mathsf H_\tau f)(x):=(\phi_\tau*f)(x)
\]
denotes the heat semigroup.  Thus
\(\partial_\tau\mathsf H_\tau f=\frac12\Delta\mathsf H_\tau f\).  The bump used below is fixed
explicitly for each dimension, and every constant that enters the final lower bound is specified
below.

\subsection{Reduction to the Integrated Heat-Score Problem}
\label{app:ou-to-heat-reduction}

The first reduction is immediate but important.  Since
\(p_0\in\mathcal Q_\varepsilon(p_0)\), the definition of the robust risk gives, for every
estimator \(\hat s_n\),
\begin{equation}
  \mathcal R_{\varepsilon,p_0}(\hat s_n)
  \ge
  \mathbb E
  \left[
    \frac{1}{T-t_0}
    \int_{t_0}^{T}
    \mathbb E_{X\sim p_t}
    \bigl[\|\hat s_n(X,t)-\nabla\log p_t(X)\|_2^2\bigr]\,\rmd t
  \right].
\label{eq:robust-dominates-nominal}
\end{equation}
Here the expectation is over $X_{1:n}\sim p_0^{\otimes n}$ and any internal
randomization used by the estimator. It therefore suffices to lower bound the
nominal time-integrated score-estimation problem.

We next record the exact correspondence between the OU and variance-exploding parametrizations.

\begin{lemma}[Exact OU--heat loss identity]
\label{lem:ou-heat-loss-identity}
Let \(p_0\) be any Borel probability measure and define
\[
  a_t:=e^{-t/2},
  \qquad
  \tau_t:=\frac{1-a_t^2}{a_t^2}=e^t-1,
  \qquad
  g_\tau:=p_0*\phi_\tau,
  \qquad
  u_\tau:=\nabla\log g_\tau.
\]
Then
\begin{equation}
  p_t(x)=a_t^{-d}g_{\tau_t}(x/a_t),
  \qquad
  \nabla\log p_t(x)=a_t^{-1}u_{\tau_t}(x/a_t).
\label{eq:ou-heat-density-score-scaling}
\end{equation}
For an arbitrary score estimator \(\hat s_n(x,t)\), define
\begin{equation}
  \tilde s_n(z,\tau)
  :=a_t\hat s_n(a_tz,t),
  \qquad
  t=\log(1+\tau).
\label{eq:transformed-score-estimator}
\end{equation}
Then, for every \(0<t_a<t_b\),
\begin{equation}
\begin{aligned}
  \int_{t_a}^{t_b}
  \|\hat s_n(\cdot,t)-\nabla\log p_t\|_{L^2(p_t)}^2\,\rmd t
  =
  \int_{\tau_{t_a}}^{\tau_{t_b}}
  \|\tilde s_n(\cdot,\tau)-u_\tau\|_{L^2(g_\tau)}^2\,\rmd \tau.
\end{aligned}
\label{eq:exact-ou-heat-loss-identity}
\end{equation}
\end{lemma}

\begin{proof}
The OU representation can be rewritten as
\(X_t=a_tX_0+\sqrt{1-a_t^2}\,Z
=a_t(X_0+\sqrt{e^t-1}\,Z)\), where \(Z\sim\mathcal N(0,I_d)\) is independent of
\(X_0\).
This proves the density scaling in \eqref{eq:ou-heat-density-score-scaling}; differentiating its logarithm gives the score scaling.  After the change of variables \(x=a_tz\),
\begin{align*}
  \|\hat s_n(\cdot,t)-\nabla\log p_t\|_{L^2(p_t)}^2
  &=
  a_t^{-2}
  \|\tilde s_n(\cdot,\tau_t)-u_{\tau_t}\|_{L^2(g_{\tau_t})}^2.
\end{align*}
Finally, \(\rmd\tau_t/\rmd t=e^t=a_t^{-2}\), equivalently
\(\rmd t=a_t^2\,\rmd\tau\).  The two scaling factors cancel exactly, proving
\eqref{eq:exact-ou-heat-loss-identity}.
\end{proof}

Set \(\tau_0=e^{t_0}-1\) and \(t_1:=\log(1+2\tau_0)\). Since
\(e^{2t_0}-(1+2\tau_0)=e^{2t_0}-2e^{t_0}+1=(e^{t_0}-1)^2\ge0\), we have
\(t_1\le2t_0\). Moreover,
\(1+2\tau_0>1+\tau_0=e^{t_0}\), so \(t_1>t_0\).  Under the assumption
\(2t_0\le T\), the nonempty interval \([t_0,t_1]\) is contained in \([t_0,T]\).
Lemma~\ref{lem:ou-heat-loss-identity} therefore reduces the desired lower bound to the
heat-time window \([\tau_0,2\tau_0]\).

\subsection{Construction of a Compactly Supported Hard Family}
\label{app:hard-family-construction}

The calibration has two targets: the sample KL divergence must remain below a fixed fraction of
the packing entropy, while the associated heat scores remain separated throughout
\([\tau_0,2\tau_0]\).  We therefore use cells of width
\(\rho=K_0\sqrt{\tau_0}\), giving a number \(m\) of cells of order
\(\tau_0^{-d/2}\), and later choose \(\eta^2=\kappa(1\wedge m/n)\).  The exact admissible
constants are fixed below.

Fix
\[
  \ell:=\frac{D}{8\sqrt d},
  \qquad
  Q:=[-4\ell,4\ell]^d,
  \qquad
  I:=[-3\ell,3\ell]^d.
\]
The cube \(Q\) is contained in \(\mathcal B(0,D/2)\), hence in \(\mathcal B(0,D)\). Let
\(c_Q:=|Q|^{-1}=(8\ell)^{-d}\) and \(f_\circ:=c_Q\mathbf 1_Q\).
To remove any auxiliary dependence from the construction, define
\[
  \psi(u)
  :=
  \begin{cases}
    \exp\left(-\dfrac{1}{1-16u^2}\right), & |u|<1/4,\\[1ex]
    0, & |u|\ge1/4,
  \end{cases}
  \qquad
  \Psi_d(x):=\prod_{k=1}^d\psi(x_k),
\]
and fix the dimension-dependent bump
\begin{equation}
  w=w_d:=\frac{\partial_1\Psi_d}{\|\partial_1\Psi_d\|_\infty}.
\label{eq:explicit-hard-family-bump}
\end{equation}
Then \(w_d\in C_c^\infty((-1/2,1/2)^d)\), \(w_d\not\equiv0\), and
\begin{equation}
  \int_{\mathbb R^d}w(x)\,\rmd x=0,
  \qquad
  \|w\|_\infty=1,
  \qquad
  A_1:=\int\|\nabla w\|_2^2>0.
\label{eq:bump-assumptions}
\end{equation}
We also write
\[
  A_0:=\int_{\mathbb R^d}w^2,
  \qquad
  A_2:=\int_{\mathbb R^d}(\Delta w)^2,
  \qquad
  L_w:=\|\nabla w\|_\infty.
\]
Thus \(A_0,A_1,A_2,L_w\in(0,\infty)\) are fixed constants determined solely
by \(d\).

Set
\begin{equation}
  K_0:=\max\left\{
    1,
    \sqrt{\frac{8A_2}{A_1}}
  \right\},
  \qquad
  \rho:=K_0\sqrt{\tau_0}.
\label{eq:persistence-scale-constant}
\end{equation}
In particular,
\begin{equation}
  A_1-\frac{2A_2}{K_0^2}\ge\frac{3A_1}{4}.
\label{eq:K0-energy-choice}
\end{equation}
For the grid below, define \(\{z_1,\ldots,z_m\}\) to be the lattice points
\begin{equation}
  2\rho\mathbb Z^d\cap[-(\ell-\rho),\ell-\rho]^d.
\label{eq:explicit-packing-grid}
\end{equation}
Whenever \(\rho\le\ell/4\), the cubes
\(z_j+\rho[-1/2,1/2]^d\), \(1\le j\le m\), are pairwise disjoint and
\begin{equation}
  2^{-d}\left(\frac{\ell}{\rho}\right)^d
  \le m\le
  \left(\frac{\ell}{\rho}\right)^d.
\label{eq:grid-cardinality}
\end{equation}
Indeed, if \(i\ne j\), then
\(|(z_i)_k-(z_j)_k|\ge2\rho\) for some coordinate \(k\).
Since each displayed cube has coordinate half-width \(\rho/2\), the two cubes are
separated by at least \(\rho\) in that coordinate.  Hence the cubes are pairwise
disjoint; in particular, the rescaled bumps defined below have pairwise disjoint supports.

To count the lattice points, put \(y:=\ell/\rho\ge4\).  In each coordinate the admissible
points are \(2\rho k\) with \(|k|\le(y-1)/2\), so their number is
\(N=2\lfloor(y-1)/2\rfloor+1\). The elementary bounds
\(u-1\le\lfloor u\rfloor\le u\) give \(y-2\le N\le y\).
Since \(y\ge4\), it follows that \(y/2\le N\le y\).  The grid is the \(d\)-fold
Cartesian product of this one-dimensional set, so \(m=N^d\), which proves
\eqref{eq:grid-cardinality}.

Define \(w_{j,\rho}(x):=w((x-z_j)/\rho)\).
For \(b=(b_1,\ldots,b_m)\in\{0,1\}^m\) and an amplitude \(0<\eta\le1/4\), set
\begin{equation}
  f_b(x)
  :=
  f_\circ(x)
  \left(1+\eta\sum_{j=1}^m b_jw_{j,\rho}(x)\right).
\label{eq:hard-density-family}
\end{equation}
Because the bump supports are disjoint,
\(\lvert\sum_{j=1}^m b_jw_{j,\rho}\rvert\le1\).  Hence the factor in parentheses in
\eqref{eq:hard-density-family} lies in \([3/4,5/4]\).  The mean-zero condition in
\eqref{eq:bump-assumptions} gives \(\int f_b=1\).  Thus every \(f_b\) is a probability
density supported in \(Q\subset\mathcal B(0,D)\), with
  \(3f_\circ/4\le f_b\le5f_\circ/4\) on \(Q\).
For \(b,b'\in\{0,1\}^m\), write
\(d_{\mathrm{Ham}}(b,b'):=\sum_{j=1}^m\mathbf 1_{\{b_j\ne b_j'\}}\).

The following lemma records the information bound for this family.

\begin{lemma}[Sample KL bound]
\label{lem:hard-family-kl}
For every \(b\in\{0,1\}^m\),
\begin{equation}
  \mathrm{KL}(f_b^{\otimes n}\|f_\circ^{\otimes n})
  \le
  c_QA_0n\eta^2m\rho^d.
\label{eq:hard-family-kl}
\end{equation}
\end{lemma}

\begin{proof}
Write \(W_b:=\sum_{j=1}^m b_jw_{j,\rho}\).  Since \(f_b\ll f_\circ\),
tensorization and the standard inequality
\(\mathrm{KL}(P\|Q)\le\chi^2(P\|Q)\) give
\begin{align*}
  \mathrm{KL}(f_b^{\otimes n}\|f_\circ^{\otimes n})
  &\le n\int_Q\frac{(f_b-f_\circ)^2}{f_\circ}\,\rmd x
  =nc_Q\eta^2\int_QW_b^2\,\rmd x \\
  &=nc_QA_0\eta^2\rho^d\sum_{j=1}^m b_j
  \le c_QA_0n\eta^2m\rho^d.
\end{align*}
The second equality uses the pairwise disjoint bump supports and \(b_j^2=b_j\).
\end{proof}

\subsection{Uniform Score Separation over the Heat-Time Window}
\label{app:uniform-score-separation}

For \(\tau>0\), let
\[
  g_{b,\tau}:=\mathsf H_\tau f_b,
  \qquad
  s_{b,\tau}:=\nabla\log g_{b,\tau}.
\]
Define the dimensionally explicit envelope constants
\begin{equation}
  \bar\tau:=\frac{\ell^2}{2\log(4d)},
  \qquad
  C_1:=8\sqrt{\frac{d}{2\pi}},
  \qquad
  C_2:=4L_w.
\label{eq:explicit-interior-envelope-constants}
\end{equation}
We first establish common density and score envelopes on the fixed interior cube \(I\).

\begin{lemma}[Interior density and score bounds]
\label{lem:interior-density-score-bounds}
For every \(0<\tau\le\bar\tau\), every \(b\in\{0,1\}^m\), every \(x\in I\),
and every \(0<\eta\le1/4\), the constants in
\eqref{eq:explicit-interior-envelope-constants} satisfy
\begin{equation}
  \frac{c_Q}{4}
  \le g_{b,\tau}(x)
  \le \frac{5c_Q}{4},
\label{eq:interior-density-bounds}
\end{equation}
and
\begin{equation}
  \|s_{b,\tau}(x)\|_2
  \le
  C_1\tau^{-1/2}\exp\left(-\frac{\ell^2}{2\tau}\right)
  +C_2\frac{\eta}{\rho}.
\label{eq:interior-score-bound}
\end{equation}
\end{lemma}

\begin{proof}
For any bounded \(h\) and \(Z\sim\mathcal N(0,I_d)\),
\((\mathsf H_\tau h)(x)=\mathbb E[h(x+\sqrt\tau Z)]\).  Since the Gaussian
kernel is nonnegative and has unit mass, this representation gives
\(\|\mathsf H_\tau h\|_\infty\le\|h\|_\infty\); the same argument applies to
vector-valued \(h\) with the Euclidean norm.

For \(x\in I\), the distance in each coordinate from \(x\) to the boundary of \(Q\)
is at least \(\ell\).  A union bound and the Gaussian tail bound therefore give
\[
  (\mathsf H_\tau f_\circ)(x)
  =c_Q\mathbb P(x+\sqrt\tau Z\in Q)
  \ge c_Q\left[1-2d\exp\left(-\frac{\ell^2}{2\tau}\right)\right].
\]
Since \(\tau\le\bar\tau\) makes the bracket at least \(1/2\), while
\(\|f_b-f_\circ\|_\infty\le\eta c_Q\), the contraction property yields
\(c_Q/2-\eta c_Q\le g_{b,\tau}(x)\le c_Q+\eta c_Q\).  As
\(\eta\le1/4\), this proves \eqref{eq:interior-density-bounds}.

The base convolution factorizes coordinatewise.  Differentiating it explicitly gives, for each \(k\),
\begin{align*}
  |\partial_k\mathsf H_\tau f_\circ(x)|
  &\le
  2c_Q(2\pi\tau)^{-1/2}
  \exp\left(-\frac{\ell^2}{2\tau}\right).
\end{align*}
Moreover, although \(f_\circ\) is discontinuous at \(\partial Q\), the difference
\(f_b-f_\circ\) is a smooth compactly supported sum of bumps lying strictly inside \(Q\).
The disjoint supports give
\(\|\nabla(f_b-f_\circ)\|_\infty\le c_Q\eta\rho^{-1}L_w\).  Since
\(\mathsf H_\tau\) commutes with derivatives, its contraction property then gives
\(\|\nabla\mathsf H_\tau(f_b-f_\circ)\|_\infty
\le c_Q\eta\rho^{-1}L_w\).  Using
\(s_{b,\tau}=[\nabla\mathsf H_\tau f_\circ+
\nabla\mathsf H_\tau(f_b-f_\circ)]/g_{b,\tau}\) and
\eqref{eq:interior-density-bounds},
\begin{align*}
  \|s_{b,\tau}(x)\|_2
  &\le
  \frac{4}{c_Q}\left[
    2c_Q\sqrt{\frac{d}{2\pi}}\,\tau^{-1/2}
    \exp\left(-\frac{\ell^2}{2\tau}\right)
    +c_Q\frac{\eta}{\rho}L_w
  \right] \\
  &=C_1\tau^{-1/2}\exp\left(-\frac{\ell^2}{2\tau}\right)
    +C_2\frac{\eta}{\rho},
\end{align*}
which proves \eqref{eq:interior-score-bound}.
\end{proof}

We next state a standard off-diagonal heat-kernel estimate.  A proof is included to make the localization step explicit.

\begin{lemma}[Off-diagonal gradient estimate]
\label{lem:off-diagonal-gradient}
Let \(A,B\subseteq\mathbb R^d\) be nonempty measurable sets, and write
\(\operatorname{dist}(A,B):=\inf_{x\in A,\,y\in B}\|x-y\|_2\).  Suppose that
\(\operatorname{dist}(A,B)\ge r>0\) and that
\(h\in L^2(\mathbb R^d)\) is supported in \(B\).  Then, for every \(\tau>0\),
\begin{equation}
  \|\mathbf 1_A\nabla\mathsf H_\tau h\|_2
  \le
  C_{\mathrm{off},d}\tau^{-1/2}
  \exp\left(-\frac{r^2}{16\tau}\right)
  \|h\|_2,
\label{eq:off-diagonal-gradient}
\end{equation}
where \(C_{\mathrm{off},d}:=\sqrt d(4/3)^{d/4}\).
\end{lemma}

\begin{proof}
The integral kernel of \(\nabla\mathsf H_\tau\) is
\(\nabla\phi_\tau(z)=-(z/\tau)\phi_\tau(z)\).  For \(x\in A\) and
  \(y\in B\), \(\|x-y\|_2\ge r\).  Let \(Z\sim\mathcal N(0,I_d)\) and
\(R:=r/\sqrt\tau\).  The bound
\(\mathbf 1_{\{\|Z\|_2\ge R\}}
\le \exp((\|Z\|_2^2-R^2)/16)\), followed by Cauchy--Schwarz, gives
\begin{align*}
  \sup_{x\in A}
  \int_B\|\nabla\phi_\tau(x-y)\|_2\,\rmd y
  &\le
  \int_{\|z\|_2\ge r}\frac{\|z\|_2}{\tau}\phi_\tau(z)\,\rmd z
  \\
  &=\tau^{-1/2}
  \mathbb E\bigl[\|Z\|_2\mathbf 1_{\{\|Z\|_2\ge R\}}\bigr]
  \\
  &\le\tau^{-1/2}e^{-R^2/16}
  (\mathbb E\|Z\|_2^2)^{1/2}
  (\mathbb Ee^{\|Z\|_2^2/8})^{1/2}
  \\
  &=C_{\mathrm{off},d}\tau^{-1/2}
  \exp\left(-\frac{r^2}{16\tau}\right),
\end{align*}
where the last line uses
\(\mathbb E\|Z\|_2^2=d\),
\(\mathbb Ee^{\|Z\|_2^2/8}=(4/3)^{d/2}\), and \(R=r/\sqrt\tau\).
The same bound holds with the supremum over \(y\in B\) and the integral over
\(x\in A\).  Thus both Schur integrals of the scalar majorant
\(\mathbf 1_A(x)\mathbf 1_B(y)\|\nabla\phi_\tau(x-y)\|_2\) are bounded by
the right-hand coefficient in \eqref{eq:off-diagonal-gradient}.  Schur's
test applied to this majorant proves \eqref{eq:off-diagonal-gradient}.
\end{proof}

For \(r>0\), write \(\log_+r:=\max\{0,\log r\}\).  We now fix the remaining
localization parameters by setting
\begin{equation}
  R_{\mathrm{pers}}
  :=\frac{4C_{\mathrm{off},d}^2A_0K_0^2}{A_1},
  \qquad
  R_{\mathrm{sc}}
  :=\frac{32A_0C_1^2K_0^2}{A_1}.
\label{eq:localization-absorption-ratios}
\end{equation}
Define
\begin{equation}
\begin{split}
  \tau_\star
  :=\min\biggl\{
  &\frac{\ell^2}{16K_0^2},\,
  \frac{\bar\tau}{2},\,
  \frac{\ell^2}{16(1+\log_+R_{\mathrm{pers}})},\,
  \frac{\ell^2}{2(1+\log_+R_{\mathrm{sc}})}
  \biggr\},
  \\
  \eta_0
  &:=\min\left\{
  \frac14,\sqrt{\frac{A_1}{32A_0C_2^2}}
  \right\}.
\end{split}
\label{eq:explicit-localization-parameters}
\end{equation}
Both constants are positive and depend only on \(d,D\).  The four constraints defining
\(\tau_\star\), together with the displayed bound defining \(\eta_0\), will be invoked below.

The key persistence estimate follows from the heat equation.

\begin{lemma}[Persistence of gradient energy]
\label{lem:gradient-energy-persistence}
Let \(K_0\) and \(\tau_\star\) be given by
\eqref{eq:persistence-scale-constant} and \eqref{eq:explicit-localization-parameters}.
If \(0<\tau_0\le\tau_\star\), set \(\rho=K_0\sqrt{\tau_0}\) and
\(h_{b,b'}:=f_b-f_{b'}\).
Then, for every \(b,b'\in\{0,1\}^m\) and every \(\tau\in[\tau_0,2\tau_0]\),
\[
  \int_I
  \|\nabla\mathsf H_\tau h_{b,b'}(x)\|_2^2\,\rmd x
  \ge
  \frac{c_Q^2A_1}{2}\,
  \eta^2\rho^{d-2}d_{\mathrm{Ham}}(b,b').
\]
\end{lemma}

\begin{proof}
Write \(H=d_{\mathrm{Ham}}(b,b')\).  The case \(H=0\) is immediate, so assume
henceforth that \(H\ge1\).  The disjoint supports of the bumps give the exact identities
\begin{align*}
  \|\nabla h_{b,b'}\|_2^2
  &=c_Q^2\eta^2H\rho^{d-2}A_1,
  \\
  \|\Delta h_{b,b'}\|_2^2
  &=c_Q^2\eta^2H\rho^{d-4}A_2.
\end{align*}
Since the bump supports lie strictly inside \(Q\),
\(h_{b,b'}=c_Q\eta\sum_j(b_j-b_j')w_{j,\rho}\in C_c^\infty\).
Set \(u_s:=\mathsf H_s h_{b,b'}\) and \(E(s):=\|\nabla u_s\|_2^2\).
Since \(\partial_su_s=\frac12\Delta u_s\), differentiation under the
integral and integration by parts give
\begin{align*}
  E'(s)
  &=2\langle\nabla u_s,\nabla\partial_su_s\rangle_{L^2}
    =\langle\nabla u_s,\nabla\Delta u_s\rangle_{L^2}
    =-\|\Delta u_s\|_2^2,\\
  E(\tau)
  &=E(0)-\int_0^\tau\|\Delta\mathsf H_s h_{b,b'}\|_2^2\,\rmd s
    \ge E(0)-\tau\|\Delta h_{b,b'}\|_2^2.
\end{align*}
The last inequality uses
\(\Delta\mathsf H_s h_{b,b'}=\mathsf H_s\Delta h_{b,b'}\) and the
\(L^2\)-contraction of \(\mathsf H_s\).
Using the preceding identities, \(\tau\le2\tau_0\), and
\(\rho^2=K_0^2\tau_0\), we obtain
\[
  E(\tau)
  \ge
  c_Q^2\eta^2H\rho^{d-2}
  \left(A_1-\frac{2A_2}{K_0^2}\right).
\]
By \eqref{eq:K0-energy-choice}, the parenthesis is at least \(3A_1/4\).

By \eqref{eq:explicit-packing-grid},
\[
  \operatorname{supp}(w_{j,\rho})
  \subset z_j+\rho[-1/2,1/2]^d
  \subset[-(\ell-\rho/2),\ell-\rho/2]^d.
\]
The distance from this cube to \(I^c\) is \(2\ell+\rho/2\), hence at least
\(\ell\).  Lemma~\ref{lem:off-diagonal-gradient} therefore gives
\begin{align*}
  \|\mathbf 1_{I^c}\nabla\mathsf H_\tau h_{b,b'}\|_2^2
  &\le
  C_{\mathrm{off},d}^2\tau^{-1}
  \exp\left(-\frac{\ell^2}{8\tau}\right)
  \|h_{b,b'}\|_2^2
  \\
  &=
  C_{\mathrm{off},d}^2c_Q^2\eta^2H\rho^d\tau^{-1}
  \exp\left(-\frac{\ell^2}{8\tau}\right)A_0.
\end{align*}
Here \(\rho^2/\tau\le K_0^2\) and
\(e^{-\ell^2/(8\tau)}\le e^{-\ell^2/(16\tau_0)}\).  Moreover, the third
constraint in \eqref{eq:explicit-localization-parameters} ensures that
\(\ell^2/(16\tau_0)\ge1+\log_+R_{\mathrm{pers}}\).
Since \(R_{\mathrm{pers}}e^{-(1+\log_+R_{\mathrm{pers}})}\le1\), it follows
that \(R_{\mathrm{pers}}e^{-\ell^2/(16\tau_0)}\le1\).
Consequently, the tail energy is at most
\((c_Q^2A_1/4)\eta^2H\rho^{d-2}\).  Combining this estimate with the
global lower bound yields
\begin{align*}
  \int_I\|\nabla\mathsf H_\tau h_{b,b'}(x)\|_2^2\,\rmd x
  &=E(\tau)
    -\|\mathbf 1_{I^c}\nabla\mathsf H_\tau h_{b,b'}\|_2^2\\
  &\ge
    \left(\frac34-\frac14\right)c_Q^2A_1\eta^2H\rho^{d-2}
   =\frac{c_Q^2A_1}{2}\eta^2H\rho^{d-2},
\end{align*}
which proves the claim.
\end{proof}

We can now pass from the derivative of the smoothed density to the self-normalized score.

\begin{lemma}[Uniform score separation]
\label{lem:uniform-window-score-separation}
Let \(\tau_\star\) and \(\eta_0\) be given by
\eqref{eq:explicit-localization-parameters}.  If \(0<\tau_0\le\tau_\star\),
\(0<\eta\le\eta_0\), and \(\rho=K_0\sqrt{\tau_0}\),
then, for every \(b,b'\in\{0,1\}^m\) and every \(\tau\in[\tau_0,2\tau_0]\),
\[
  \int_I
  \|s_{b,\tau}(x)-s_{b',\tau}(x)\|_2^2\,\rmd x
  \ge
  \frac{2A_1}{25}\eta^2\rho^{d-2}d_{\mathrm{Ham}}(b,b').
\]
\end{lemma}

\begin{proof}
Because \(\tau\le2\tau_0\le2\tau_\star\le\bar\tau\),
Lemma~\ref{lem:interior-density-score-bounds} applies throughout the argument.
Set \(u:=g_{b,\tau}\), \(v:=g_{b',\tau}\), and
\(\delta_\tau:=u-v=\mathsf H_\tau(f_b-f_{b'})\).  On \(I\), the identity
\(\nabla\log u-\nabla\log v
=(\nabla\delta_\tau-\delta_\tau\nabla\log v)/u\) and the upper density bound in
\eqref{eq:interior-density-bounds} imply
\begin{align*}
  \int_I\|s_{b,\tau}-s_{b',\tau}\|_2^2
  &\ge
  \frac{16}{25c_Q^2}\int_I
  \|\nabla\delta_\tau-\delta_\tau s_{b',\tau}\|_2^2
  \\
  &\ge
  \frac{8}{25c_Q^2}\|\nabla\delta_\tau\|_{L^2(I)}^2
  -\frac{16}{25c_Q^2}\|\delta_\tau s_{b',\tau}\|_{L^2(I)}^2.
\end{align*}
Here we used \(\|a-b\|_2^2\ge\frac12\|a\|_2^2-\|b\|_2^2\).

Writing \(H=d_{\mathrm{Ham}}(b,b')\), the \(L^2\)-contraction of the heat
semigroup gives
\(\|\delta_\tau\|_2^2\le\|f_b-f_{b'}\|_2^2
=c_Q^2\eta^2H\rho^dA_0\).  By
Lemma~\ref{lem:interior-density-score-bounds}, the pointwise score bound,
\((a+b)^2\le2(a^2+b^2)\), and
\(\|\delta_\tau\|_{L^2(I)}\le\|\delta_\tau\|_2\),
\begin{align*}
  \|\delta_\tau s_{b',\tau}\|_{L^2(I)}^2
  &\le
  2\left[
    C_1^2\tau^{-1}\exp\left(-\frac{\ell^2}{\tau}\right)
    +C_2^2\frac{\eta^2}{\rho^2}
  \right]
  \|\delta_\tau\|_2^2
  \\
  &\le
  2c_Q^2A_0\eta^2H\rho^{d-2}
  \left[
    C_1^2\frac{\rho^2}{\tau}
    \exp\left(-\frac{\ell^2}{\tau}\right)
    +C_2^2\eta^2
  \right].
\end{align*}
For \(\tau\in[\tau_0,2\tau_0]\) and \(\rho^2=K_0^2\tau_0\), the first
term in brackets is at most
\(C_1^2K_0^2\exp(-\ell^2/(2\tau_0))\).
The fourth constraint in \eqref{eq:explicit-localization-parameters} and the definition of
\(R_{\mathrm{sc}}\) imply that this quantity is at most \(A_1/(32A_0)\).
The definition of \(\eta_0\) gives the same upper bound for \(C_2^2\eta^2\).
It follows that
\((16/(25c_Q^2))\|\delta_\tau s_{b',\tau}\|_{L^2(I)}^2
\le(2A_1/25)\eta^2H\rho^{d-2}\).
On the other hand, Lemma~\ref{lem:gradient-energy-persistence} shows that the positive
term in the decomposition above is at least
\((4A_1/25)\eta^2H\rho^{d-2}\).
Subtracting proves the claim.
\end{proof}

Integrating the separation over the full window gives the scale needed by Fano's method.

\begin{corollary}[Integrated packing separation]
\label{cor:integrated-packing-separation}
Under the assumptions of Lemma~\ref{lem:uniform-window-score-separation}, suppose that
\[
  d_{\mathrm{Ham}}(b,b')\ge \frac{m}{8}.
\]
Then
\begin{equation}
  \int_{\tau_0}^{2\tau_0}
  \int_I
  \|s_{b,\tau}(x)-s_{b',\tau}(x)\|_2^2\,\rmd x\,\rmd\tau
  \ge c_{\mathrm{int}}\eta^2,
  \qquad
  c_{\mathrm{int}}
  :=\frac{A_1\ell^d}{100\,2^dK_0^2}.
\label{eq:integrated-score-packing-separation}
\end{equation}
\end{corollary}

\begin{proof}
Integrate the bound in Lemma~\ref{lem:uniform-window-score-separation}
over \([\tau_0,2\tau_0]\) and use \(d_{\mathrm{Ham}}(b,b')\ge m/8\).
Since this interval has length \(\tau_0\), we obtain
\[
  \int_{\tau_0}^{2\tau_0}\!\int_I
  \|s_{b,\tau}-s_{b',\tau}\|_2^2\,\rmd x\,\rmd\tau
  \ge
  \frac{A_1}{100}\eta^2\tau_0m\rho^{d-2}.
\]
Since \(\rho^2=K_0^2\tau_0\) and
\(m\rho^d\ge2^{-d}\ell^d\) by \eqref{eq:grid-cardinality}, the right-hand side is at least
\(c_{\mathrm{int}}\eta^2\), as claimed.
\end{proof}

\subsection{Fano Reduction and Proof of Theorem~\ref{thm:integrated-statistical-lower-bound}}
\label{app:fano-proof-statistical-lower}

We use the Gilbert--Varshamov packing bound
\citep[Lemma~2.9]{tsybakov2009nonparametric}: for every \(m\ge8\), there exists a set
\(\mathcal B\subseteq\{0,1\}^m\), containing the zero vector, such that
\begin{equation}
  \log|\mathcal B|\ge \frac{\log2}{8}m,
  \qquad
  \min_{b\ne b'\in\mathcal B}
  d_{\mathrm{Ham}}(b,b')\ge\frac{m}{8}.
\label{eq:gilbert-varshamov}
\end{equation}
Below we use Fano's inequality
\citep[Corollary~2.6]{tsybakov2009nonparametric}.

For clarity, we collect the remaining constants before applying the two
results.  Set
\(\kappa:=\min\{\eta_0^2,8^d\log2/(128A_0)\}\), and define
\begin{equation}
\begin{aligned}
  c_0
  &:=
  \log\left(
    1+\min\left\{
      1,\,
      \tau_\star,\,
      \left(\frac{\ell}{K_0}\right)^2
      \left(\frac{1}{16\,2^d}\right)^{2/d}
    \right\}
  \right),\\
  c_1
  &:=
  \frac{7c_Qc_{\mathrm{int}}\kappa}{256}
  \min\left\{
    1,\,
    2^{-d}\left(\frac{\ell}{K_0}\right)^d
  \right\}.
\end{aligned}
\label{eq:explicit-statistical-lower-bound-constants}
\end{equation}
The constants \(\kappa,c_0,c_1\) are strictly positive and depend only on
\(d,D\).  The displayed \(c_0,c_1\) are admissible choices in
Theorem~\ref{thm:integrated-statistical-lower-bound}.

\begin{proofof}{Theorem~\ref{thm:integrated-statistical-lower-bound}}
Assume \(0<t_0\le c_0\).  By
\eqref{eq:explicit-statistical-lower-bound-constants},
\(\tau_0=e^{t_0}-1\) is
bounded above by each term in the minimum defining \(c_0\).  In particular,
\(\tau_0\le\tau_\star\), and the first constraint in
\eqref{eq:explicit-localization-parameters} gives
\(\rho=K_0\sqrt{\tau_0}\le\ell/4\).  Moreover,
\eqref{eq:grid-cardinality} and the other term in that minimum give
\(m\ge2^{-d}(\ell/(K_0\sqrt{\tau_0}))^d\ge16\), so the
Gilbert--Varshamov bound applies.
Let \(\mathcal B\) be the packing in \eqref{eq:gilbert-varshamov}.  Choose
\[
  \eta^2
  :=
  \kappa
  \left(1\wedge\frac{m}{n}\right),
\]
where \(\kappa\) is defined above; in particular,
\(0<\eta\le\eta_0\).  By Lemma~\ref{lem:hard-family-kl},
\eqref{eq:grid-cardinality}, and \(m\rho^d\le\ell^d\),
\begin{equation}
  \max_{b\in\mathcal B}
  \mathrm{KL}(f_b^{\otimes n}\|f_\circ^{\otimes n})
  \le
  \frac{A_0}{8^d}n\eta^2
  \le
  \frac{A_0}{8^d}\kappa m
  \le\frac{(\log2)m}{128}
  \le\frac1{16}\log|\mathcal B|.
\label{eq:fano-kl-calibration}
\end{equation}

Consider an arbitrary measurable, possibly randomized, vector-field estimator
\(\tilde s_n(x,\tau)\) based on the sample.  Represent the estimator's internal randomization by a seed
\(U\sim\lambda\), independent of the sample.  For \(b\in\mathcal B\), let
\[
  \mathbb P_b:=(f_b(x)\,\rmd x)^{\otimes n}\otimes\lambda,
  \qquad
  \mathbb E_b[\,\cdot\,]:=\int (\,\cdot\,)\,\rmd\mathbb P_b.
\]
Thus \(\mathbb P_b\) is the augmented observation law under hypothesis \(b\), and
\[
  \mathrm{KL}
  \bigl(f_b^{\otimes n}\otimes\lambda
  \,\|\,f_\circ^{\otimes n}\otimes\lambda\bigr)
  =
  \mathrm{KL}(f_b^{\otimes n}\|f_\circ^{\otimes n}).
\]
Thus the calibration in \eqref{eq:fano-kl-calibration} also applies to randomized estimators;
all probabilities \(\mathbb P_b\) and expectations \(\mathbb E_b\) below include the seed
\(U\).  Introduce the common semimetric
\begin{equation}
  d_{\mathcal I}(r,r')
  :=
  \left(
    \int_{\tau_0}^{2\tau_0}
    \int_I\|r(x,\tau)-r'(x,\tau)\|_2^2\,\rmd x\,\rmd\tau
  \right)^{1/2}.
\label{eq:common-score-semimetric}
\end{equation}
Let \(\hat b\) be a nearest-neighbour decoder of \(\tilde s_n\) in the finite family
\(\{s_{b,\cdot}:b\in\mathcal B\}\) with respect to \(d_{\mathcal I}\), using a fixed deterministic
rule to break ties.  Since \(\mathcal B\) is finite and each extended-real-valued map
\(d_{\mathcal I}(\tilde s_n,s_{b,\cdot})\), \(b\in\mathcal B\), is measurable, so is
\(\hat b\).  By Lemma~\ref{lem:interior-density-score-bounds}, for every \(b\in\mathcal B\),
\begin{align}
  \int_{\tau_0}^{2\tau_0}
  \|\tilde s_n(\cdot,\tau)-s_{b,\tau}\|_{L^2(g_{b,\tau})}^2\,\rmd\tau
  &\ge
  \frac{c_Q}{4}
  d_{\mathcal I}^2(\tilde s_n,s_{b,\cdot}).
\label{eq:weighted-risk-dominates-common-metric}
\end{align}
If the distance on the right is infinite with positive \(\mathbb P_b\)-probability for some \(b\),
then the corresponding expected risk is infinite and the desired lower bound is immediate.  Hence
we may assume that \(d_{\mathcal I}(\tilde s_n,s_{b,\cdot})<\infty\)
\(\mathbb P_b\)-almost surely for every \(b\in\mathcal B\).

By Corollary~\ref{cor:integrated-packing-separation}, the common separation satisfies
\[
  d_{\mathcal I}^2(s_{b,\cdot},s_{b',\cdot})
  \ge c_{\mathrm{int}}\eta^2
  \qquad(b\ne b').
\]
If \(\hat b\ne b\), the triangle inequality implies
\[
  \sqrt{c_{\mathrm{int}}}\,\eta
  \le d_{\mathcal I}(s_{\hat b,\cdot},s_{b,\cdot})
  \le d_{\mathcal I}(s_{\hat b,\cdot},\tilde s_n)
     +d_{\mathcal I}(\tilde s_n,s_{b,\cdot})
  \le 2d_{\mathcal I}(\tilde s_n,s_{b,\cdot}),
\]
where the last inequality follows from the nearest-neighbour property.  Therefore
\(d_{\mathcal I}^2(\tilde s_n,s_{b,\cdot})
\ge(c_{\mathrm{int}}/4)\eta^2\).
Since \(m\ge16\), \eqref{eq:gilbert-varshamov} gives
\(\log|\mathcal B|\ge(\log2)m/8\ge2\log2\).
Fano's inequality and \eqref{eq:fano-kl-calibration} therefore give
\begin{equation}
  \inf_{\hat b}
  \sup_{b\in\mathcal B}
  \mathbb P_b(\hat b\ne b)
  \ge
  1-\frac{1}{16}-\frac{\log 2}{\log|\mathcal B|}
  \ge\frac7{16}.
\label{eq:fano-error-probability}
\end{equation}
Combining the decoding bound above with
\eqref{eq:weighted-risk-dominates-common-metric} gives, under hypothesis \(b\),
\[
  \int_{\tau_0}^{2\tau_0}
  \|\tilde s_n(\cdot,\tau)-s_{b,\tau}\|_{L^2(g_{b,\tau})}^2\,\rmd\tau
  \ge
  \frac{c_Qc_{\mathrm{int}}}{16}\eta^2
  \mathbf 1_{\{\hat b\ne b\}}.
\]
Taking \(\mathbb E_b\), then the supremum over \(b\), and using
\eqref{eq:fano-error-probability} yields
\begin{equation}
\begin{aligned}
  \sup_{b\in\mathcal B}
  \mathbb E_b
  \int_{\tau_0}^{2\tau_0}
  \|\tilde s_n(\cdot,\tau)-s_{b,\tau}\|_{L^2(g_{b,\tau})}^2\,\rmd\tau
  \ge
  \frac{7c_Qc_{\mathrm{int}}}{256}\eta^2.
\end{aligned}
\label{eq:heat-integrated-fano-lower-bound}
\end{equation}

The lower half of \eqref{eq:grid-cardinality} and
\(\rho=K_0\sqrt{\tau_0}\) give
\(m\ge2^{-d}(\ell/K_0)^d\tau_0^{-d/2}\).
Consequently, using the elementary inequality
\(1\wedge(uv)\ge(1\wedge u)(1\wedge v)\) for \(u,v>0\),
\[
  \eta^2
  =\kappa\left(1\wedge\frac mn\right)
  \ge
  \kappa
  \min\left\{1,2^{-d}\left(\frac{\ell}{K_0}\right)^d\right\}
  \left(1\wedge\frac{1}{n\tau_0^{d/2}}\right).
\]

Now take an arbitrary OU score estimator \(\hat s_n\) and transform it into
\(\tilde s_n\) by \eqref{eq:transformed-score-estimator}.  Since every \(f_b\),
\(b\in\mathcal B\), belongs to \(\mathcal P_D\), restricting first the supremum to this
finite family and then the time integral to \([t_0,t_1]\) gives, by the exact identity
\eqref{eq:exact-ou-heat-loss-identity},
\begin{align*}
&\sup_{p_0\in\mathcal P_D}
\mathbb E_{p_0^{\otimes n}\otimes\lambda}
\int_{t_0}^{T}
\|\hat s_n(\cdot,t)-\nabla\log p_t\|_{L^2(p_t)}^2\rmd t
\\
&\quad\ge
\sup_{b\in\mathcal B}
\mathbb E_b
\int_{\tau_0}^{2\tau_0}
\|\tilde s_n(\cdot,\tau)-s_{b,\tau}\|_{L^2(g_{b,\tau})}^2\rmd\tau
\\
&\quad\ge
 c_1
 \left(1\wedge\frac{1}{n\tau_0^{d/2}}\right).
\end{align*}
Combining this display with \eqref{eq:robust-dominates-nominal}, dividing by \(T-t_0\), and then
taking the infimum over
\(\hat s_n\) proves \eqref{eq:integrated-statistical-lower-bound}.
In the sample-rich regime \(n\tau_0^{d/2}\ge1\), the minimum equals its second argument.
\end{proofof}

\subsection{Fixed-Time Density Separation at the Heat Scale}
\label{app:fixed-time-kl-separation}

We now reuse the family from
Subsection~\ref{app:hard-family-construction} for the fixed-time KL problem.
Unlike the integrated score lower bound, the present argument does not require
localization to the interior cube \(I\) or persistence over an interval of heat
times.  It suffices to show that the \(L^2\) distance between the smoothed
densities remains visible at the single time \(\tau_0\).

For probability measures \(\mu\) and \(\nu\), define the squared Hellinger
distance by
\begin{equation}
  \mathsf h^2(\mu,\nu)
  :=\int
  \left(
    \sqrt{\frac{\rmd\mu}{\rmd\lambda}}
    -\sqrt{\frac{\rmd\nu}{\rmd\lambda}}
  \right)^2\rmd\lambda,
  \qquad \lambda:=\mu+\nu.
\label{eq:hellinger-definition-kl}
\end{equation}
This definition is independent of the dominating measure, is invariant under
a common bimeasurable bijection, and satisfies
\begin{equation}
  \mathrm{KL}(\mu\|\nu)\ge \mathsf h^2(\mu,\nu).
\label{eq:kl-dominates-hellinger}
\end{equation}
We identify a probability density with its induced measure whenever no
confusion can arise.
To verify \eqref{eq:kl-dominates-hellinger}, the inequality is immediate if
\(\mu\not\ll\nu\).  Otherwise, with
\[
  A(\mu,\nu)
  :=\int
  \sqrt{\frac{\rmd\mu}{\rmd\lambda}
         \frac{\rmd\nu}{\rmd\lambda}}\,\rmd\lambda,
\]
Jensen's inequality gives
\(\mathrm{KL}(\mu\|\nu)\ge-2\log A(\mu,\nu)\).  Now use
\(-\log x\ge1-x\) and
\(\mathsf h^2(\mu,\nu)=2[1-A(\mu,\nu)]\).

Recall the notation \(g_{b,\tau}:=\mathsf H_\tau f_b\) from
Subsection~\ref{app:uniform-score-separation}.

\begin{lemma}[Fixed-time \(L^2\) persistence]
\label{lem:fixed-time-density-persistence}
Let \(K_0\) be defined by
\eqref{eq:persistence-scale-constant}, set
\(\rho=K_0\sqrt{\tau_0}\), and assume that \(\rho\le\ell/4\) and
\(0<\eta\le1/4\).  Then, for every
\(b,b'\in\{0,1\}^m\),
\begin{equation}
  \|g_{b,\tau_0}-g_{b',\tau_0}\|_2^2
  \ge
  \frac{c_Q^2A_0}{2}\,
  \eta^2\rho^d d_{\mathrm{Ham}}(b,b').
\label{eq:fixed-time-L2-separation}
\end{equation}
\end{lemma}

\begin{proof}
Set \(r:=f_b-f_{b'}\) and \(H:=d_{\mathrm{Ham}}(b,b')\).  The disjoint
bump supports give
\begin{equation}
  \|r\|_2^2=c_Q^2\eta^2H\rho^dA_0,
  \qquad
  \|\nabla r\|_2^2=c_Q^2\eta^2H\rho^{d-2}A_1.
\label{eq:density-energy-at-zero}
\end{equation}
Since \(r\in C_c^\infty(\mathbb R^d)\), the function
\(F(s):=\|\mathsf H_sr\|_2^2\) satisfies
\[
  F'(s)
  =2\left\langle
    \mathsf H_sr,\frac12\Delta\mathsf H_sr
  \right\rangle_{L^2}
  =-\|\nabla\mathsf H_sr\|_2^2.
\]
The heat semigroup commutes with derivatives and is an \(L^2\) contraction,
so
  \begin{equation}
  \|\mathsf H_{\tau_0}r\|_2^2
  \ge \|r\|_2^2-\tau_0\|\nabla r\|_2^2
  =c_Q^2\eta^2H\rho^d\left(A_0-\frac{A_1}{K_0^2}\right).
\label{eq:density-energy-dissipation}
  \end{equation}
Integration by parts and Cauchy--Schwarz give
\(A_1=-\int w\Delta w\le\sqrt{A_0A_2}\).  Moreover,
\eqref{eq:persistence-scale-constant} gives
\[
  \frac{A_1}{K_0^2}
  \le\frac{A_1^2}{8A_2}
  \le\frac{A_0}{8}.
\]
Thus the parenthesis in \eqref{eq:density-energy-dissipation} is at least
\(7A_0/8\), which implies \eqref{eq:fixed-time-L2-separation}.
\end{proof}

\begin{corollary}[Fixed-time Hellinger packing]
\label{cor:fixed-time-hellinger-packing}
Suppose that \(\rho\le\ell/4\) and
\(d_{\mathrm{Ham}}(b,b')\ge m/8\).  Then
\begin{equation}
  \mathsf h^2(g_{b,\tau_0},g_{b',\tau_0})
  \ge c_{\mathrm H,d,D}\eta^2,
  \qquad
  c_{\mathrm H,d,D}:=\frac{c_QA_0\ell^d}{80\,2^d}>0.
\label{eq:fixed-time-hellinger-separation}
\end{equation}
The same bound holds for the corresponding OU marginals
\(p_{b,t_0}:=P_{t_0}(f_b(x)\,\rmd x)\).
\end{corollary}

\begin{proof}
The \(L^\infty\) contraction of the heat semigroup and
\(0\le f_b\le5c_Q/4\) give
\[
  0\le g_{b,\tau_0}(x)\le\frac{5c_Q}{4}
  \qquad (x\in\mathbb R^d).
\]
Therefore,
  \[
  \mathsf h^2(g_{b,\tau_0},g_{b',\tau_0})
  =\int
  \frac{(g_{b,\tau_0}-g_{b',\tau_0})^2}
       {(\sqrt{g_{b,\tau_0}}+\sqrt{g_{b',\tau_0}})^2}\,\rmd x
  \ge\frac{1}{5c_Q}
  \|g_{b,\tau_0}-g_{b',\tau_0}\|_2^2.
  \]
Lemma~\ref{lem:fixed-time-density-persistence}, the Hamming separation, and
\eqref{eq:grid-cardinality} yield
  \[
  \mathsf h^2(g_{b,\tau_0},g_{b',\tau_0})
  \ge\frac{c_QA_0}{10}\eta^2\rho^d
  d_{\mathrm{Ham}}(b,b')
  \ge\frac{c_QA_0}{80}\eta^2m\rho^d
  \ge\frac{c_QA_0\ell^d}{80\,2^d}\eta^2.
  \]
Finally, with \(a_{t_0}=e^{-t_0/2}\), let
\(S_{a_{t_0}}(x):=a_{t_0}x\). Then
\eqref{eq:ou-heat-density-score-scaling} gives
\(p_{b,t_0}=(S_{a_{t_0}})_\#(g_{b,\tau_0}(x)\,\rmd x)\).
The invariance of Hellinger distance under this common invertible dilation
proves the last claim.
\end{proof}

\subsection{Fixed-Time KL Lower Bound and Nominal Minimax Rate}
\label{app:nominal-kl-minimax-proof}

We first record the stronger lower bound that includes the sample-poor
truncation.

With \(c_{\mathrm H,d,D}\) defined in
Corollary~\ref{cor:fixed-time-hellinger-packing}, \(\kappa\) as in
Subsection~\ref{app:fano-proof-statistical-lower}, and \(\ell,K_0\) as in
Subsection~\ref{app:hard-family-construction}, set
\begin{equation}
  c_{d,D}^{\mathrm{KL}}
  :=\frac{7c_{\mathrm H,d,D}\kappa}{64}
  \min\left\{
    1,\,
    2^{-d}\left(\frac{\ell}{K_0}\right)^d
  \right\}.
\label{eq:explicit-fixed-time-kl-lower-constant}
\end{equation}
This constant is strictly positive and depends only on \(d,D\).

\begin{proposition}[Fixed-time statistical KL lower bound]
\label{prop:fixed-time-kl-statistical-lower}
Let \(c_0\) be defined by
\eqref{eq:explicit-statistical-lower-bound-constants}, and let
\(c_{d,D}^{\mathrm{KL}}\) be defined by
\eqref{eq:explicit-fixed-time-kl-lower-constant}.  Then, for every
\(0<t_0\le c_0\) and \(n\ge1\),
\begin{equation}
  \inf_{\hat\mu_n}
  \sup_{p_0\in\mathcal P_D}
  \mathbb E\,
  \mathrm{KL}(p_{t_0}\|\hat\mu_n)
  \ge
  c_{d,D}^{\mathrm{KL}}
  \left(1\wedge\frac{1}{n\tau_0^{d/2}}\right),
\label{eq:fixed-time-kl-statistical-lower}
\end{equation}
where the infimum is over all measurable, possibly randomized,
probability-measure estimators based on \(n\) i.i.d. observations from
\(p_0\); the expectation includes both sample and estimator randomness.
\end{proposition}

\begin{proof}
Assume \(0<t_0\le c_0\).  By
\eqref{eq:explicit-statistical-lower-bound-constants},
\(\rho=K_0\sqrt{\tau_0}\le\ell/4\) and \(m\ge16\).  Let
\(\mathcal B\subset\{0,1\}^m\) be the Gilbert--Varshamov packing in
\eqref{eq:gilbert-varshamov}, and choose
\begin{equation}
  \eta^2:=\kappa\left(1\wedge\frac{m}{n}\right),
\label{eq:kl-amplitude-choice}
\end{equation}
  with \(\kappa\) as above.  The calculation leading to
\eqref{eq:fano-kl-calibration} applies verbatim and gives
\begin{equation}
  \max_{b\in\mathcal B}
  \mathrm{KL}(f_b^{\otimes n}\|f_\circ^{\otimes n})
  \le\frac1{16}\log|\mathcal B|.
\label{eq:kl-fano-calibration-reused}
\end{equation}

Consider an arbitrary measurable, possibly randomized, probability-measure
estimator \(\hat\mu_n\) based on the sample.  Represent its internal
randomization by a seed \(U\sim\lambda\), independent of the sample, and define
\[
  \mathbb P_b:=(f_b(x)\,\rmd x)^{\otimes n}\otimes\lambda,
  \qquad b\in\mathcal B.
\]
Let \(\mathbb E_b\) denote expectation under \(\mathbb P_b\).
Decode among the OU marginals by
\[
  \hat b\in\arg\min_{c\in\mathcal B}
  \mathsf h(\hat\mu_n,p_{c,t_0}),
\]
using a fixed deterministic rule to break ties. Equip
\(\mathcal P(\mathbb R^d)\) with the Borel $\sigma$-field induced by weak
convergence. Since Hellinger distance is lower semicontinuous, each map
\(\mu\mapsto\mathsf h(\mu,p_{c,t_0})\) is Borel measurable. Because
\(\mathcal B\) is finite, the nearest-neighbour decoder with deterministic
tie-breaking is measurable. Corollary~
\ref{cor:fixed-time-hellinger-packing} gives, for distinct packing elements,
\[
  \mathsf h^2(p_{b,t_0},p_{b',t_0})
  \ge c_{\mathrm H,d,D}\eta^2.
\]
If \(\hat b\ne b\), the triangle inequality and the nearest-neighbour
property imply
\[
  \sqrt{c_{\mathrm H,d,D}}\,\eta
\le\mathsf h(p_{b,t_0},p_{\hat b,t_0})
\le2\mathsf h(p_{b,t_0},\hat\mu_n).
\]
Consequently, \eqref{eq:kl-dominates-hellinger} yields
\begin{equation}
\mathrm{KL}(p_{b,t_0}\|\hat\mu_n)
  \ge\frac{c_{\mathrm H,d,D}}{4}\eta^2
\mathbf 1_{\{\hat b\ne b\}}.
\label{eq:kl-loss-controls-decoding}
\end{equation}
The Fano calculation in \eqref{eq:fano-error-probability} applies without
change because \eqref{eq:kl-fano-calibration-reused} is the same information
calibration.  Taking expectations in
\eqref{eq:kl-loss-controls-decoding} therefore gives
\begin{equation}
  \sup_{b\in\mathcal B}
  \mathbb E_b\,
  \mathrm{KL}(p_{b,t_0}\|\hat\mu_n)
  \ge\frac{7c_{\mathrm H,d,D}}{64}\eta^2.
\label{eq:kl-risk-after-fano}
\end{equation}

By \eqref{eq:grid-cardinality}, \(\rho=K_0\sqrt{\tau_0}\), and
\(1\wedge(uv)\ge(1\wedge u)(1\wedge v)\) for \(u,v>0\),
\begin{equation}
  \eta^2
  \ge\kappa
  \min\left\{1,2^{-d}\left(\frac{\ell}{K_0}\right)^d\right\}
  \left(1\wedge\frac{1}{n\tau_0^{d/2}}\right).
\label{eq:kl-amplitude-lower}
\end{equation}
Every \(f_b\) is the density of a member of \(\mathcal P_D\).  Combining
\eqref{eq:kl-risk-after-fano} and \eqref{eq:kl-amplitude-lower}, restricting
the supremum to this finite family, and then taking the infimum over
\(\hat\mu_n\), with \(c_{d,D}^{\mathrm{KL}}\) as in
\eqref{eq:explicit-fixed-time-kl-lower-constant}, proves
\eqref{eq:fixed-time-kl-statistical-lower}.
\end{proof}

\begin{theorem}[Nominal positive-time KL minimax rate]
\label{thm:nominal-kl-minimax}
Fix $d\ge1$ and $D>0$. Let $c_0$ and $c_{d,D}^{\mathrm{KL}}$ be defined by
\eqref{eq:explicit-statistical-lower-bound-constants} and
\eqref{eq:explicit-fixed-time-kl-lower-constant}, respectively. If
$0<t_0\le c_0$ and $n\tau_0^{d/2}\ge1$, then
\begin{equation}
  \frac{c_{d,D}^{\mathrm{KL}}}{n\tau_0^{d/2}}
  \le
  \mathfrak K_{n,0}^{\star}(D;t_0)
  \le
  \frac{C_{d,D}}{2n\tau_0^{d/2}}.
  \label{eq:nominal-kl-minimax-rate}
\end{equation}
\end{theorem}

\begin{proof}
The lower bound follows from
Proposition~\ref{prop:fixed-time-kl-statistical-lower}; under
\(n\tau_0^{d/2}\ge1\), the minimum in
\eqref{eq:fixed-time-kl-statistical-lower} equals its second argument.  The
upper bound follows from \eqref{eq:self-consistent-kl-sampling-bound} with
\(\varepsilon=0\), followed by \(n+1\ge n\).  This proves
\eqref{eq:nominal-kl-minimax-rate}.
\end{proof}

\section{Proofs for Shift Lower Bounds and Minimax Optimality}
\label{app:intrinsic-lower-bound}

This appendix proves the shift lower bounds and completes the robust score
and KL minimax results. Subsections~\ref{app:intrinsic-hard-family}--
\ref{app:intrinsic-one-dimensional-lower-bound} develop the one-dimensional
oscillatory construction through Wasserstein admissibility, two-point score
separation, and persistence under OU smoothing.
Subsection~\ref{app:score-lower-and-minimax} lifts the construction to
arbitrary fixed dimension and proves
Theorems~\ref{thm:fixed-dimensional-lower-main}
and~\ref{thm:finite-sample-minimax-optimality}. Finally,
Subsection~\ref{app:robust-kl-lower-and-minimax} proves the positive-time KL
shift lower bound by a two-point Gaussian argument and combines it with the
nominal KL lower bound from Appendix~\ref{app:statistical-minimax-lower-bound}
to prove Theorem~\ref{thm:robust-kl-minimax}.

\subsection{One-Dimensional Construction and Lower-Bound Statement}
\label{app:intrinsic-hard-family}

Throughout the one-dimensional argument, write \(p_0:=p_0^{(1)}\).

Consider the one-dimensional Ornstein--Uhlenbeck process
\(\rmd X_t=-\frac12X_t\,\rmd t+\rmd W_t\). Its explicit solution is
\[
    X_t=a_tX_0+\sigma_tZ,
    \qquad
    a_t=e^{-t/2},
    \qquad
    \sigma_t^2=1-e^{-t},
    \qquad
    Z\sim 
\mathcal{N}(0,1).
\]
We use below the elementary bound \(\sigma_t^2\le t\).

Let the reference density be \(p_0(x):=\frac12\mathbf 1_{[-1,1]}(x)\).
For an integer \(k\ge1\), set
$
    \lambda_k:=\pi k,
   $ and
   $
    h_k(x):=\cos(\lambda_kx).
$
Since
\[
    \int_{-1}^{1}h_k(x)\,\rmd x
    =
    \int_{-1}^{1}\cos(\pi kx)\,\rmd x
    =
    \frac{2\sin(\pi k)}{\pi k}
    =
    0,
\]
the following perturbations preserve total mass. For \(0<\alpha\le1/2\),
define
\[
    q_0^\pm(x)
    :=
    \frac12\bigl(1\pm\alpha h_k(x)\bigr)\mathbf 1_{[-1,1]}(x).
\]
Because $|h_k|\le1$ and $\alpha\le1/2$, these are probability densities, and
\(\operatorname{supp}(q_0^\pm)=\operatorname{supp}(p_0)=[-1,1]\).
Thus the alternatives are compactly supported.

For \(t>0\), let
\[
    K_t(x,y)
    :=
    \frac1{\sqrt{2\pi}\sigma_t}
    \exp\left(
        -\frac{(x-a_ty)^2}{2\sigma_t^2}
    \right)
\]
denote the OU transition density. The smoothed reference density is
\[
    p_t(x)
    =
    \frac12\int_{-1}^{1}K_t(x,y)\,\rmd y.
\]
Similarly,
$
    q_t^\pm(x)
    =
    p_t(x)\pm\alpha r_t(x),
$
where
$
    r_t(x)
    :=
    \frac12\int_{-1}^{1}h_k(y)K_t(x,y)\,\rmd y .
$
Define \(\eta_t(x):=r_t(x)/p_t(x)\). Then
\[
    q_t^\pm(x)
    =
    p_t(x)\bigl(1\pm\alpha\eta_t(x)\bigr).
\]
Under the joint law induced by \(X_0\sim p_0\) and the OU transition,
\(\eta_t(x)=\mathbb E[h_k(X_0)\mid X_t=x]\), so
\(|\eta_t(x)|\le1\) and \(1\pm\alpha\eta_t(x)\ge1-\alpha\ge1/2\).
Thus $q_t^\pm$ are strictly positive for every $t>0$, and their scores are
well defined on all of $\mathbb R$.

The preceding construction underlies the following one-dimensional lower
bound, whose proof is given after the auxiliary estimates.

\begin{proposition}
\label{prop:ou-wasserstein-lower-bound}
Let \(p_0(x)=\frac12\mathbf 1_{[-1,1]}(x)\),
and let \(P_t\) denote the semigroup of the one-dimensional
Ornstein--Uhlenbeck process
\(\rmd X_t=-\frac12X_t\,\rmd t+\rmd W_t\).
Fix \(0<t_0\le 1/(16\pi^2)\), \(2t_0\le T\), and
\(0<\varepsilon\le\sqrt{t_0}\). For each
\(q_0\in\mathcal Q_\varepsilon(p_0)\), set
$
    q_t=P_tq_0.
$
Then
\begin{equation}
\label{eq:one-dimensional-wasserstein-lower-bound}
\inf_{\hat s}
\sup_{q_0\in\mathcal Q_\varepsilon(p_0)}
\mathbb E\left[
    \mathcal L_{q_0}(\hat s)
\right]
\ge
\frac{e^{-4/9}}{2592}
\frac{\varepsilon^2}{t_0(T-t_0)}.
\end{equation}
The infimum is over all possibly random score
estimators whose realizations are measurable maps from
\(\mathbb R\times[t_0,T]\) to \(\mathbb R\), and the expectation is over all
randomness used to construct \(\hat s\).
\end{proposition}

\subsection{Auxiliary Estimates}
\label{app:intrinsic-aux}

We next establish the three estimates used in the proof of
Proposition~\ref{prop:ou-wasserstein-lower-bound}. The Wasserstein estimate
verifies admissibility of the alternatives, the two-point inequality reduces
the score risk to the energy of \(\eta_t\), and the persistence estimate
controls this energy over the early-time window.

\subsubsection{Wasserstein Size of the Perturbation}
\label{app:intrinsic-wasserstein-size}

\begin{lemma}%
\label{lem:oscillatory-wasserstein-size}
For every \(k\ge1\) and every \(0<\alpha\le1/2\),
\[
    \wass_2(p_0,q_0^\pm)
    \le
    \frac{\alpha}{\lambda_k}.
\]
\end{lemma}

\begin{proof}
We prove the claim for $q_0^+$, and the proof for $q_0^-$ is identical after
reversing the sign.
For \(s\in[0,1]\), define the interpolating density and its flux by
\begin{align*}
    \rho_s(x)
    &:=\frac12\bigl(1+s\alpha\cos(\lambda_kx)\bigr)
    \mathbf 1_{[-1,1]}(x),\\
    \rho_s(x)v_s(x)
    &:=-\frac{\alpha}{2\lambda_k}
    \sin(\lambda_kx)\mathbf 1_{[-1,1]}(x).
\end{align*}
Since \(\lambda_k=\pi k\), \(\sin(\lambda_k)=\sin(\pi k)=0\). Thus the flux
vanishes at the boundary points \(\pm1\). Moreover,
\begin{align*}
    \partial_s\rho_s(x)
    &=\frac{\alpha}{2}\cos(\lambda_kx),
    &
    \partial_x(\rho_sv_s)(x)
    &=-\frac{\alpha}{2}\cos(\lambda_kx).
\end{align*}
Thus \(\partial_s\rho_s+\partial_x(\rho_sv_s)=0\)
on \([-1,1]\). After extending both \(\rho_s\) and
\(\rho_sv_s\) by zero outside \([-1,1]\), the same continuity equation holds
on \(\mathbb R\) in the distributional sense because the flux vanishes at
\(\pm1\). By the Benamou--Brenier formula,
\[
\begin{aligned}
\wass_2^2(p_0,q_0^+)
\le
\int_0^1\int_{-1}^{1}|v_s(x)|^2\rho_s(x)\,\rmd x\,\rmd s
=
\int_0^1\int_{-1}^{1}
\frac{\alpha^2}{4\lambda_k^2}
\frac{\sin^2(\lambda_kx)}{\rho_s(x)}\,\rmd x\,\rmd s .
\end{aligned}
\]
Since \(\alpha\le1/2\),
\(\rho_s(x)\ge(1-\alpha)/2\ge1/4\) for all \(s\in[0,1]\) and
\(x\in[-1,1]\). Therefore
\[
    \wass_2^2(p_0,q_0^+)
    \le
    \frac{\alpha^2}{\lambda_k^2}
    \int_{-1}^{1}\sin^2(\lambda_kx)\,\rmd x
    =
    \frac{\alpha^2}{\lambda_k^2}.
\]
Taking square roots proves the claim.
\end{proof}

\subsubsection{Two-Point Score Separation}
\label{app:intrinsic-two-point}

\begin{lemma}%
\label{lem:two-point-score-separation}
Assume \(0<\alpha\le1/2\). For every measurable
\(s:\mathbb R\times[t_0,T]\to\mathbb R\) and every \(t\in[t_0,T]\), define
\[
    R_\xi(t,s)
    :=
    \mathbb E_{q_t^\xi}
    \left[
        |s(X,t)-\partial_x\log q_t^\xi(X)|^2
    \right],
    \qquad \xi\in\{+,-\}.
\]
The two risks satisfy
\nopagebreak[4]
\[
    R_+(t,s)+R_-(t,s)
    \ge
    2\alpha^2
    \int_{\mathbb R}
        |\partial_x\eta_t(x)|^2p_t(x)\,\rmd x .
\]
\end{lemma}

\begin{proof}
For any \(\lambda,\mu>0\) and \(a,b,z\in\mathbb R\),
\(\lambda|z-a|^2+\mu|z-b|^2
\ge\lambda\mu|a-b|^2/(\lambda+\mu)\). Apply this pointwise with
\(\lambda=q_t^+(x)\), \(\mu=q_t^-(x)\),
\(a=\partial_x\log q_t^+(x)\),
\(b=\partial_x\log q_t^-(x)\), and \(z=s(x,t)\).
This gives
\[
\begin{aligned}
R_+(t,s)+R_-(t,s)
&\ge
\int_{\mathbb R}
\frac{q_t^+(x)q_t^-(x)}
     {q_t^+(x)+q_t^-(x)}
\left|
    \partial_x\log q_t^+(x)
    -
    \partial_x\log q_t^-(x)
\right|^2\,\rmd x .
\end{aligned}
\]

Using \(q_t^\pm=p_t(1\pm\alpha\eta_t)\), we have
$
    \partial_x\log q_t^\pm
    =
    \partial_x\log p_t
    +
    \partial_x\log(1\pm\alpha\eta_t).
$
Moreover,
\begin{align*}
    \partial_x\log q_t^+
    -
    \partial_x\log q_t^-
    &=
    \frac{2\alpha\,\partial_x\eta_t}
         {1-\alpha^2\eta_t^2},\\
    \frac{q_t^+q_t^-}{q_t^++q_t^-}
    &=
    \frac12p_t(1-\alpha^2\eta_t^2).
\end{align*}
Substitution yields
\[
\begin{aligned}
R_+(t,s)+R_-(t,s)
&\ge
2\alpha^2
\int_{\mathbb R}
\frac{|\partial_x\eta_t(x)|^2}
     {1-\alpha^2\eta_t(x)^2}
p_t(x)\,\rmd x .
\end{aligned}
\]
Since \(|\eta_t|\le1\) and \(\alpha\le1/2\), we have
\((1-\alpha^2\eta_t^2)^{-1}\ge1\), which proves the result.
\end{proof}

\subsubsection{Persistence of One Oscillatory Mode}
\label{app:intrinsic-mode-persistence}

We next show that an oscillatory mode remains visible during the early-time
window whenever its frequency is below the critical scale $t_0^{-1/2}$.

Set \(u:=x/a_t\) and
\(\tau_t:=\sigma_t^2/a_t^2=e^t-1\), so
\(\sqrt{\tau_t}=\sigma_t/a_t\). After the change of variables \(x=a_tu\), define
\begin{align*}
p_t(x)&=\frac1{2a_t}A_t(u),
&A_t(u)&:=\int_{-1}^{1}\frac{1}{\sqrt{2\pi\tau_t}}
\exp\left(-\frac{(u-y)^2}{2\tau_t}\right)\,\rmd y,\\
r_t(x)&=\frac1{2a_t}B_t(u),
&B_t(u)&:=\int_{-1}^{1}\cos(\lambda_ky)\frac{1}{\sqrt{2\pi\tau_t}}
\exp\left(-\frac{(u-y)^2}{2\tau_t}\right)\,\rmd y,\\
\eta_t(x)&=\frac{B_t(u)}{A_t(u)}.
\end{align*}
Primes below denote derivatives with respect to \(u\).

\begin{lemma}%
\label{lem:ou-mode-persistence}
Assume \(k\ge1\) and that the small-time and frequency--time conditions
\[
    0<t\le\frac1{64},
    \qquad
    \lambda_k^2t\le1
\]
hold. Then
\begin{equation}
\label{eq:ou-mode-persistence-general}
    \int_{\mathbb R}
        |\partial_x\eta_t(x)|^2p_t(x)\,\rmd x
    \ge
    \frac1{32}\lambda_k^2
    \exp\left(-2\lambda_k^2\sigma_t^2\right).
\end{equation}
\end{lemma}

\begin{proof}
We work on the central interval \(|u|\le1/2\).

\paragraph{Step 1: Bounds for the truncated Gaussian mass.}
For \(0<t\le1/64\), the elementary inequality
\(e^t\le(1-t)^{-1}\) gives
\[
    \tau_t
    =
    e^t-1
    \le
    \frac{t}{1-t}
    \le
    \frac1{63}
    <
    \frac1{8\log 4}.
\]
Thus \(\tau_t\le1\). For \(|u|\le1/2\), the Gaussian mass omitted from \(A_t(u)\)
outside \([-1,1]\) lies at distance at least \(1/(2\sqrt{\tau_t})\) in standard
normal units.  Hence, the one-dimensional Gaussian tail bound gives
\begin{equation}
\label{eq:A-lower-upper}
    \frac12\le A_t(u)\le1,
    \qquad |u|\le1/2,
\end{equation}
and
\begin{equation}
\label{eq:A-prime-small}
    |A_t'(u)|
    \le
    \frac{2}{\sqrt{2\pi}}\tau_t^{-1/2}
    \exp\left(-\frac{1}{8\tau_t}\right),
    \qquad |u|\le1/2.
\end{equation}

\paragraph{Step 2: Comparison with the full Gaussian convolution.} 
Let
\[
    \mathcal B_t(u)
    :=
    \int_{\mathbb R}
    \cos(\lambda_ky)
    \frac1{\sqrt{2\pi\tau_t}}
    \exp\left(
        -\frac{(u-y)^2}{2\tau_t}
    \right)\,\rmd y .
\]
The Gaussian convolution of a cosine can be computed explicitly, as
$
    \mathcal B_t(u)
    =
    e^{-\lambda_k^2\tau_t/2}\cos(\lambda_ku),
$
and hence \(\mathcal B_t'(u)
=-\lambda_k e^{-\lambda_k^2\tau_t/2}\sin(\lambda_ku)\).
The difference between \(B_t\) and \(\mathcal B_t\) comes only from the
Gaussian tails outside \([-1,1]\). For \(|u|\le1/2\), those tails are at
distance at least \(1/(2\sqrt{\tau_t})\) in standard normal units.  Using
\(|\cos(\lambda_k y)|\le1\), differentiating the Gaussian kernel, and
\(\tau_t\le1\), the same boundary-tail estimate gives
\begin{equation}
\label{eq:B-tail-error}
    |B_t'(u)-\mathcal B_t'(u)|
    \le
    \frac{2}{\sqrt{2\pi}}\tau_t^{-1}
    \exp\left(-\frac{1}{8\tau_t}\right),
    \qquad |u|\le1/2 .
\end{equation}

\paragraph{Step 3: Lower bound for the derivative of \(B_t/A_t\).}
Since
\[
\partial_u\left(\frac{B_t}{A_t}\right)
=
\frac{B_t'}{A_t}
-
\frac{B_tA_t'}{A_t^2}
= \frac{\mathcal B_t'}{A_t}+
\frac{B_t'-\mathcal B_t'}{A_t}-
\frac{B_tA_t'}{A_t^2},
\]
we obtain
\[
\begin{aligned}
\left|
\partial_u\left(\frac{B_t}{A_t}\right)
\right|
&\ge
\frac{|\mathcal B_t'|}{A_t}-
\frac{|B_t'-\mathcal B_t'|}{A_t}-
\frac{|B_t||A_t'|}{A_t^2}.
\end{aligned}
\]
Because $A_t\le1$, we have \(|\mathcal B_t'|/A_t\ge|\mathcal B_t'|\).
Moreover, $|B_t|\le A_t$, and \eqref{eq:A-lower-upper},
\eqref{eq:A-prime-small}, and \eqref{eq:B-tail-error} imply
\[
\frac{|B_t'-\mathcal B_t'|}{A_t}
+
\frac{|B_t||A_t'|}{A_t^2}
\le
4\tau_t^{-1}
\exp\left(-\frac1{8\tau_t}\right).
\]
Now define \(\delta_t:=4\tau_t^{-1}\exp(-1/(8\tau_t))\).
For \(|u|\le1/2\), it holds that
\[
\left|
    \partial_u\left(\frac{B_t}{A_t}\right)(u)
\right|
\ge
\left(
\lambda_k e^{-\lambda_k^2\tau_t/2}
|\sin(\lambda_k u)|
-
\delta_t
\right)_+ .
\]
For \(r,s\ge0\), the elementary inequality
\((r-s)_+^2\ge \frac12 r^2-s^2\) holds. Hence
\[
\begin{aligned}
\int_{-1/2}^{1/2}
\left|
    \partial_u\left(\frac{B_t}{A_t}\right)(u)
\right|^2\,\rmd u
&\ge
\frac12\lambda_k^2 e^{-\lambda_k^2\tau_t}
\int_{-1/2}^{1/2}\sin^2(\lambda_k u)\,\rmd u
-
\delta_t^2 .
\end{aligned}
\]
Since \(\lambda_k=\pi k\),
\(\int_{-1/2}^{1/2}\sin^2(\lambda_ku)\,\rmd u=1/2\).
Moreover, since \(0<t\le1/64<\log2\) and \(\lambda_k^2t\le1\), it follows that
\[
    \lambda_k^2\tau_t
    =
    \lambda_k^2(e^t-1)
    \le
    2\lambda_k^2t
    \le
    2 .
\]
Since $k\ge1$,
\[
    \lambda_k^2 e^{-\lambda_k^2\tau_t}
    \ge
    \lambda_k^2e^{-2}
    \ge
    \pi^2e^{-2}.
\]
On the other hand, set \(g(y):=16y^{-2}\exp(-1/(4y))\). This function is
increasing on \((0,1/8]\), since
\[
\frac{g'(y)}{g(y)}=\frac{1-8y}{4y^2}\ge0.
\]
As \(\tau_t\le1/63<1/8\), we obtain the explicit bound
\[
\begin{aligned}
    \delta_t^2
    &=
    g(\tau_t)
    \le
    g(1/63)
    =
    16\cdot63^2e^{-63/4}
    <
    \frac{\pi^2}{8}e^{-2}.
\end{aligned}
\]
Consequently,
\[
    \delta_t^2
    \le
    \frac{\pi^2}{8}e^{-2}
    \le
    \frac18
    \lambda_k^2 e^{-\lambda_k^2\tau_t}
\]
uniformly over all \(k\ge1\) satisfying \(\lambda_k^2t\le1\). Therefore,
\[
\int_{-1/2}^{1/2}
\left|
    \partial_u\left(\frac{B_t}{A_t}\right)(u)
\right|^2\,\rmd u
\ge
\frac18\lambda_k^2e^{-\lambda_k^2\tau_t}.
\]

\paragraph{Step 4: Return to the original spatial variable.}
Since \(x=a_tu\), we have
\[
    \partial_x\eta_t(a_tu)
    =a_t^{-1}\partial_u\left(\frac{B_t}{A_t}\right)(u),
    \qquad
    p_t(a_tu)a_t\,\rmd u
    =\frac12A_t(u)\,\rmd u .
\]
Using $A_t(u)\ge1/2$ and $a_t^{-2}\ge1$, we get
\[
\begin{aligned}
\int_{\mathbb R}
    |\partial_x\eta_t(x)|^2p_t(x)\,\rmd x
&\ge
\int_{-1/2}^{1/2}
    |\partial_x\eta_t(a_tu)|^2p_t(a_tu)a_t\,\rmd u\\
&\ge
\frac14
\int_{-1/2}^{1/2}
\left|
    \partial_u\left(\frac{B_t}{A_t}\right)(u)
\right|^2\,\rmd u
\ge
\frac1{32}\lambda_k^2e^{-\lambda_k^2\tau_t}.
\end{aligned}
\]
Finally, since \(0<t\le1/64\),
\(\tau_t=e^t\sigma_t^2\le2\sigma_t^2\).
Thus,
\[
\int_{\mathbb R}
|\partial_x\eta_t(x)|^2p_t(x)\,\rmd x
\ge
    \frac1{32}\lambda_k^2e^{-\lambda_k^2\tau_t}
    \ge
    \frac1{32}\lambda_k^2e^{-2\lambda_k^2\sigma_t^2},
\]
which proves the claimed lower bound \eqref{eq:ou-mode-persistence-general}.
\end{proof}

\subsection{Proof of Proposition~\ref{prop:ou-wasserstein-lower-bound}}
\label{app:intrinsic-one-dimensional-lower-bound}

\begin{proof}
Since \(\pi^2>8\), the assumption \(t_0\le1/(16\pi^2)\) implies
$
t_0<\frac1{128}.
$
We select the largest integer frequency that is uniformly compatible with the
amplitude constraint under \(\varepsilon\le\sqrt{t_0}\). Let
\[
    \kappa_0:=\frac1{2\pi\sqrt{t_0}},
    \qquad
    k:=\lfloor\kappa_0\rfloor,
    \qquad
    \lambda_k=\pi k .
\]
The assumption on \(t_0\) gives \(\kappa_0\ge2\). For every
\(\kappa_0\ge2\), \(2\kappa_0/3\le\lfloor\kappa_0\rfloor\le\kappa_0\).
Indeed, this is immediate when \(2\le\kappa_0<3\), while for
\(\kappa_0\ge3\),
\(\lfloor\kappa_0\rfloor\ge\kappa_0-1\ge2\kappa_0/3\). Therefore, with
\(y:=\lambda_k^2t_0=k^2/(4\kappa_0^2)\), we have
\begin{equation}
\label{eq:selected-frequency-window}
    \frac19\le y\le\frac14.
\end{equation}

Set
$
    \alpha:=\varepsilon\lambda_k.
$
By \eqref{eq:selected-frequency-window} and
\(\varepsilon\le\sqrt{t_0}\),
\[
    0<\alpha
    \le
    \sqrt{t_0}\lambda_k
    =
    \sqrt y
    \le
    \frac12.
\]
By Lemma~\ref{lem:oscillatory-wasserstein-size},
\[
    \wass_2(p_0,q_0^\pm)
    \le
    \frac{\alpha}{\lambda_k}
    =
    \varepsilon.
\]
Thus \(q_0^\pm\) are valid probability densities supported on \([-1,1]\) and
satisfy
$
    q_0^\pm\in\mathcal Q_\varepsilon(p_0).
$
First fix an arbitrary deterministic measurable score field \(s\). For
\(t\in[t_0,2t_0]\),
we have \(t\le2t_0<1/64\), \(\lambda_k^2t\le2y\le1/2<1\), and
\(2\lambda_k^2\sigma_t^2\le4y\). Hence
Lemma~\ref{lem:ou-mode-persistence} and
Lemma~\ref{lem:two-point-score-separation} give
\begin{align*}
    \int_{\mathbb R}
        |\partial_x\eta_t(x)|^2p_t(x)\,\rmd x
    &\ge
    \frac1{32}\lambda_k^2e^{-4y},\\
    R_+(t,s)+R_-(t,s)
    &\ge
    \frac1{16}\alpha^2\lambda_k^2e^{-4y}.
\end{align*}
Therefore
\[
\begin{aligned}
\frac12\left[
\mathcal L_{q_0^+}(s)+\mathcal L_{q_0^-}(s)
\right]
&=
\frac1{2(T-t_0)}
\int_{t_0}^{T}
\bigl(R_+(t,s)+R_-(t,s)\bigr)\,\rmd t
\\
&\ge
\frac1{2(T-t_0)}
\int_{t_0}^{2t_0}
\bigl(R_+(t,s)+R_-(t,s)\bigr)\,\rmd t
\\
&\ge
\frac1{32}
\frac{\alpha^2\lambda_k^2t_0}{T-t_0}e^{-4y}
\\
&=
\frac{\varepsilon^2}{32t_0(T-t_0)}y^2e^{-4y}.
\end{aligned}
\]
The map \(y\mapsto y^2e^{-4y}\) is increasing on \([1/9,1/4]\), since
\(\frac{\rmd}{\rmd y}(y^2e^{-4y})=2y(1-2y)e^{-4y}\ge0\).
Using \eqref{eq:selected-frequency-window}, we conclude that
\[
\frac12\left[
\mathcal L_{q_0^+}(s)+\mathcal L_{q_0^-}(s)
\right]
\ge
\frac{e^{-4/9}}{2592}
\frac{\varepsilon^2}{t_0(T-t_0)}.
\]
Now let \(\hat s\) be any possibly random score estimator.
Applying the preceding inequality to each realization and then taking expectation gives
\[
\frac12\sum_{\xi\in\{+,-\}}
\mathbb E\left[\mathcal L_{q_0^\xi}(\hat s)\right]
\ge
\frac{e^{-4/9}}{2592}
\frac{\varepsilon^2}{t_0(T-t_0)}.
\]
Since both alternatives belong to \(\mathcal Q_\varepsilon(p_0)\),
\[
\sup_{q_0\in\mathcal Q_\varepsilon(p_0)}
\mathbb E\left[\mathcal L_{q_0}(\hat s)\right]
\ge
\max_{\xi\in\{+,-\}}
\mathbb E\left[\mathcal L_{q_0^\xi}(\hat s)\right]
\ge
\frac12\sum_{\xi\in\{+,-\}}
\mathbb E\left[\mathcal L_{q_0^\xi}(\hat s)\right].
\]
Taking the infimum over \(\hat s\) proves
\eqref{eq:one-dimensional-wasserstein-lower-bound}.
\end{proof}

\subsection{Proofs of Theorems~\ref{thm:fixed-dimensional-lower-main}
and~\ref{thm:finite-sample-minimax-optimality}}
\label{app:score-lower-and-minimax}

We finally lift the preceding one-dimensional construction to arbitrary fixed
dimension through a product embedding.

\begin{proofof}{Theorem~\ref{thm:fixed-dimensional-lower-main}}
The case \(d=1\) is exactly
Proposition~\ref{prop:ou-wasserstein-lower-bound}. We therefore assume
\(d\ge2\).

Let \(q_{0,1}^{\pm}\) be the one-dimensional alternatives constructed in the
proof of Proposition~\ref{prop:ou-wasserstein-lower-bound}. Thus they use the
same integer \(k\), frequency \(\lambda_k=\pi k\), and amplitude
\(\alpha=\varepsilon\lambda_k\). With \(y:=\lambda_k^2t_0\), they satisfy
\(1/9\le y\le1/4\) and \(0<\alpha\le1/2\).
Define the \(d\)-dimensional embeddings by
\[
    q_0^{\pm,(d)}(x)
    :=
    q_{0,1}^{\pm}(x_1)
    \prod_{j=2}^d p_0^{(1)}(x_j).
\]
Then \(\operatorname{supp}(q_0^{\pm,(d)})
=\operatorname{supp}(p_0^{(d)})=[-1,1]^d
\subseteq\mathcal B(0,\sqrt d)\).

We next verify the Wasserstein constraint. Let \(\pi\) be a coupling of
\(q_{0,1}^{\pm}\) and \(p_0^{(1)}\). Let
\(Y_2,\ldots,Y_d\) be independent with common law \(p_0^{(1)}\), independent
of \((U,V)\sim\pi\). Then
\(X=(U,Y_2,\ldots,Y_d)\) and \(Y=(V,Y_2,\ldots,Y_d)\) form a coupling of
\(q_0^{\pm,(d)}\) and \(p_0^{(d)}\), with \(\|X-Y\|_2^2=|U-V|^2\).
Taking the infimum over all one-dimensional couplings \(\pi\) gives
\[
    \wass_2\bigl(q_0^{\pm,(d)},p_0^{(d)}\bigr)
    \le
    \wass_2\bigl(q_{0,1}^{\pm},p_0^{(1)}\bigr)
    \le
    \varepsilon.
\]
The preceding Wasserstein bound therefore gives
$
 q_0^{\pm,(d)}\in\mathcal Q_\varepsilon(p_0^{(d)}).
$
Let \(q_{t,1}^{\pm}\) and \(p_t^{(1)}\) denote the one-dimensional OU
evolutions of \(q_{0,1}^{\pm}\) and \(p_0^{(1)}\), respectively. Since the
\(d\)-dimensional OU semigroup factorizes across coordinates, set
\begin{align*}
    q_t^{\pm,(d)}&:=P_t^{(d)}q_0^{\pm,(d)},
    \qquad
    p_t^{(d)}:=P_t^{(d)}p_0^{(d)},\\
    q_t^{\pm,(d)}(x)
    &=q_{t,1}^{\pm}(x_1)\prod_{j=2}^d p_t^{(1)}(x_j),
    \qquad
    p_t^{(d)}(x)=\prod_{j=1}^d p_t^{(1)}(x_j).
\end{align*}
For \(x=(x_1,\ldots,x_d)\in\mathbb R^d\), write
\(x_{2:d}:=(x_2,\ldots,x_d)\in\mathbb R^{d-1}\) and set
\(p_{t,\perp}(x_{2:d}):=\prod_{j=2}^d p_t^{(1)}(x_j)\).
Then \(p_{t,\perp}\) is a probability density on \(\mathbb R^{d-1}\), and
$
    q_t^{\pm,(d)}(x)
    =
    q_{t,1}^{\pm}(x_1)p_{t,\perp}(x_{2:d}).
$
The logarithm of this product factorizes additively, and therefore
\[
    \nabla\log q_t^{\pm,(d)}(x)
    =
    \left(
        \partial_x\log q_{t,1}^{\pm}(x_1),
        \partial_x\log p_t^{(1)}(x_2),
        \ldots,
        \partial_x\log p_t^{(1)}(x_d)
    \right).
\]
Thus the score separation occurs only in the first coordinate:
\[
    \nabla\log q_t^{+,(d)}(x)
    -
    \nabla\log q_t^{-,(d)}(x)
    =
    \left(
        \partial_x\log q_{t,1}^{+}(x_1)
        -
        \partial_x\log q_{t,1}^{-}(x_1),
        0,\ldots,0
    \right).
\]

For a deterministic measurable vector-valued score field \(s\), define
\[
    R_\xi^{(d)}(t,s)
    :=
    \mathbb E_{q_t^{\xi,(d)}}
    \left[
        \left\|
            s(X,t)-\nabla\log q_t^{\xi,(d)}(X)
        \right\|_2^2
    \right],
    \qquad \xi\in\{+,-\}.
\]
The Hilbert-space version of the same two-point inequality used in
Lemma~\ref{lem:two-point-score-separation} gives
\[
\begin{aligned}
R_+^{(d)}(t,s)+R_-^{(d)}(t,s)
&\ge
\int_{\mathbb R^d}
\frac{
    q_t^{+,(d)}(x)q_t^{-,(d)}(x)
}{
    q_t^{+,(d)}(x)+q_t^{-,(d)}(x)
}
\left\|
    \nabla\log q_t^{+,(d)}(x)
    -
    \nabla\log q_t^{-,(d)}(x)
\right\|_2^2\,\rmd x .
\end{aligned}
\]
The overlap factor satisfies
\[
\frac{
    q_t^{+,(d)}(x)q_t^{-,(d)}(x)
}{
    q_t^{+,(d)}(x)+q_t^{-,(d)}(x)
}
=
\frac{
    q_{t,1}^{+}(x_1)q_{t,1}^{-}(x_1)
}{
    q_{t,1}^{+}(x_1)+q_{t,1}^{-}(x_1)
}
p_{t,\perp}(x_{2:d}).
\]
Since \(p_{t,\perp}\) integrates to one, the resulting \(d\)-dimensional lower bound
reduces exactly to the one-dimensional two-point separation:
\[
    R_+^{(d)}(t,s)+R_-^{(d)}(t,s)
    \ge
    2\alpha^2
    \int_{\mathbb R}
        |\partial_x\eta_t(x_1)|^2p_t^{(1)}(x_1)\,\rmd x_1 .
\]
For \(t\in[t_0,2t_0]\), the same bounds as in the one-dimensional proof give
\[
    t\le2t_0<\frac1{64},
    \qquad
    \lambda_k^2t\le2y\le\frac12<1,
    \qquad
    2\lambda_k^2\sigma_t^2\le4y.
\]
Hence Lemma~\ref{lem:ou-mode-persistence} and the preceding two-point bound give
\begin{align*}
    \int_{\mathbb R}
        |\partial_x\eta_t(x_1)|^2p_t^{(1)}(x_1)\,\rmd x_1
    &\ge
    \frac1{32}\lambda_k^2e^{-4y},\\
    R_+^{(d)}(t,s)+R_-^{(d)}(t,s)
    &\ge
    \frac1{16}\alpha^2\lambda_k^2e^{-4y}.
\end{align*}

Therefore, for every deterministic measurable vector-valued score field \(s\),
\[
\begin{aligned}
\frac12\left[
\mathcal L_{q_0^{+,(d)}}(s)+
\mathcal L_{q_0^{-,(d)}}(s)
\right]
&=
\frac1{2(T-t_0)}
\int_{t_0}^{T}
\bigl(R_+^{(d)}(t,s)+R_-^{(d)}(t,s)\bigr)\,\rmd t
\\
&\ge
\frac1{2(T-t_0)}
\int_{t_0}^{2t_0}
\bigl(R_+^{(d)}(t,s)+R_-^{(d)}(t,s)\bigr)\,\rmd t
\\
&\ge
\frac1{32}
\frac{\alpha^2\lambda_k^2t_0}{T-t_0}e^{-4y}
\\
&=
\frac{\varepsilon^2}{32t_0(T-t_0)}y^2e^{-4y}
\ge
\frac{e^{-4/9}}{2592}
\frac{\varepsilon^2}{t_0(T-t_0)} .
\end{aligned}
\]
The last step uses the monotonicity of \(y^2e^{-4y}\) on
\([1/9,1/4]\), as in the one-dimensional proof. The numerical constant
\(e^{-4/9}/2592\) is
independent of \(t_0\), \(\varepsilon\), \(T\), \(d\), and the estimator \(s\).
Now let \(\hat s\) be any possibly random score estimator.
Taking expectation in the preceding inequality and using
\(q_0^{\pm,(d)}\in\mathcal Q_\varepsilon(p_0^{(d)})\) gives
\[
\sup_{q_0\in\mathcal Q_\varepsilon(p_0^{(d)})}
\mathbb E\left[\mathcal L_{q_0}(\hat s)\right]
\ge
\frac12\sum_{\xi\in\{+,-\}}
\mathbb E\left[\mathcal L_{q_0^{\xi,(d)}}(\hat s)\right]
\ge
\frac{e^{-4/9}}{2592}
\frac{\varepsilon^2}{t_0(T-t_0)}.
\]
Taking the infimum over \(\hat s\) proves
\eqref{eq:fixed-dimensional-lower-main}.
\end{proofof}

\begin{proofof}{Theorem~\ref{thm:finite-sample-minimax-optimality}}
The assumed bound on $t_0$ implies $\tau_0\le1$. Applying
\eqref{eq:rate-finite-sample-robust-bound} with $\vartheta^2=D^2/d$, and using
$n+1\ge n$, $\tau_0\ge t_0$, and $C_{d,D}\ge2$, gives the stated bound
uniformly over $p_0\in\mathcal P_D$. Taking the supremum over $p_0$ and then
the infimum over estimators proves the minimax upper bound.

For the lower bound, set
\(A:=1/[n(T-t_0)\tau_0^{d/2}]\) and
\(B:=\varepsilon^2/[t_0(T-t_0)]\).
Theorem~\ref{thm:integrated-statistical-lower-bound} and the sample-rich condition give
$\mathfrak R_{n,\varepsilon}^{\star}(D;t_0,T)\ge c_1A$. Because $D\ge\sqrt d$, the reference law
$p_0^{(d)}$ belongs to $\mathcal P_D$. Theorem~\ref{thm:fixed-dimensional-lower-main} therefore gives
$\mathfrak R_{n,\varepsilon}^{\star}(D;t_0,T)\ge(e^{-4/9}/2592)B$ for $\varepsilon>0$.
For $\varepsilon=0$, this bound holds trivially because $B=0$. Consequently,
\[
\begin{aligned}
  \mathfrak R_{n,\varepsilon}^{\star}(D;t_0,T)
  \ge
  (c_1A)\vee\left(\frac{e^{-4/9}}{2592}B\right)
  \ge
  \frac12\min\left\{c_1,\frac{e^{-4/9}}{2592}\right\}(A+B)
  =
  c_{d,D}(A+B).
\end{aligned}
\]
This is the lower bound in \eqref{eq:combined-finite-sample-minimax-rate}.
\end{proofof}

\subsection{Proof of Theorem~\ref{thm:robust-kl-minimax}}
\label{app:robust-kl-lower-and-minimax}

\begin{proposition}[Intrinsic positive-time KL shift cost]
\label{prop:kl-shift-lower}
Fix $d\ge1$ and $D>0$. For every $n\ge1$, $t_0>0$, and
$0\le\varepsilon\le\sqrt{\tau_0}$,
\begin{equation}
  \mathfrak K_{n,\varepsilon}^{\star}(D;t_0)
  \ge
  \frac{e^{-1}}{\pi}\frac{\varepsilon^2}{\tau_0}.
  \label{eq:kl-shift-lower}
\end{equation}
\end{proposition}

\begin{proof}
The case $\varepsilon=0$ is immediate, so assume
$0<\varepsilon\le\sqrt{\tau_0}$. Fix an arbitrary measurable, possibly
randomized probability-measure estimator $\hat\mu_n$ based on $n$ reference
samples. Take $p_0:=\delta_0\in\mathcal P_D$ and
$q_0^\pm:=\delta_{\pm\varepsilon e_1}$, where $e_1$ is the first standard
basis vector. Then $\wass_2(q_0^\pm,p_0)=\varepsilon$, hence
$q_0^\pm\in\mathcal Q_\varepsilon(p_0)$. Under $p_0$, all reference samples
vanish almost surely, so the observation law is identical for the two target
alternatives and the only randomness in $\hat\mu_n$ is internal.
Moreover,
$q_{t_0}^\pm=\mathcal N(\pm a_{t_0}\varepsilon e_1,\rho_{t_0}I_d)$.
Let $\Phi$ and $\phi$ be the standard Gaussian distribution function and
density, and set
$x:=a_{t_0}\varepsilon/\sqrt{\rho_{t_0}}
=\varepsilon/\sqrt{\tau_0}\in(0,1]$.
Let $d_{\mathrm{TV}}(\mu,\nu):=\sup_{A\text{ Borel}}|\mu(A)-\nu(A)|$
denote the total variation distance. Taking the half-space
$A:=\{z\in\mathbb R^d:z_1\ge0\}$ gives
\begin{equation}
\begin{aligned}
  d_{\mathrm{TV}}(q_{t_0}^+,q_{t_0}^-)
  &\ge \lvert q_{t_0}^+(A)-q_{t_0}^-(A)\rvert
  =2\Phi(x)-1\\
  &=2\int_0^x\phi(u)\,\rmd u\ge 2\phi(1)x
  =2\phi(1)\frac{\varepsilon}{\sqrt{\tau_0}}.
\end{aligned}
\label{eq:kl-two-gaussian-tv}
\end{equation}

For any probability measure $\mu$, Pinsker's inequality and the triangle
inequality for total variation imply
\[
\begin{aligned}
  \mathrm{KL}(q_{t_0}^+\|\mu)
  +\mathrm{KL}(q_{t_0}^-\|\mu)
  &\ge
  2d_{\mathrm{TV}}(q_{t_0}^+,\mu)^2
  +2d_{\mathrm{TV}}(q_{t_0}^-,\mu)^2
  \\
  &\ge \left[
    d_{\mathrm{TV}}(q_{t_0}^+,\mu)
    +d_{\mathrm{TV}}(q_{t_0}^-,\mu)
  \right]^2
  \ge d_{\mathrm{TV}}(q_{t_0}^+,q_{t_0}^-)^2,
\end{aligned}
\]
with the inequality understood trivially if either KL divergence is infinite.
Applying this pointwise to each realization of $\hat\mu_n$, taking expectation
over its internal randomness, and using
\eqref{eq:kl-two-gaussian-tv} yields
\[
  \frac12\sum_{\xi\in\{+,-\}}
  \mathbb E\left[\mathrm{KL}(q_{t_0}^\xi\|\hat\mu_n)\right]
  \ge
  \frac12d_{\mathrm{TV}}(q_{t_0}^+,q_{t_0}^-)^2
  \ge 2\phi(1)^2\frac{\varepsilon^2}{\tau_0}
   =\frac{e^{-1}}{\pi}\frac{\varepsilon^2}{\tau_0}.
\]
The supremum over $q_0\in\mathcal Q_\varepsilon(p_0)$ is at least the
maximum, hence the average, of these two risks. Taking the infimum over
$\hat\mu_n$ proves \eqref{eq:kl-shift-lower}.
\end{proof}

\begin{proofof}{Theorem~\ref{thm:robust-kl-minimax}}
Take \(c_0\) and \(c_{d,D}^{\mathrm{KL}}\) as in
Theorem~\ref{thm:nominal-kl-minimax}, and set
\[
  \tilde c_{d,D}
  :=\frac12\min\left\{
    c_{d,D}^{\mathrm{KL}},
    \frac{e^{-1}}{\pi}
  \right\}.
\]
Set $A:=1/(n\tau_0^{d/2})$ and $B:=\varepsilon^2/\tau_0$.
Since $p_0\in\mathcal Q_\varepsilon(p_0)$, the robust minimax risk dominates
its nominal counterpart. Theorem~\ref{thm:nominal-kl-minimax} and the
sample-rich condition therefore give
$\mathfrak K_{n,\varepsilon}^{\star}(D;t_0)
\ge c_{d,D}^{\mathrm{KL}}A$. Proposition~\ref{prop:kl-shift-lower} also gives
$\mathfrak K_{n,\varepsilon}^{\star}(D;t_0)\ge(e^{-1}/\pi)B$. Consequently,
\[
\begin{aligned}
  \mathfrak K_{n,\varepsilon}^{\star}(D;t_0)
  &\ge \bigl(c_{d,D}^{\mathrm{KL}}A\bigr)
  \vee\left(\frac{e^{-1}}{\pi}B\right)
  \ge \frac12\min\left\{c_{d,D}^{\mathrm{KL}},\frac{e^{-1}}{\pi}\right\}(A+B)
  =\tilde c_{d,D}(A+B).
\end{aligned}
\]

For the upper bound, use the self-consistent Gaussian-blanket output
$\hat\mu_n(\rmd x)=\hat p_{n,t_0}^{D/\sqrt d}(x)\,\rmd x$. Since
$t_0\le c_0\le\log2$, one has $\tau_0\le1$, and
\eqref{eq:self-consistent-kl-sampling-bound} gives, uniformly over
$p_0\in\mathcal P_D$ and $q_0\in\mathcal Q_\varepsilon(p_0)$,
$\mathbb E[\mathrm{KL}(q_{t_0}\|\hat\mu_n)]
\le B+C_{d,D}/[2(n+1)\tau_0^{d/2}]
\le C_{d,D}(A+B)/2$, where the last step uses $n+1\ge n$ and
$C_{d,D}\ge2$. Taking the relevant
suprema and then the infimum over estimators proves the upper bound in
\eqref{eq:robust-kl-minimax-rate}.
\end{proofof}

\section{Proofs for Adaptation to Intrinsic Dimension}
\label{app:unknown-subspace}

This appendix proves the results of Section~\ref{sec:unknown-subspace}.
The upper bound combines two ingredients: a bound on the source energy
missed by the empirical span, and an OU comparison that separates this error
from estimation within the span. Conditional Gaussian-blanket estimates then
give Theorem~\ref{thm:unknown-subspace-upper}. The final subsection
establishes the intrinsic-dimensional minimax rate for score estimation.
The corresponding positive-time KL results are given in
Appendix~\ref{app:robust-reverse-sampling}.

Throughout this appendix, write
\(
  \tau_t:=\rho_t/a_t^2=e^t-1
\)
for $t>0$, retaining $\tau_0=e^{t_0}-1$ for its value at $t_0$.

\subsection{Missed Energy of the Empirical Span}
\label{app:unknown-subspace-energy}

\begin{lemma}[Missed source energy]
\label{lem:unknown-subspace-energy}
Let $X_1,X_2,\ldots$ be i.i.d.\ samples from the law of a random vector
$X\in\mathbb R^d$ satisfying $\|X\|_2\le D$ almost surely for some $D>0$.
Suppose that the second-moment matrix
$C:=\mathbb E[XX^\top]$ has rank at most $k\ge1$. If $\Pi_j$ is the
orthogonal projector onto $\operatorname{span}(X_1,\ldots,X_j)$, with
$\Pi_0=0$, then, for every integer $m\ge0$,
\begin{equation}
  \mathbb E\operatorname{tr}[C(I_d-\Pi_m)]
  \le \frac{kD^2}{m+k}.
  \label{eq:unknown-subspace-energy-bound}
\end{equation}
\end{lemma}

\begin{proof}
Set $Q_j:=I_d-\Pi_j$,
$e_j:=\operatorname{tr}(CQ_j)$, and $b_j:=\mathbb E e_j$.
Since $Q_j^\top=Q_j^2=Q_j$,
\[
  e_j=\operatorname{tr}(Q_jCQ_j)
  =\int\|Q_jx\|_2^2\,P_X(\rmd x),
\]
where $P_X$ denotes the law of $X$. Thus $e_j$ is the source energy outside
the observed span. The next observation
$X_{j+1}$ is independent of $X_1,\ldots,X_j$. Set $z:=Q_jX_{j+1}$.
On $\{z\ne0\}$, the rank-one projector update and
$\|z\|_2\le\|X_{j+1}\|_2\le D$ give
\[
  \Pi_{j+1}=\Pi_j+\frac{zz^\top}{\|z\|_2^2},
  \qquad
  e_j-e_{j+1}=\frac{z^\top Cz}{\|z\|_2^2}
  \ge\frac{z^\top Cz}{D^2}.
\]
On $\{z=0\}$, the span is unchanged and the same inequality holds with both
sides equal to zero. Since
$\mathbb E[zz^\top\mid X_1,\ldots,X_j]=Q_jCQ_j$,
\begin{align}
  \mathbb E[e_{j+1}\mid X_1,\ldots,X_j]
  &\le e_j-\frac{\operatorname{tr}(CQ_jCQ_j)}{D^2}\notag\\
  &=e_j-\frac{\operatorname{tr}[(Q_jCQ_j)^2]}{D^2}
  \le e_j-\frac{e_j^2}{kD^2}.
  \label{eq:unknown-subspace-energy-recursion-conditional}
\end{align}
Indeed, $Q_jCQ_j$ is positive semidefinite and has rank $r\le k$.
Writing $\lambda_1,\ldots,\lambda_r$ for its nonzero eigenvalues,
Cauchy--Schwarz gives
\[
  e_j^2=\left(\sum_{i=1}^r\lambda_i\right)^2
  \le r\sum_{i=1}^r\lambda_i^2
  \le k\operatorname{tr}[(Q_jCQ_j)^2].
\]
Taking expectations and using Jensen's inequality gives
\begin{equation}
  b_{j+1}\le b_j-\frac{b_j^2}{kD^2}.
  \label{eq:unknown-subspace-energy-recursion}
\end{equation}
If $b_m=0$, the claim is immediate. Otherwise,
\eqref{eq:unknown-subspace-energy-recursion} gives $b_j\ge b_m>0$ for
$0\le j\le m$, so, for $j<m$,
\[
  \frac1{b_{j+1}}
  \ge\frac1{b_j(1-b_j/(kD^2))}
  \ge\frac1{b_j}+\frac1{kD^2}.
\]
Since $b_0\le D^2$, iteration yields
\[
  b_m\le\frac{1}{b_0^{-1}+m/(kD^2)}
  \le\frac{kD^2}{m+k}.
\]
\end{proof}

Under Assumptions~\ref{assum:support} and~\ref{assum:unknown-subspace},
$C=U\mathbb E[ZZ^\top]U^\top$ has rank at most $k$, even when the latent
second-moment matrix is singular. Thus the lemma bounds the missed energy
without a lower bound on the nonzero eigenvalues.

\subsection{An Orthogonal OU Comparison}
\label{app:unknown-subspace-oracle}

The projected source law is supported on the empirical span, whereas the
target need not be. The following comparison separates error within the
span from the energy in its orthogonal complement.

\begin{lemma}[Orthogonal KL comparison]
\label{lem:unknown-subspace-oracle}
Fix $t_0>0$. Let $\Pi$ be an orthogonal projector on $\mathbb R^d$
and let $p_0\in\mathcal P_2(\mathbb R^d)$. Define
\[
  p_0^\Pi:=\Pi_\#p_0,
  \qquad
  e_\Pi:=\int\|(I_d-\Pi)x\|_2^2\,p_0(\rmd x).
\]
Let $\hat\mu_0\in\mathcal P_2(\mathbb R^d)$ be supported on
$\operatorname{range}(\Pi)$, and set
$\hat p_t:=P_t\hat\mu_0$. If
$D_2(P_{t_0}p_0^\Pi\|\hat p_{t_0})<\infty$, then, for every
$q_0\in\mathcal P_2(\mathbb R^d)$,
\begin{equation}
  \mathrm{KL}\left(q_{t_0}\middle\|\hat p_{t_0}\right)
  \le
  D_2\left(P_{t_0}p_0^\Pi\middle\|\hat p_{t_0}\right)
  +\frac{\wass_2^2(q_0,p_0)+e_\Pi}{\tau_0},
  \label{eq:unknown-subspace-oracle-kl}
\end{equation}
where $q_t:=P_tq_0$.
\end{lemma}

\begin{proof}
Set $Q:=I_d-\Pi$, $p_t^\Pi:=P_tp_0^\Pi$, and
$q_t^\Pi:=P_t(\Pi_\#q_0)$. Here projection is applied to the initial law
before the full $d$-dimensional OU evolution; in general,
$q_t^\Pi\ne\Pi_\#q_t$.
Let $r:=\operatorname{rank}(\Pi)$ and choose orthonormal coordinates
$x=(z,w)\in\mathbb R^r\times\mathbb R^{d-r}$ adapted to
$\operatorname{range}(\Pi)\oplus\operatorname{range}(Q)$.
Write $\varphi_v^{(j)}$ for the density of $\mathcal N(0,vI_j)$,
with $\varphi_v^{(0)}\equiv1$.
Let $q_{\parallel,t}$ and $\hat p_{\parallel,t}$ denote the densities of
the $z$-marginals of $q_t$ and $\hat p_t$, respectively.
Projection of the initial law preserves its $z$-component, so
$q_t^\Pi$ and $q_t$ have the same $z$-marginal.
Since $\Pi_\#q_0$ and $\hat\mu_0$ are supported on
$\operatorname{range}(\Pi)$, their $w$-components under the OU evolution
consist solely of Gaussian noise, independent of their $z$-components. Hence
\[
  q_t^\Pi(z,w)=q_{\parallel,t}(z)\varphi_{\rho_t}^{(d-r)}(w),
  \qquad
  \hat p_t(z,w)=\hat p_{\parallel,t}(z)\varphi_{\rho_t}^{(d-r)}(w).
\]
The density $q_t$ itself need not factor.
All KL divergences below are finite by Lemma~\ref{lem:fixed-time-ou-kl},
since the initial laws have finite second moments.
The log density ratio $\log(q_t^\Pi/\hat p_t)$ depends only on $z$, and thus
\[
  \int_{\mathbb R^d}q_t(x)\log\frac{q_t^\Pi(x)}{\hat p_t(x)}\,\rmd x
  =\int_{\mathbb R^r}q_{\parallel,t}(z)
    \log\frac{q_{\parallel,t}(z)}{\hat p_{\parallel,t}(z)}\,\rmd z
  =\int_{\mathbb R^d}q_t^\Pi(x)\log\frac{q_t^\Pi(x)}{\hat p_t(x)}\,\rmd x.
\]
Decomposing $\log(q_t/\hat p_t)$ through $q_t^\Pi$ therefore gives
\begin{equation}
  \mathrm{KL}(q_t\|\hat p_t)
  =\mathrm{KL}(q_t\|q_t^\Pi)
  +\mathrm{KL}(q_t^\Pi\|\hat p_t).
  \label{eq:unknown-subspace-entropy-chain}
\end{equation}

Let $(X,Y)$ be an optimal coupling of $(p_0,q_0)$ and put
$\delta:=Y-X$, so that
$\mathbb E\|\delta\|_2^2=\wass_2^2(p_0,q_0)$.
Applying Lemma~\ref{lem:fixed-time-ou-kl} with the couplings
$(Y,\Pi Y)$ and $(\Pi Y,\Pi X)$ gives
\[
  \mathrm{KL}(q_{t_0}\|q_{t_0}^\Pi)
  \le\frac{\mathbb E\|QY\|_2^2}{2\tau_0},
  \qquad
  \mathrm{KL}(q_{t_0}^\Pi\|p_{t_0}^\Pi)
  \le\frac{\mathbb E\|\Pi\delta\|_2^2}{2\tau_0}.
\]
Applying Lemma~\ref{lem:three-density-kl} to the second term in
\eqref{eq:unknown-subspace-entropy-chain} at $t=t_0$, with intermediate
density $p_{t_0}^\Pi$, and using the preceding bounds yields
\[
  \mathrm{KL}(q_{t_0}\|\hat p_{t_0})
  \le D_2(p_{t_0}^\Pi\|\hat p_{t_0})
  +\frac{\mathbb E\|\Pi\delta\|_2^2
         +\tfrac12\mathbb E\|QY\|_2^2}{\tau_0}.
\]
Since $QY=QX+Q\delta$ and $\mathbb E\|QX\|_2^2=e_\Pi$,
\[
  \frac12\mathbb E\|QY\|_2^2
  \le e_\Pi+\mathbb E\|Q\delta\|_2^2.
\]
Together with the orthogonal decomposition
$\mathbb E\|\Pi\delta\|_2^2+\mathbb E\|Q\delta\|_2^2
=\wass_2^2(p_0,q_0)$, this proves \eqref{eq:unknown-subspace-oracle-kl}.
\end{proof}

\subsection{Proof of Theorem~\ref{thm:unknown-subspace-upper}}
\label{app:unknown-subspace-upper}

\begin{proof}
Let $\mathcal F_1:=\sigma(X_1,\ldots,X_{n_1})$ be the $\sigma$-field
generated by the pilot sample, and write $M:=M_{\vartheta_k,t_0}$.
Set $\hat r:=\operatorname{rank}(\hat\Pi)=\dim\hat L$.
For $t>0$, write
$\hat p_{n,t}^{\mathrm{sub}}:=P_t\hat\mu_{n,0}^{\mathrm{sub}}$.
The empirical projector $\hat\Pi$ is measurable, since it can be constructed
by applying Gram--Schmidt to the pilot sample with zero residuals omitted.
Hence the projected observations $Y_i$ are also measurable.
The maps $y\mapsto\delta_y$
and $\Sigma\mapsto\mathcal N(0,\Sigma)$ are continuous for weak convergence,
including singular covariance matrices. Their finite mixture
$\hat\mu_{n,0}^{\mathrm{sub}}$ is therefore measurable for the Borel
$\sigma$-field of the weak topology. Almost surely,
$\|Y_i\|_2\le D$, $\hat r\le k$, and
\begin{equation}
  \int\|x\|_2^2\,\hat\mu_{n,0}^{\mathrm{sub}}(\rmd x)
  =\frac{\sum_{i=1}^{n_2}\|Y_i\|_2^2+M\vartheta_k^2\hat r}{n_2+M}
  \le D^2.
  \label{eq:unknown-subspace-estimator-moment}
\end{equation}
Since $Y_i\in\operatorname{range}(\hat\Pi)$ and
$\mathcal N(0,\vartheta_k^2\hat\Pi)$ is supported on the same subspace,
so is their mixture $\hat\mu_{n,0}^{\mathrm{sub}}$, as required by
Lemma~\ref{lem:unknown-subspace-oracle}.

\emph{Step 1: Conditional R\'enyi bound.}
Condition on $\mathcal F_1$. Then $\hat\Pi$ is fixed and the $Y_i$ are
independent with law $p_0^{\hat\Pi}:=\hat\Pi_\#p_0$, supported on
$\mathcal B(0,D)\cap\operatorname{range}(\hat\Pi)$. Set
\[
  \Sigma_t:=\rho_tI_d+a_t^2\vartheta_k^2\hat\Pi,
  \qquad
  k_{t,y}(x):=\varphi_{\rho_tI_d}(x-a_ty),
\]
where $\varphi_\Sigma$ denotes the density of
$\mathcal N(0,\Sigma)$ for positive definite $\Sigma$. Then
\begin{equation}
  \hat p_{n,t}^{\mathrm{sub}}(x)
  =\frac{\sum_{i=1}^{n_2}k_{t,Y_i}(x)+M\varphi_{\Sigma_t}(x)}{n_2+M}.
  \label{eq:unknown-subspace-mixture-density}
\end{equation}
This formula also shows that the score is jointly measurable in the
observations and $(x,t)$ for $t>0$.
For $y\in\operatorname{range}(\hat\Pi)$, the Gaussian factors of
$k_{t,y}$ and $\varphi_{\Sigma_t}$ on
$\operatorname{range}(\hat\Pi)^\perp$ coincide and cancel in the density ratio.
Repeating the Gaussian-ratio calculation of
Lemma~\ref{lem:gaussian-blanket-domination} in
the $\hat r$ tangential directions gives, for
$y\in\mathcal B(0,D)\cap\operatorname{range}(\hat\Pi)$ and $t\ge t_0$,
\begin{equation}
  \sup_{x\in\mathbb R^d}\frac{k_{t,y}(x)}{\varphi_{\Sigma_t}(x)}
  =\exp\left(\frac{\|y\|_2^2}{2\vartheta_k^2}\right)
   \left(1+\frac{\vartheta_k^2}{\tau_t}\right)^{\hat r/2}
  \le M.
  \label{eq:unknown-subspace-conditional-domination}
\end{equation}
For $\hat r=0$, the only possible center is $y=0$ and the ratio equals one. The
last inequality uses $\|y\|_2\le D$, $\hat r\le k$, and $\tau_t\ge\tau_0$.

Conditional on $\mathcal F_1$, the mean kernel density is
$\mathbb E[k_{t,Y_i}(x)\mid\mathcal F_1]=(P_tp_0^{\hat\Pi})(x)$.
By \eqref{eq:unknown-subspace-mixture-density} and
\eqref{eq:unknown-subspace-conditional-domination}, applying
Lemma~\ref{lem:dominated-kernel-add-one} with $n=n_2$, kernels $k_{t,y}$,
reference density $\varphi_{\Sigma_t}$, and domination constant $M$ yields
\begin{equation}
  \mathbb E\left[
  \int\frac{(P_tp_0^{\hat\Pi})(x)^2}
  {\hat p_{n,t}^{\mathrm{sub}}(x)}\,\rmd x
  \;\middle|\;\mathcal F_1\right]
  \le\frac{n_2+M}{n_2+1}.
  \label{eq:unknown-subspace-conditional-second-moment}
\end{equation}
Conditional Jensen's inequality for the logarithm yields
\begin{equation}
  \mathbb E\left[
  D_2(P_tp_0^{\hat\Pi}\|\hat p_{n,t}^{\mathrm{sub}})
  \;\middle|\;\mathcal F_1\right]
  \le\log\left(1+\frac{M-1}{n_2+1}\right).
  \label{eq:unknown-subspace-conditional-renyi}
\end{equation}
The finiteness condition in Lemma~\ref{lem:unknown-subspace-oracle} holds
for every such sample, since
\eqref{eq:unknown-subspace-conditional-domination} and
\eqref{eq:unknown-subspace-mixture-density} give
\[
  P_tp_0^{\hat\Pi}\le M\varphi_{\Sigma_t},
  \qquad
  \hat p_{n,t}^{\mathrm{sub}}\ge\frac{M\varphi_{\Sigma_t}}{n_2+M},
  \qquad
  \int\frac{(P_tp_0^{\hat\Pi})^2}{\hat p_{n,t}^{\mathrm{sub}}}
  \le n_2+M.
\]

\emph{Step 2: KL and score-risk bounds.}
Fix $q_0\in\mathcal Q_\varepsilon(p_0)$ and define
\[
  e_{n_1}:=\int\|(I_d-\hat\Pi)x\|_2^2\,p_0(\rmd x)
  =\operatorname{tr}[C(I_d-\hat\Pi)],
  \qquad C:=\mathbb E_{p_0}[XX^\top].
\]
Taking conditional expectation in Lemma~\ref{lem:unknown-subspace-oracle}
and using \eqref{eq:unknown-subspace-conditional-renyi} at $t=t_0$ gives
\[
  \mathbb E\left[
    \mathrm{KL}\left(q_{t_0}\middle\|\hat p_{n,t_0}^{\mathrm{sub}}\right)
    \;\middle|\;\mathcal F_1\right]
  \le
  \log\left(1+\frac{M-1}{n_2+1}\right)
  +\frac{\varepsilon^2+e_{n_1}}{\tau_0}.
\]
Since $\operatorname{rank}(C)\le k$,
\eqref{eq:unknown-subspace-energy-bound} gives
$\mathbb E e_{n_1}\le kD^2/(n_1+k)$.
Averaging over the pilot sample therefore yields
\begin{equation}
  \mathbb E\left[
    \mathrm{KL}\left(q_{t_0}\middle\|\hat p_{n,t_0}^{\mathrm{sub}}\right)
  \right]
  \le
  \log\left(1+\frac{M-1}{n_2+1}\right)
  +\frac{\varepsilon^2}{\tau_0}
  +\frac{kD^2}{(n_1+k)\tau_0}.
  \label{eq:unknown-subspace-mean-kl}
\end{equation}
For each sample realization, $q_0$ and $\hat\mu_{n,0}^{\mathrm{sub}}$
belong to $\mathcal P_2(\mathbb R^d)$, so
Lemma~\ref{lem:relative-de-bruijn-ou} gives
\[
  \mathcal L_{q_0}(\hat s_n^{\mathrm{sub}})
  =\frac{2}{T-t_0}\left[
    \mathrm{KL}\left(q_{t_0}\middle\|\hat p_{n,t_0}^{\mathrm{sub}}\right)
    -\mathrm{KL}\left(q_T\middle\|\hat p_{n,T}^{\mathrm{sub}}\right)
  \right]
  \le\frac{2}{T-t_0}
    \mathrm{KL}\left(q_{t_0}\middle\|\hat p_{n,t_0}^{\mathrm{sub}}\right).
\]
Taking expectations and using \eqref{eq:unknown-subspace-mean-kl} gives a
bound independent of $q_0\in\mathcal Q_\varepsilon(p_0)$.
Taking the supremum over this class proves
\eqref{eq:unknown-subspace-finite}.

Finally, if $k\ge2$ and $\tau_0\le1$, then
$M\le e^{k/2}(1+D^2/k)^{k/2}\tau_0^{-k/2}$ and
$\tau_0^{-1}\le\tau_0^{-k/2}$. Applying these bounds,
$n_2+1\ge n/2$, $n_1+k\ge n/2$, and $\log(1+u)\le u$ to
\eqref{eq:unknown-subspace-finite} gives
\eqref{eq:unknown-subspace-rate} with the stated constant.
\end{proof}

\subsection{Intrinsic-Dimensional Score Minimax Rate}
\label{app:unknown-subspace-minimax}

We use the source class $\mathcal P_{k,D}$ and minimax risk
$\mathfrak R_{n,\varepsilon}^{\star}(k,D;t_0,T)$ defined in
Subsection~\ref{subsec:unknown-subspace-minimax}.

\begin{corollary}
\label{cor:unknown-subspace-minimax}
Fix $k\ge2$ and $D\ge\sqrt{k}$, and let $d\ge k$ be arbitrary.
Let $c_0,c_{k,D}>0$ be the
lower-bound constants in Theorem~\ref{thm:finite-sample-minimax-optimality}
with the dimension set to $k$.
If $n\ge2$,
\begin{equation}
  0<t_0\le\min\{c_0,1/(16\pi^2)\},
  \qquad
  2t_0\le T,
  \qquad
  n\tau_0^{k/2}\ge1,
  \qquad
  0\le\varepsilon\le\sqrt{t_0},
  \label{eq:unknown-subspace-minimax-regime}
\end{equation}
then
\begin{equation}
  \begin{aligned}
  c_{k,D}
  \left[
    \frac{1}{n(T-t_0)\tau_0^{k/2}}
    +\frac{\varepsilon^2}{t_0(T-t_0)}
  \right]
  &\le
  \mathfrak R_{n,\varepsilon}^{\star}(k,D;t_0,T)\\
  &\le
  \tilde C_{k,D}
  \left[
    \frac{1}{n(T-t_0)\tau_0^{k/2}}
    +\frac{\varepsilon^2}{t_0(T-t_0)}
  \right].
  \end{aligned}
  \label{eq:unknown-subspace-minimax-rate}
\end{equation}
Here $\tilde C_{k,D}$ is the constant in
Theorem~\ref{thm:unknown-subspace-upper}.
\end{corollary}

\begin{proof}
The conditions imply $t_0\le1/(16\pi^2)<\log2$, so $\tau_0\le1$.
The upper bound follows uniformly over the source class from
Theorem~\ref{thm:unknown-subspace-upper}, using
$\tau_0\ge t_0$ and $\tilde C_{k,D}\ge2$.

To transfer the $k$-dimensional lower bound, we construct from any ambient
estimator a latent estimator with no greater risk. Fix a linear isometry
$U:\mathbb R^k\to\mathbb R^d$, and let the columns of
$V\in\mathbb R^{d\times(d-k)}$ form an orthonormal basis of
$\operatorname{range}(U)^\perp$. For latent laws $\bar p_0,\bar q_0$, set
$p_0:=U_\#\bar p_0$ and $q_0:=U_\#\bar q_0$. Then
\begin{equation}
  \wass_2(p_0,q_0)=\wass_2(\bar p_0,\bar q_0).
  \label{eq:unknown-subspace-isometric-wasserstein}
\end{equation}
Indeed, one inequality follows by pushing a latent coupling forward, and the
other by applying $U^\top$ to an ambient coupling, whose marginals are
supported on $\operatorname{range}(U)$.

At positive time, with $\bar q_t:=P_t^{(k)}\bar q_0$,
\begin{equation}
  \begin{aligned}
  q_t(Uz+Vw)
  &=\bar q_t(z)\varphi_{\rho_t}^{(d-k)}(w),\\
  U^\top\nabla\log q_t(Uz+Vw)
  &=\nabla\log\bar q_t(z).
  \end{aligned}
  \label{eq:unknown-subspace-lower-factorization}
\end{equation}
Consider an arbitrary measurable, possibly randomized ambient estimator
$\hat s_n$, and represent its internal randomization by a seed $\omega$
independent of the training sample. We write
$\hat s_n(x,t;x_1,\ldots,x_n,\omega)$ to make the sample and seed dependence
explicit. Given latent observations
$Z_1,\ldots,Z_n$, feed $UZ_1,\ldots,UZ_n$ to the estimator.
To match the orthogonal OU noise, introduce an additional seed
$G\sim\mathcal N(0,I_{d-k})$, independent of the sample and $\omega$. Define
\begin{equation}
  \tilde s_n(z,t)
  :=U^\top\hat s_n\left(
  Uz+\sqrt{\rho_t}VG,t;
  UZ_1,\ldots,UZ_n,\omega\right).
  \label{eq:unknown-subspace-randomized-reduction}
\end{equation}
The same $G$ is used for all $(z,t)$, and composition of the measurable maps in
\eqref{eq:unknown-subspace-randomized-reduction} defines a measurable random
score field. Its construction does not depend on $\bar p_0$ or $\bar q_0$.
When $d=k$, the orthogonal component and the extra seed are absent.

Fix the training sample and $\omega$. Since
$\sqrt{\rho_t}G\sim\mathcal N(0,\rho_tI_{d-k})$ and $[U,V]$ is
orthogonal, \eqref{eq:unknown-subspace-lower-factorization} gives
\[
  \begin{aligned}
  \mathbb E_G\int_{\mathbb R^k}
  \|\tilde s_n(z,t)-\nabla\log\bar q_t(z)\|_2^2
  \bar q_t(z)\,\rmd z
  &=
  \int_{\mathbb R^d}
  \|U^\top(\hat s_n(x,t)-\nabla\log q_t(x))\|_2^2
  q_t(x)\,\rmd x\\
  &\le
  \int_{\mathbb R^d}
  \|\hat s_n(x,t)-\nabla\log q_t(x)\|_2^2q_t(x)\,\rmd x.
  \end{aligned}
\]
The inequality uses contraction under orthogonal projection. Integrating in
$t$, dividing by $T-t_0$, and averaging over the sample and $\omega$ yields
\begin{equation}
  \mathbb E\mathcal L_{\bar q_0}^{(k)}(\tilde s_n)
  \le
  \mathbb E\mathcal L_{U_\#\bar q_0}^{(d)}(\hat s_n).
  \label{eq:unknown-subspace-reduction-risk}
\end{equation}
Here $\mathcal L^{(j)}$ denotes the score loss in dimension $j$, and the
expectation on the left also includes $G$. All integrands are nonnegative,
so Tonelli's theorem applies even if a risk is infinite.

Every latent source law supported on $\mathcal B(0,D)\subset\mathbb R^k$
embeds into $\mathcal P_{k,D}$. Fix such a law $\bar p_0$ and set
$p_0=U_\#\bar p_0$. By
\eqref{eq:unknown-subspace-isometric-wasserstein}, every target in its latent
Wasserstein ball $\mathcal Q_\varepsilon^{(k)}(\bar p_0)$ embeds into
$\mathcal Q_\varepsilon(p_0)$. The training sample and both seeds have the
same joint law for all these targets. Restricting the target supremum to
embedded laws and using \eqref{eq:unknown-subspace-reduction-risk} gives
\[
  \sup_{q_0\in\mathcal Q_\varepsilon(p_0)}
  \mathbb E\mathcal L_{q_0}^{(d)}(\hat s_n)
  \ge
  \sup_{\bar q_0\in\mathcal Q_\varepsilon^{(k)}(\bar p_0)}
  \mathbb E\mathcal L_{U_\#\bar q_0}^{(d)}(\hat s_n)
  \ge
  \sup_{\bar q_0\in\mathcal Q_\varepsilon^{(k)}(\bar p_0)}
  \mathbb E\mathcal L_{\bar q_0}^{(k)}(\tilde s_n).
\]
Now take the supremum over latent sources and the infimum over ambient
estimators. Since the induced estimators form a subset of all admissible
randomized latent estimators, the unknown-subspace minimax risk is at least
the $k$-dimensional minimax risk in \eqref{eq:finite-sample-minimax-risk}.
The lower bound in Theorem~\ref{thm:finite-sample-minimax-optimality}, with
$d=k$, now proves \eqref{eq:unknown-subspace-minimax-rate}, with the same
constants $c_0$ and $c_{k,D}$.
\end{proof}

\section{Proofs for Robust Reverse Sampling}
\label{app:robust-reverse-sampling}

Subsection~\ref{app:kl-sampling} establishes the KL comparison with general
initialization.
Subsection~\ref{app:standard-gaussian-initialization} gives the standard
Gaussian initialization bound. Subsection~\ref{app:unknown-subspace-kl}
states and proves Theorem~\ref{thm:unknown-subspace-kl} for the
unknown-subspace estimator, and
Subsection~\ref{app:unknown-subspace-kl-minimax} establishes its
intrinsic-dimensional KL minimax rate.
The robust KL minimax result in Theorem~\ref{thm:robust-kl-minimax} is proved
in Appendix~\ref{app:robust-kl-lower-and-minimax}.

\subsection{KL Comparison with General Initialization}
\label{app:kl-sampling}

Given a possibly randomized score estimator $\hat s_n$ and a terminal law
$\nu_T$, consider the approximate reverse process
\begin{equation}
\begin{aligned}
  \rmd\hat X_t^{\leftarrow}
  &=\frac12\left(\hat X_t^{\leftarrow}
    +2\hat s_n(\hat X_t^{\leftarrow},T-t)\right)\rmd t+\rmd W_t,\\
  \hat X_0^{\leftarrow}&\sim\nu_T,\qquad 0\le t\le T-t_0.
\end{aligned}
\label{eq:estimated-reverse-sde-general-init}
\end{equation}
Conditional on the training sample and any randomness used to construct
$\hat s_n$ and $\nu_T$, denote the law of $\hat X_{T-t_0}^{\leftarrow}$ by
$\hat\mu_{n,t_0}^{\,\nu}$.

\begin{lemma}[Conditional score-to-KL bridge]
\label{lem:conditional-score-to-kl-bridge}
Fix $q_0\in\mathcal P_2(\mathbb R^d)$.  Conditional on the training sample
and the randomness used to construct $\hat s_n$ and $\nu_T$, suppose that the exact shifted reverse
equation and \eqref{eq:estimated-reverse-sde-general-init} admit
non-explosive weak solutions that are unique in law and whose drifts are
jointly locally bounded in space and reverse time.  Then, in the extended-real
sense,
\begin{equation}
  \mathrm{KL}\left(q_{t_0}\middle\|\hat\mu_{n,t_0}^{\,\nu}\right)
  \le
  \mathrm{KL}\left(q_T\middle\|\nu_T\right)
  +\frac12\int_{t_0}^{T}
  \|\hat s_n(\cdot,t)-\nabla\log q_t\|_{L^2(q_t)}^2\,\rmd t.
  \label{eq:score-to-kl-bridge}
\end{equation}
\end{lemma}

\begin{proof}
Condition on the training sample and the randomness used to construct
$\hat s_n$ and $\nu_T$, and set
$S:=T-t_0$.  Let $\mathbf Q$ be the path law of the exact shifted reverse
process initialized from $q_T$ and driven by $\nabla\log q_{T-u}$, and let
$\hat{\mathbf Q}$ be the path law of
\eqref{eq:estimated-reverse-sde-general-init}.  By the standard time-reversal
formula, the time-$u$ marginal of $\mathbf Q$ is $q_{T-u}$; see
\citet{anderson1982reverse} and
\citet[Theorem~4.9]{cattiaux2023time}.

If either term on the right of \eqref{eq:score-to-kl-bridge} is infinite, the
claim is immediate.  Otherwise $q_T\ll\nu_T$.  The standard localized
Girsanov argument on path space applies after disintegrating with respect to
the initial coordinate.  For a related localization argument in diffusion
sampling, see \citet{chen2022sampling}.  Indeed, let $X_u$
be the canonical coordinate, set
$\tau_R:=\inf\{u\in[0,S]:\|X_u\|_2\ge R\}\wedge S$, and define
$X_u^{(R)}:=X_{u\wedge\tau_R}$.  Let $\mathbf Q^{(R)}$ and
$\hat{\mathbf Q}^{(R)}$ denote the laws of $X^{(R)}$ under $\mathbf Q$ and
$\hat{\mathbf Q}$, respectively.  The stopped drift difference is bounded, so
Novikov's condition and the entropy chain rule give
\[
  \mathrm{KL}(\mathbf Q^{(R)}\|\hat{\mathbf Q}^{(R)})
  \le
  \mathrm{KL}(q_T\|\nu_T)
  +\frac12\mathbb E_{\mathbf Q}\int_0^{\tau_R}
  \|\nabla\log q_{T-u}(X_u)-\hat s_n(X_u,T-u)\|_2^2\,\rmd u.
\]
Non-explosion implies that $\tau_R\uparrow S$ almost surely under both path
laws, and hence $\mathbf Q^{(R)}\Rightarrow\mathbf Q$ and
$\hat{\mathbf Q}^{(R)}\Rightarrow\hat{\mathbf Q}$.  Therefore, the lower
semicontinuity of relative entropy and monotone convergence yield
\[
  \mathrm{KL}(\mathbf Q\|\hat{\mathbf Q})
  \le
  \mathrm{KL}(q_T\|\nu_T)
  +\frac12\mathbb E_{\mathbf Q}\int_0^S
  \|\nabla\log q_{T-u}(X_u)-\hat s_n(X_u,T-u)\|_2^2\,\rmd u.
\]
Applying data processing under the endpoint map, using that the time-$u$
marginal of $\mathbf Q$ is $q_{T-u}$, and changing variables $t=T-u$ prove
\eqref{eq:score-to-kl-bridge}.
\end{proof}

\begin{proposition}[OU-compatible KL comparison]
\label{prop:ou-compatible-kl-comparison}
Suppose Assumption~\ref{assum:support} holds and fix $\vartheta>0$. Take
$\hat s_n=\hat s_n^\vartheta$ in
\eqref{eq:estimated-reverse-sde-general-init}. Then the reverse equation has
a unique non-explosive strong solution. For every
$q_0\in\mathcal Q_\varepsilon(p_0)$ and every possibly data-dependent
terminal law $\nu_T$, in the extended-real sense,
\begin{equation}
  \mathrm{KL}\left(
    q_{t_0}\middle\|\hat\mu_{n,t_0}^{\,\nu}
  \right)
  \le
  \mathrm{KL}\left(q_{t_0}\middle\|\hat p_{n,t_0}^{\vartheta}\right)
  +\mathrm{KL}\left(q_T\middle\|\nu_T\right)
  -\mathrm{KL}\left(q_T\middle\|\hat p_{n,T}^{\vartheta}\right)
  \label{eq:ou-compatible-excess-terminal-kl-appendix}
\end{equation}
almost surely.
\end{proposition}

\begin{proof}
Condition on the training sample.  Let
$\gamma_d:=\mathcal N(0,I_d)$ denote the stationary OU law.  Since
$q_0\in\mathcal P_2(\mathbb R^d)$, Lemma~\ref{lem:fixed-time-ou-kl}, applied
with $(p_0,q_0)=(\gamma_d,q_0)$, gives
$\mathrm{KL}(q_{t_0}\|\gamma_d)<\infty$.  The forward OU path laws on
$[t_0,T]$ initialized from $q_{t_0}$ and $\gamma_d$ have the same conditional
path kernel given their initial coordinate.  Hence the entropy chain rule
shows that their path-space relative entropy equals
$\mathrm{KL}(q_{t_0}\|\gamma_d)$.  The finite-entropy time-reversal theorem
therefore realizes the time reversal of the former path law as a
non-explosive weak solution of the exact shifted reverse equation; see
\citet{anderson1982reverse} and
\citet[Theorem~4.9]{cattiaux2023time}.  Moreover, the map
$(x,t)\mapsto q_t(x)$ is strictly positive and smooth on
$\mathbb R^d\times(0,\infty)$.  Thus the exact reverse drift is jointly locally
bounded in space and reverse time and locally Lipschitz in space.  Together
with the non-explosive time-reversed construction, local Lipschitzness gives
uniqueness in law up to the terminal reverse time.

Conditional on the training sample, $\hat p_{n,t}^{\vartheta}$ is a finite
mixture of nondegenerate Gaussians and is therefore strictly positive and
smooth.  The empirical-component scores are
$-(x-a_tX_i)/\rho_t$, while the blanket-component score is $-x/v_t$.  Since
$\|X_i\|_2\le D$, $\rho_t\ge\rho_{t_0}$, and $v_t\ge\rho_{t_0}$, their convex
combination satisfies
\[
  \|\nabla\log\hat p_{n,t}^{\vartheta}(x)\|_2
  \le \frac{\|x\|_2+D}{\rho_{t_0}},
  \qquad t\in[t_0,T].
\]
Consequently, the estimated reverse drift is jointly locally bounded, locally
Lipschitz in space, and of linear growth uniformly in reverse time.  The
reverse equation therefore has a unique non-explosive strong solution, and
both reverse drifts satisfy the conditions of
Lemma~\ref{lem:conditional-score-to-kl-bridge}.

Finally, $q_0,\hat\mu_{n,0}^{\vartheta}\in\mathcal P_2(\mathbb R^d)$, so
Lemma~\ref{lem:relative-de-bruijn-ou} gives
\[
  \frac12\int_{t_0}^{T}
  \|\nabla\log q_t-\nabla\log\hat p_{n,t}^{\vartheta}\|_{L^2(q_t)}^2\,\rmd t
  =
  \mathrm{KL}(q_{t_0}\|\hat p_{n,t_0}^{\vartheta})
  -\mathrm{KL}(q_T\|\hat p_{n,T}^{\vartheta}).
\]
Substituting this identity into \eqref{eq:score-to-kl-bridge} proves
\eqref{eq:ou-compatible-excess-terminal-kl-appendix}.
\end{proof}

\subsection{Standard Gaussian Initialization}
\label{app:standard-gaussian-initialization}

Throughout this subsection, let $\gamma_d:=\mathcal N(0,I_d)$.

\begin{lemma}[Standard Gaussian terminal mismatch]
\label{lem:standard-gaussian-terminal-kl}
Under Assumption~\ref{assum:support}, for every \(T>0\) and
\(q_0\in\mathcal Q_\varepsilon(p_0)\),
\begin{equation}
  \mathrm{KL}(q_T\|\gamma_d)
  \le
  \frac{\varepsilon^2}{\tau_T}
  +\frac{e^{-T}D^2}{1+e^{-T}}
  -\frac d2\log(1-e^{-2T}).
  \label{eq:standard-gaussian-terminal-kl-appendix}
\end{equation}
\end{lemma}

\begin{proof}
Lemma~\ref{lem:three-density-kl}, applied with
\((q,p,r)=(q_T,p_T,\gamma_d)\), gives
\[
  \mathrm{KL}(q_T\|\gamma_d)
  \le
  2\mathrm{KL}(q_T\|p_T)
  +D_2(p_T\|\gamma_d).
\]
By Lemma~\ref{lem:fixed-time-ou-kl},
\[
  2\mathrm{KL}(q_T\|p_T)
  \le \frac{\varepsilon^2}{\tau_T}.
\]

It remains to control \(D_2(p_T\|\gamma_d)\). Write
\(a=a_T=e^{-T/2}\) and \(\rho=\rho_T=1-e^{-T}\), and let
\(k_y(x):=\varphi_\rho(x-ay)\). A direct completion of the square gives
\begin{equation}
  \int_{\mathbb R^d}\frac{k_y(x)^2}{\varphi_1(x)}\,\rmd x
  =
  (1-a^4)^{-d/2}
  \exp\left(\frac{a^2\|y\|_2^2}{1+a^2}\right).
  \label{eq:gaussian-terminal-renyi-integral}
\end{equation}
Since \(p_T(x)=\mathbb E_{Y\sim p_0}k_Y(x)\), Jensen's inequality gives
\[
  p_T(x)^2
  \le
  \mathbb E_{Y\sim p_0}k_Y(x)^2.
\]
Consequently, Tonelli's theorem and \(\|Y\|_2\le D\) yield
\[
  D_2(p_T\|\gamma_d)
  \le \log\mathbb E_{Y\sim p_0}\int\frac{k_Y(x)^2}{\varphi_1(x)}\,\rmd x
  \le \frac{a^2D^2}{1+a^2}-\frac d2\log(1-a^4).
\]
Combining the two estimates and substituting \(a^2=e^{-T}\) proves
\eqref{eq:standard-gaussian-terminal-kl-appendix}.
\end{proof}

\begin{corollary}[Standard Gaussian initialization]
\label{cor:standard-gaussian-kl-sampling}
Under Assumption~\ref{assum:support}, set $\vartheta=D/\sqrt d$, take
$\hat s_n=\hat s_n^\vartheta$ in
\eqref{eq:estimated-reverse-sde-general-init}, and initialize the reverse
process from $\nu_T=\gamma_d$. If $\tau_0\le1$ and $T\ge\log2$, then
\begin{equation}
\begin{aligned}
  \sup_{q_0\in\mathcal Q_\varepsilon(p_0)}
  \mathbb E\,\mathrm{KL}\left(
    q_{t_0}\middle\|\hat\mu_{n,t_0}^{\,\gamma}
  \right)
  \le{}&
  \frac{C_{d,D}}{2(n+1)\tau_0^{d/2}}
  +\frac{\varepsilon^2}{\tau_0}
  +(D^2+2\varepsilon^2)e^{-T}
  +de^{-2T}.
\end{aligned}
\label{eq:standard-gaussian-kl-sampling-bound}
\end{equation}
\end{corollary}

\begin{proof}
With \(\vartheta=D/\sqrt d\), taking \(\nu_T=\gamma_d\) in
\eqref{eq:ou-compatible-excess-terminal-kl-appendix} and dropping its final
nonpositive term gives
\[
  \mathrm{KL}\left(
    q_{t_0}\middle\|\hat\mu_{n,t_0}^{\,\gamma}
  \right)
  \le
  \mathrm{KL}\left(q_{t_0}\middle\|\hat p_{n,t_0}^{\vartheta}\right)
  +\mathrm{KL}(q_T\|\gamma_d).
\]
Taking expectations and the supremum, 
\eqref{eq:self-consistent-kl-sampling-bound} and
Lemma~\ref{lem:standard-gaussian-terminal-kl} yield
\begin{align*}
  \sup_{q_0\in\mathcal Q_\varepsilon(p_0)}
  \mathbb E\,\mathrm{KL}\left(
    q_{t_0}\middle\|\hat\mu_{n,t_0}^{\,\gamma}
  \right)
  \le{}&
  \frac{C_{d,D}}{2(n+1)\tau_0^{d/2}}
  +\frac{\varepsilon^2}{\tau_0}
  +\frac{\varepsilon^2}{\tau_T}
  +\frac{e^{-T}D^2}{1+e^{-T}}
  -\frac d2\log(1-e^{-2T}).
\end{align*}
Since \(T\ge\log2\),
\[
  \frac1{\tau_T}\le2e^{-T},
  \qquad
  \frac{e^{-T}D^2}{1+e^{-T}}\le D^2e^{-T},
  \qquad
  -\frac d2\log(1-e^{-2T})\le de^{-2T},
\]
where the last inequality uses \(-\log(1-x)\le2x\) for \(0\le x\le1/2\).
Substitution proves \eqref{eq:standard-gaussian-kl-sampling-bound}.
\end{proof}

\subsection{Intrinsic-Dimensional KL Upper Bound}
\label{app:unknown-subspace-kl}

\begin{theorem}
\label{thm:unknown-subspace-kl}
Under the assumptions and construction of
Theorem~\ref{thm:unknown-subspace-upper}, suppose that $k\ge2$ and
$\tau_0\le1$. Then
\begin{equation}
  \sup_{q_0\in\mathcal Q_\varepsilon(p_0)}
  \mathbb E\left[
    \mathrm{KL}\,\left(q_{t_0}\middle\|
    \hat p_{n,t_0}^{\mathrm{sub}}\right)
  \right]
  \le
  \frac{\tilde C_{k,D}}{2n\tau_0^{k/2}}
  +\frac{\varepsilon^2}{\tau_0}.
  \label{eq:unknown-subspace-kl-rate}
\end{equation}
Here $\tilde C_{k,D}$ is the constant in
Theorem~\ref{thm:unknown-subspace-upper}.
\end{theorem}

\begin{proof}
Under the assumptions and construction of
Theorem~\ref{thm:unknown-subspace-upper}, taking the supremum over
$q_0\in\mathcal Q_\varepsilon(p_0)$ in
\eqref{eq:unknown-subspace-mean-kl} yields
\begin{equation}
  \sup_{q_0\in\mathcal Q_\varepsilon(p_0)}
  \mathbb E\left[
    \mathrm{KL}\left(q_{t_0}\middle\|
    \hat p_{n,t_0}^{\mathrm{sub}}\right)
  \right]
  \le
  \log\left(1+\frac{M_{\vartheta_k,t_0}-1}{n_2+1}\right)
  +\frac{\varepsilon^2}{\tau_0}
  +\frac{kD^2}{(n_1+k)\tau_0}.
  \label{eq:unknown-subspace-kl-finite}
\end{equation}
For $k\ge2$ and $\tau_0\le1$, simplifying
\eqref{eq:unknown-subspace-kl-finite} as in the derivation of
\eqref{eq:unknown-subspace-rate} gives
\eqref{eq:unknown-subspace-kl-rate}.

We also verify that self-consistent initialization yields the conditional
output law $\hat p_{n,t_0}^{\mathrm{sub}}$. Condition on the full sample and
consider the reverse process \eqref{eq:estimated-reverse-sde-general-init}
with $\hat s_n=\hat s_n^{\mathrm{sub}}$ and self-consistent initialization
$\nu_T=\hat p_{n,T}^{\mathrm{sub}}$. By
\eqref{eq:unknown-subspace-mixture-density}, the score is a convex combination
of the component scores $-(x-a_tY_i)/\rho_t$ and $-\Sigma_t^{-1}x$, where
$\Sigma_t=\rho_tI_d+a_t^2\vartheta_k^2\hat\Pi$. Since
$\Sigma_t\succeq\rho_tI_d$ and $\|Y_i\|_2\le D$,
\[
  \|\hat s_n^{\mathrm{sub}}(x,t)\|_2
  \le\frac{\|x\|_2+D}{\rho_{t_0}},
  \qquad t\in[t_0,T].
\]
The mixture density is strictly positive and jointly smooth in $(x,t)$ for
$t>0$. Thus, on each compact spatial set, its score has bounded spatial
derivatives and is locally Lipschitz, uniformly over $t\in[t_0,T]$.
Together with the linear-growth bound, this gives a unique non-explosive
strong solution of \eqref{eq:estimated-reverse-sde-general-init} with the
stated score and initialization.

To identify the output law, start a forward OU process from
$\hat\mu_{n,0}^{\mathrm{sub}}$ and restrict it to $[t_0,T]$.
The moment bound \eqref{eq:unknown-subspace-estimator-moment} and
Lemma~\ref{lem:fixed-time-ou-kl} give
\(\mathrm{KL}(\hat p_{n,t_0}^{\mathrm{sub}}\|
\mathcal N(0,I_d))<\infty\). On $[t_0,T]$, this process and the stationary
OU process have the same conditional path law given their position at
$t_0$. By the entropy chain rule, their relative entropy on path space
equals this finite initial KL divergence.
By the finite-entropy time-reversal theorem
\citep[Theorem~4.9]{cattiaux2023time}, the reversed OU process is a weak
solution of the estimated reverse equation with marginals
$\hat p_{n,T-t}^{\mathrm{sub}}$. By uniqueness in law, the strong solution
has the same marginals. Its output law at reverse time $T-t_0$ is therefore
$\hat p_{n,t_0}^{\mathrm{sub}}$.
\end{proof}

\subsection{Intrinsic-Dimensional KL Minimax Rate}
\label{app:unknown-subspace-kl-minimax}

For the same source class $\mathcal P_{k,D}$, define the minimax positive-time
KL risk by
\begin{equation}
  \mathfrak K_{n,\varepsilon}^{\star}(k,D;t_0)
  :=\inf_{\hat\mu_n}\sup_{p_0\in\mathcal P_{k,D}}
  \sup_{q_0\in\mathcal Q_\varepsilon(p_0)}
  \mathbb E\left[\mathrm{KL}(q_{t_0}\|\hat\mu_n)\right],
  \label{eq:unknown-subspace-kl-minimax-risk}
\end{equation}
with the same estimator class as in
Subsection~\ref{subsec:positive-time-kl-sampling}.

\begin{corollary}
\label{cor:unknown-subspace-kl-minimax}
Fix $k\ge2$ and $D>0$, and let $d\ge k$ be arbitrary. Let
$c_0\in(0,\log2]$ and $\tilde c_{k,D}>0$ be as in
Theorem~\ref{thm:robust-kl-minimax} with dimension $k$. If $n\ge2$ and
\[
  0<t_0\le c_0,\qquad n\tau_0^{k/2}\ge1,\qquad
  0\le\varepsilon\le\sqrt{\tau_0},
\]
then
\begin{equation}
  \tilde c_{k,D}\left[\frac{1}{n\tau_0^{k/2}}+\frac{\varepsilon^2}{\tau_0}\right]
  \le\mathfrak K_{n,\varepsilon}^{\star}(k,D;t_0)
  \le\frac{\tilde C_{k,D}}{2n\tau_0^{k/2}}+\frac{\varepsilon^2}{\tau_0}.
  \label{eq:unknown-subspace-kl-minimax-rate}
\end{equation}
Here $\tilde C_{k,D}$ is the constant in
Theorem~\ref{thm:unknown-subspace-upper}.
\end{corollary}

\begin{proof}
The upper bound follows from Theorem~\ref{thm:unknown-subspace-kl}, since
$\tau_0\le1$. For the lower bound, use the embedding and notation from the
proof of Corollary~\ref{cor:unknown-subspace-minimax}. Given any ambient
distribution estimator $\hat\mu_n$, apply it to the embedded observations
and define $\tilde\mu_n:=(U^\top)_\#\hat\mu_n$, retaining any internal
randomization. Since $(U^\top)_\#q_{t_0}=\bar q_{t_0}$, the data processing
inequality gives
\[
  \mathrm{KL}(q_{t_0}\|\hat\mu_n)
  \ge\mathrm{KL}\!\left((U^\top)_\#q_{t_0}\,\middle\|\,
  (U^\top)_\#\hat\mu_n\right)
  =\mathrm{KL}(\bar q_{t_0}\|\tilde\mu_n).
\]
The same reduction bounds this minimax risk below by its $k$-dimensional
counterpart. The lower bound follows from Theorem~\ref{thm:robust-kl-minimax}
with dimension $k$.
\end{proof}

\section{Neural Approximation of the Gaussian-Blanket Score}
\label{app:blanket-neural-approximation}

We approximate the isotropic Gaussian-blanket score from
Subsection~\ref{subsec:finite-sample-score}, uniformly over
$\mathbb R^d\times[t_0,T]$.  We separate the known score of the Gaussian
blanket from the sample-dependent correction, retain the former exactly,
and approximate the latter by a ReLU network.
The construction combines spatial localization with elementary ReLU
arithmetic and gives explicit width, depth, and sparsity bounds, together
with a time-integrated bound on the compatibility defect, uniformly over
samples in the support ball.
The final subsection transfers this approximation guarantee to robust score
risk via Corollary~\ref{cor:ou-score-defect}.

\subsection{Uniform Approximation Guarantee}
\label{gbnn:subsec-guarantee}

Fix
\[
 \Theta=(d,D,\vartheta,t_0,T),\qquad
 d\geq 1,\quad D,\vartheta>0,\quad 0<t_0<T<\infty,
\]
and let $X_1,\ldots,X_n\in\mathcal B(0,D)$, with $n\ge1$. Write
$X_{1:n}:=(X_1,\ldots,X_n)$ and initially regard the sample as deterministic.
Retain $a_t$, $\rho_t$, $\tau_0$, and the Gaussian
density $\varphi_v$ from Subsection~\ref{subsec:finite-sample-score}, and write
\begin{equation}
\begin{gathered}
 v_t=\rho_t+a_t^2\vartheta^2,\qquad
 M:=M_{\vartheta,t_0}
   =\exp\left(\frac{D^2}{2\vartheta^2}\right)
      \left(1+\frac{\vartheta^2}{\tau_0}\right)^{d/2}.
\end{gathered}
\label{gbnn:schedules}
\end{equation}
Recall the density estimator \eqref{eq:blanket-density},
\begin{equation}
 \hat p_{n,t}^{\vartheta}(x)
  =\frac{M\varphi_{v_t}(x)+
       \sum_{i=1}^n\varphi_{\rho_t}(x-a_tX_i)}{n+M}.
\label{gbnn:estimator}
\end{equation}
Since the Gaussian blanket has score
$\nabla_x\log\varphi_{v_t}(x)=-x/v_t$, we only need to approximate
the correction
\[
 G_{n,t}(x):=\nabla_x\log
       \frac{\hat p_{n,t}^{\vartheta}(x)}{\varphi_{v_t}(x)}
   =\nabla_x\log\hat p_{n,t}^{\vartheta}(x)+\frac{x}{v_t}.
\]
Lemma~\ref{gbnn:lem-localization} shows that this correction decays at
spatial infinity, allowing its uniform approximation to be reduced to
a bounded region.

We use feedforward ReLU networks with layer dimensions
$m_0,\ldots,m_L$ and affine maps
$\mathcal A_j(z)=A_jz+b_j$, where
$A_j\in\mathbb R^{m_j\times m_{j-1}}$ and
$b_j\in\mathbb R^{m_j}$.  The network has the form
\[
 \Phi=\mathcal A_L\circ\operatorname{ReLU}\circ\cdots
      \circ\operatorname{ReLU}\circ\mathcal A_1,
\]
with $\operatorname{ReLU}(z)=z_+$ applied componentwise between
consecutive affine maps.  Its width, depth, and sparsity are
\[
 W(\Phi):=\max_{1\leq j<L}m_j,\qquad
 L(\Phi):=L,\qquad
 S(\Phi):=\sum_{j=1}^L\bigl(\|A_j\|_0+\|b_j\|_0\bigr),
\]
where $\|\cdot\|_0$ counts nonzero entries.  Depth includes the output
affine layer; for a purely affine network ($L=1$), set $W(\Phi)=0$.
Weights and biases may be arbitrary real numbers.
For the network approximating $G_{n,t}$, $m_0=d+1$ and $m_L=d$;
the complexity bounds below concern this correction network.

All vector norms are Euclidean, and all logarithms are natural unless a
base is indicated.  The $L^\infty(\mathbb R^d)$ norm denotes the essential supremum
with respect to Lebesgue measure; for continuous vector fields,
$\|f\|_{L^\infty(\mathbb R^d)}
=\sup_{x\in\mathbb R^d}\|f(x)\|_2$.
Constants denoted by $C_\Theta$ are finite and depend only on $\Theta$;
their values may change from line to line.

\begin{proposition}[Uniform neural score approximation]
\label{gbnn:prop-approximation}
For every $n\geq1$, $X_{1:n}\in\mathcal B(0,D)^n$, and
$\eta\in(0,1/2)$, there exists a ReLU network
$\Phi_{X_{1:n},\eta}:\mathbb R^{d+1}\to\mathbb R^d$ such that
\begin{equation}
 \sup_{\substack{x\in\mathbb R^d\\t\in[t_0,T]}}
    \|\Phi_{X_{1:n},\eta}(x,t)-G_{n,t}(x)\|_2\leq\eta.
\label{gbnn:uniform-approximation}
\end{equation}
Writing $\Phi=\Phi_{X_{1:n},\eta}$ and
$\ell_{n,\eta}=1+\log((n+1)/\eta)$, we have
\begin{equation}
 W(\Phi)\leq C(n+d),\quad
 L(\Phi)\leq C_\Theta\ell_{n,\eta}^{\,2},\quad
 S(\Phi)\leq C_\Theta\bigl[nd+(n+d)\ell_{n,\eta}^{\,2}\bigr],
\label{gbnn:complexity}
\end{equation}
where $C$ is absolute.
For fixed $\Theta,n,\eta$, the networks admit a common architecture
and set of allowed parameter positions, with parameters continuous in
$X_{1:n}$.
\end{proposition}

\subsection{Ratio Representation and Localization}
\label{gbnn:subsec-localization}

To obtain a representation with a denominator bounded below by one, we
compare each sample kernel with the Gaussian blanket.  Define
\begin{equation}
 R_{i,t}(x):=\frac{\varphi_{\rho_t}(x-a_tX_i)}
                       {M\varphi_{v_t}(x)},\qquad
 U_{n,t}(x):=\sum_{i=1}^n R_{i,t}(x).
\label{gbnn:ratio-definitions}
\end{equation}
By Lemma~\ref{lem:gaussian-blanket-domination}, for every
$x\in\mathbb R^d$ and $t\in[t_0,T]$,
\begin{equation}
 0<R_{i,t}(x)\leq1\quad(i=1,\ldots,n),\qquad
 1\leq1+U_{n,t}(x)\leq n+1.
\label{gbnn:ratio-bounds}
\end{equation}
Factoring the density gives
\[
 \hat p_{n,t}^{\vartheta}
   =\frac{M\varphi_{v_t}}{n+M}(1+U_{n,t}),\qquad
 G_{n,t}=\nabla\log(1+U_{n,t})
        =\frac{\nabla U_{n,t}}{1+U_{n,t}}.
\]
Thus the Gaussian density cancels from the score correction, leaving a
denominator bounded below by one.  To compute the numerator, expand the
Gaussian log ratio as
\begin{equation}
 \log R_{i,t}(x)
  =c_t-\frac{\lambda_t}{2}\|x\|_2^2
       +\gamma_t x^\top X_i-\beta_t\|X_i\|_2^2,
\label{gbnn:log-ratio}
\end{equation}
where the four time-dependent coefficients are
\[
\begin{gathered}
 c_t:=\frac d2\log\frac{v_t}{\rho_t}-\log M,\qquad
 \lambda_t:=\rho_t^{-1}-v_t^{-1}
          =\frac{a_t^2\vartheta^2}{\rho_t v_t},\\
 \gamma_t:=\frac{a_t}{\rho_t},\qquad
 \beta_t:=\frac{a_t^2}{2\rho_t}.
\end{gathered}
\]
Differentiation yields
$\nabla R_{i,t}(x)=R_{i,t}(x)(\gamma_tX_i-\lambda_tx)$.
Introducing the weighted sample sum
$V_{n,t}(x):=\sum_{i=1}^nR_{i,t}(x)X_i$, we therefore obtain
\begin{equation}
 G_{n,t}(x)
   =\frac{\gamma_t V_{n,t}(x)-\lambda_t xU_{n,t}(x)}
          {1+U_{n,t}(x)}.
\label{gbnn:stable-quotient}
\end{equation}
The next lemma shows that the correction is uniformly small outside a
ball, providing the localization needed for the network construction.

\begin{lemma}[Localization of the score correction]
\label{gbnn:lem-localization}
For every $\eta\in(0,1/2)$, there is a radius
$R=R_{n,\eta}\geq1$, independent of the sample locations, such that
\begin{equation}
 R\leq C_\Theta\sqrt{1+\log\frac{n+1}{\eta}},
 \qquad
 \sup_{\substack{\|x\|_2\geq R\\t\in[t_0,T]}}
      \|G_{n,t}(x)\|_2\leq\frac{\eta}{4}.
\label{gbnn:localized-tail}
\end{equation}
\end{lemma}

\begin{proof}
Applying the completed-square identity
\eqref{eq:gaussian-blanket-completed-square} with $y=X_i$ gives
\begin{equation}
 R_{i,t}(x)=R_{i,t}(z_{i,t})
       \exp\left(-\frac{\lambda_t}{2}\|x-z_{i,t}\|_2^2\right),
 \qquad z_{i,t}=\frac{v_t}{a_t\vartheta^2}X_i.
\label{gbnn:completed-square}
\end{equation}
Here $0<R_{i,t}(z_{i,t})\leq1$ by \eqref{gbnn:ratio-bounds}, and
differentiation gives
\[
 \nabla R_{i,t}(x)=-\lambda_t(x-z_{i,t})R_{i,t}(x).
\]

For the tail estimate, set
\begin{equation}
 \lambda_-:=\min_{t\in[t_0,T]}\lambda_t,\qquad
 C_*:=D\max_{t\in[t_0,T]}\frac{v_t}{a_t\vartheta^2},\qquad
 B_*:=\max\left\{1,\sqrt{\frac2e
                   \max_{t\in[t_0,T]}\lambda_t}\right\}.
\label{gbnn:tail-constants}
\end{equation}
These constants depend only on $\Theta$, and $\lambda_->0$ because
$\lambda_t$ is continuous and strictly positive on $[t_0,T]$.
Set $u=\sqrt{\lambda_t}\|x-z_{i,t}\|_2$.  The inequality
$u e^{-u^2/2}\leq\sqrt{2/e}\,e^{-u^2/4}$ for $u\geq0$ then gives
\[
\begin{aligned}
 \|\nabla R_{i,t}(x)\|_2
  &=\lambda_t\|x-z_{i,t}\|_2 R_{i,t}(z_{i,t})
       e^{-\lambda_t\|x-z_{i,t}\|_2^2/2}\\
  &\leq\sqrt{\lambda_t}\,u e^{-u^2/2}
   \leq\sqrt{\frac{2\lambda_t}{e}}\,e^{-u^2/4}\\
  &\leq B_*\,
        e^{-\lambda_-\|x-z_{i,t}\|_2^2/4}.
\end{aligned}
\]
Using $G_{n,t}=\nabla U_{n,t}/(1+U_{n,t})$ and $1+U_{n,t}\geq1$,
the triangle inequality and $\|z_{i,t}\|_2\leq C_*$ give, for
$\|x\|_2\geq C_*$,
\begin{equation}
 \|G_{n,t}(x)\|_2
  \leq n B_*\,
       e^{-\lambda_-(\|x\|_2-C_*)^2/4}.
\label{gbnn:tail-bound}
\end{equation}
Choose
\begin{equation}
 R_{n,\eta}:=\max\left\{1,\ C_*+
     \sqrt{\frac4{\lambda_-}\log\frac{4nB_*}{\eta}}\right\}.
\label{gbnn:radius}
\end{equation}
Substitution into \eqref{gbnn:tail-bound} proves the tail bound in
\eqref{gbnn:localized-tail}.  Since
$\log(4nB_*/\eta)\leq(1+\log(4B_*))\ell_{n,\eta}$ and
$\ell_{n,\eta}\geq1$, the radius bound in this lemma holds with
\[
 C_\Theta=1+C_*+
       2\sqrt{\frac{1+\log(4B_*)}{\lambda_-}},
\]
which depends only on $\Theta$.
\end{proof}

\subsection{Elementary ReLU Arithmetic}
\label{gbnn:subsec-arithmetic}

For $a<b$, the clipping map
\[
 \operatorname{clip}_{[a,b]}(u)=a+(u-a)_+-(u-b)_+
\]
is represented exactly by a ReLU network of constant width and depth.
Clipping cannot increase the distance to a point in $[a,b]$.

We use multiplication, negative-exponential, and reciprocal networks to
approximate the score correction in \eqref{gbnn:stable-quotient}.
The logarithm, evaluated on a fixed positive interval, is used only to
construct the time coefficient $c_t$.
The square and multiplication modules build on
\citet[Propositions~2--3]{yarotsky2017error}.
The remaining modules are constructed from these building blocks using
repeated squaring and elementary series expansions.

\begin{lemma}[Elementary ReLU arithmetic modules]
\label{gbnn:lem-arithmetic}
Fix $0<\xi<1/2$.  The following ReLU networks exist with width at most
an absolute constant $C$.  In each item, $L$ and $S$ denote the depth
and sparsity of the stated network.
\begin{enumerate}
\item[(i)] Multiplication: for $B\geq1$,
$\operatorname{Mult}_{B,\xi}:\mathbb R^2\to\mathbb R$ satisfies
\[
 \sup_{|u|,|v|\leq B}
 |\operatorname{Mult}_{B,\xi}(u,v)-uv|\leq\xi,\qquad
 L,S\leq C(1+\log B+\log\xi^{-1}).
\]

\item[(ii)] Negative exponential: for $B\geq1$,
$\operatorname{Exp}_{B,\xi}:\mathbb R\to[0,1]$ satisfies
\[
 \sup_{z\in[0,B]}|\operatorname{Exp}_{B,\xi}(z)-e^{-z}|\leq\xi,\qquad
 L,S\leq C(1+\log B+\log\xi^{-1})^2.
\]

\item[(iii)] Reciprocal: for $B\geq1$,
$\operatorname{Rec}_{B,\xi}:\mathbb R\to[0,1]$ satisfies
\[
 \sup_{z\in[1,B]}
 |\operatorname{Rec}_{B,\xi}(z)-z^{-1}|\leq\xi,\qquad
 L,S\leq C(1+\log B+\log\xi^{-1})^2.
\]

\item[(iv)] Logarithm: for $0<b<B<\infty$,
$\operatorname{Log}_{b,B,\xi}:\mathbb R\to\mathbb R$ satisfies
\[
 \sup_{z\in[b,B]}
 |\operatorname{Log}_{b,B,\xi}(z)-\log z|\leq\xi,\qquad
 L,S\leq C\frac{B}{b}
       \left(1+\log\frac{B}{b}+\log\xi^{-1}\right)^2.
\]
\end{enumerate}
\end{lemma}

\begin{proof}
\emph{(i) Multiplication.}
We first construct a square network.  Let $S_m$ be the piecewise-linear
interpolant of $u^2$ on the dyadic grid
$\{k2^{-m}:0\leq k\leq2^m\}$, so $S_0(u)=u$ on $[0,1]$.
When refining from level $m-1$ to level $m$, the values at old nodes
remain unchanged.  At the midpoint of each coarse interval $[a,b]$,
the old interpolant exceeds $u^2$ by
\[
 S_{m-1}\left(\frac{a+b}{2}\right)
       -\left(\frac{a+b}{2}\right)^2
       =\frac{(b-a)^2}{4}=4^{-m}.
\]
To implement this correction, use the tent map and its iterates
\[
 g(u)=2u_+-4(u-1/2)_++2(u-1)_+,
 \qquad g_j=g^{\circ j}.
\]
The function $g_m$ is zero at the old nodes, one at the new midpoints,
and linear between consecutive nodes of the refined grid.  Hence
\[
 S_m=S_{m-1}-4^{-m}g_m,
 \qquad S_m(u)=u-\sum_{j=1}^m4^{-j}g_j(u).
\]
On each refined interval $[a,b]$, the interpolation error is
$(u-a)(b-u)$, so
\[
 \sup_{u\in[0,1]}|S_m(u)-u^2|\leq4^{-m-1}.
\]
Keeping the current iterate $g_j$ and the partial sum gives constant
width and $O(m+1)$ depth and sparsity.
For later use, denote by $\operatorname{Sq}_{\xi_*}:
\mathbb R\to[0,1]$ the network obtained by choosing
the smallest integer $m\geq0$ with $4^{-m-1}\leq\xi_*$ and clipping
$S_m$ to $[0,1]$.
For $0<\xi_*<1/2$, it approximates $u^2$ on $[0,1]$ with error
at most $\xi_*$ and depth and sparsity $O(1+\log\xi_*^{-1})$.

For $|u|,|v|\leq B$, use
\[
 uv=B^2\left[
       \left(\frac{|u+v|}{2B}\right)^2
       -\left(\frac{|u-v|}{2B}\right)^2\right].
\]
The absolute values are exact ReLU operations.  Replacing the two
squares by parallel copies of $\operatorname{Sq}_{\xi/(2B^2)}$
gives error at most $\xi$ and the claimed complexity.

\emph{(ii) Negative exponential.}
We approximate $e^{-z}$ on $[0,B]$ by $(1-z/N)^N$ and implement this
power by repeated squaring.  Let $N=2^k$ be the smallest power of two with
\[
 N\geq\max\{2B,2B^2/\xi\}.
\]
For $0\leq w\leq1/2$, the logarithm series gives
\[
 0\leq-\log(1-w)-w
   =\sum_{r=2}^\infty\frac{w^r}{r}
   \leq\frac{w^2}{2(1-w)}\leq w^2.
\]
Taking $w=z/N$ yields
$-z-z^2/N\leq N\log(1-z/N)\leq-z$.
Exponentiating and using $1-e^{-r}\leq r$ for $r\geq0$, we obtain
\begin{equation}
 0\leq e^{-z}-(1-z/N)^N
   \leq e^{-z}(1-e^{-z^2/N})
   \leq B^2/N\leq\xi/2.
\label{gbnn:exp-bias}
\end{equation}

It remains to approximate the power by a network.  The exact recursion
$y_0=1-z/N$, $y_{j+1}=y_j^2$ stays in $[0,1]$ and ends at
$y_k=(1-z/N)^N$.  With $\xi_*=\xi/(2N)$, use the square network
from part (i) to define
\[
 \tilde y_0=\operatorname{clip}_{[0,1]}(y_0),\qquad
 \tilde y_{j+1}=\operatorname{Sq}_{\xi_*}(\tilde y_j),
 \qquad 0\leq j<k.
\]
All approximate iterates also lie in $[0,1]$, and the initial clipping
is inactive for $z\in[0,B]$.  Since
$|u^2-v^2|\leq2|u-v|$ on $[0,1]$, the square approximation gives
\[
 |\tilde y_{j+1}-y_{j+1}|
       \leq2|\tilde y_j-y_j|+\xi_*,
 \qquad
 |\tilde y_k-y_k|\leq(N-1)\xi_*\leq\xi/2.
\]
Define $\operatorname{Exp}_{B,\xi}(z):=\tilde y_k$.
Together with \eqref{gbnn:exp-bias}, the last estimate proves the
required error bound.  The initial clipping and the range of
$\operatorname{Sq}_{\xi_*}$ ensure that the output lies in $[0,1]$
for every real input.
There are $k=O(1+\log B+\log\xi^{-1})$ squaring stages, each with depth
and sparsity $O(1+\log\xi_*^{-1})$ of the same order.  Only the current
iterate is needed, so the width remains bounded.  This proves (ii).

\emph{(iii) Reciprocal.}
For $z\in[1,B]$, set $a=1-z/B$.  The geometric series gives
\[
 z^{-1}=\frac1B\sum_{r=0}^\infty a^r.
\]
We approximate this series by partial sums whose number of terms
doubles at each stage.  Specifically, set
\[
 p_j=a^{2^j},\qquad
 h_j=\frac1B\sum_{r=0}^{2^j-1}a^r
    =\frac{1-a^{2^j}}{z}.
\]
These quantities satisfy the recursion
\begin{equation}
 p_0=a,\quad h_0=B^{-1},\qquad
 p_{j+1}=p_j^2,\quad h_{j+1}=h_j(1+p_j).
\label{gbnn:rec-exact}
\end{equation}
Since $z\geq1$ and $0\leq a\leq1-1/B$, we have $p_j,h_j\in[0,1]$
and the truncation error satisfies
\begin{equation}
 |h_j-z^{-1}|=\frac{a^{2^j}}{z}
   \leq(1-1/B)^{2^j}\leq e^{-2^j/B}.
\label{gbnn:rec-bias}
\end{equation}
Choose
\[
 k=\left\lceil\log_2\left(B\log\frac2\xi\right)\right\rceil,
 \qquad \xi_*=\frac{\xi}{2\cdot3^k}.
\]
Starting from $\tilde p_0=\operatorname{clip}_{[0,1]}(a)$ and
$\tilde h_0=B^{-1}$, define
\[
 \tilde p_{j+1}=\operatorname{Sq}_{\xi_*}(\tilde p_j),\qquad
 \tilde h_{j+1}
    =\operatorname{clip}_{[0,1]}
       \bigl(\operatorname{Mult}_{2,\xi_*}
                    (\tilde h_j,1+\tilde p_j)\bigr).
\]
Both updates use the preceding pair of iterates; initial clipping is
inactive for $z\in[1,B]$.
Define $\operatorname{Rec}_{B,\xi}(z):=\tilde h_k(z)$.
For fixed $z\in[1,B]$, set
$e_j:=\max\{|\tilde p_j-p_j|,|\tilde h_j-h_j|\}$,
with the dependence on $z$ suppressed.  Then $e_0=0$ and
\[
 e_{j+1}\leq3e_j+\xi_*,
 \qquad e_k\leq\frac{3^k-1}{2}\xi_*\leq\xi/4.
\]
Indeed, the square map is $2$-Lipschitz on $[0,1]$, while
\[
 |\tilde h(1+\tilde p)-h(1+p)|
       \leq2|\tilde h-h|+|\tilde p-p|
\]
for $p,\tilde p,h,\tilde h\in[0,1]$.
The choice of $k$ makes \eqref{gbnn:rec-bias} at most $\xi/2$.
Hence the total error is less than $\xi$.
The output is clipped to $[0,1]$, and only a fixed number of scalars
is carried between stages.  Since
\[
 k\leq C(1+\log B+\log\xi^{-1}),\qquad
 1+\log\xi_*^{-1}\leq C(1+\log B+\log\xi^{-1}),
\]
part (i) proves the claimed complexity.

\emph{(iv) Logarithm.}
Let $u=1-z/B$ and $q=1-b/B\in(0,1)$.  For $z\in[b,B]$,
\[
 \log z=\log B-\sum_{j=1}^\infty\frac{u^j}{j},
 \qquad 0\leq u\leq q.
\]
Set $\kappa:=B/b>1$ and choose
\[
 m=\left\lceil\kappa\log\frac{2\kappa}{\xi}\right\rceil.
\]
Since $q=1-\kappa^{-1}\leq e^{-1/\kappa}$, the omitted tail satisfies
\[
 \sum_{j=m+1}^\infty\frac{u^j}{j}
 \leq\frac{q^{m+1}}{1-q}
 \leq\kappa e^{-(m+1)/\kappa}\leq\xi/2.
\]
Approximate the powers recursively, starting at $\tilde v_0=1$,
by
\[
 \tilde v_j
   =\operatorname{clip}_{[0,1]}
       \bigl(\operatorname{Mult}_{1,\xi_*}
           (\tilde v_{j-1},u)\bigr),
 \qquad \xi_*=\frac{\xi}{2m}.
\]
Since $u,\tilde v_{j-1}\in[0,1]$ and clipping cannot increase the
distance to $u^j$, the multiplication error gives
\[
 |\tilde v_j-u^j|
   \leq\xi_*+u|\tilde v_{j-1}-u^{j-1}|
   \leq\xi_*+|\tilde v_{j-1}-u^{j-1}|.
\]
Starting from $\tilde v_0=1$, induction therefore gives
$|\tilde v_j-u^j|\leq j\xi_*$.  Define the output network by
\[
 \operatorname{Log}_{b,B,\xi}(z)
    :=\log B-\sum_{j=1}^m\frac{\tilde v_j(z)}{j}.
\]
Its arithmetic error relative to the exact truncated series is at most
$\sum_{j=1}^m j\xi_*/j=m\xi_*=\xi/2$.
Adding the truncation error gives total error at most $\xi$.
Only $u$, the current power, and the running sum need to be stored.
Thus the width is bounded.  Since
$m\leq C\kappa(1+\log\kappa+\log\xi^{-1})$, the depth and sparsity satisfy
\[
 L,S\leq C m(1+\log(m/\xi))
     \leq C\kappa(1+\log\kappa+\log\xi^{-1})^2,
\]
where $C$ is universal.
\end{proof}

\subsection{Proof of the Uniform Approximation Guarantee}
\label{gbnn:subsec-proof}

\begin{proofof}{Proposition~\ref{gbnn:prop-approximation}}
The network implements the ratio formula \eqref{gbnn:stable-quotient}
through the following computations:
\[
\begin{aligned}
 (x,t)&\longmapsto
 \bigl(\|x\|_2^2,(x^\top X_i)_{i=1}^n,
       c_t,\lambda_t,\gamma_t,\beta_t\bigr)\\
 &\longmapsto (R_{i,t}(x))_{i=1}^n
 \longmapsto \bigl(U_{n,t}(x),V_{n,t}(x)\bigr)
 \longmapsto G_{n,t}(x).
\end{aligned}
\]
The spatial input and time coefficients are carried forward for use in
the final quotient.  We approximate these operations on a cube,
constructing each ratio approximation $\tilde R_i$ to take values in
$[0,1]$.  The resulting denominator $1+\sum_{i=1}^n\tilde R_i$ then lies
in $[1,n+1]$, allowing us to apply the reciprocal approximation.
The tail estimate then extends the approximation to the whole space.

\medskip
\noindent\emph{Step 1: Choose the internal accuracy and approximate the time coefficients.}
Choose $R=R_{n,\eta}$ from \eqref{gbnn:radius}.  Throughout the local
construction, $(x,t)\in[-R,R]^d\times[t_0,T]$.
Let
\begin{equation}
\begin{gathered}
 K:=1+\max_{t\in[t_0,T]}
          \max\{|c_t|,\lambda_t,\gamma_t,\beta_t\},
 \qquad Q:=1+\sqrt d\,R+D,\\
 \xi:=\frac{\eta}{96(n+1)K^2Q^3},
 \qquad B_{\mathrm{op}}:=8K(n+1)Q^2.
\end{gathered}
\label{gbnn:internal-parameters}
\end{equation}
Here $K$ bounds the time coefficients and $Q$ controls the spatial
scales.  We use $\xi$ as the internal accuracy and $B_{\mathrm{op}}$ as
a common input bound for the subsequent multiplication modules; the
choice of $\xi$ accounts for the accumulated error in Step~3.
In particular, $K,Q\geq1$, $Q\geq\sqrt d$, $D+\|x\|_2\leq Q$, and
$0<\xi<1/2$.

We construct a constant-width network with outputs
$\tilde c_t,\tilde\lambda_t,\tilde\gamma_t,\tilde\beta_t$
in $[-K,K]$ such that
\begin{equation}
 \sup_{t\in[t_0,T]}
 \max\bigl\{|\tilde c_t-c_t|,
           |\tilde\lambda_t-\lambda_t|,
           |\tilde\gamma_t-\gamma_t|,
           |\tilde\beta_t-\beta_t|\bigr\}\leq\xi.
\label{gbnn:time-errors}
\end{equation}
Its depth and sparsity are $O_\Theta((1+\log\xi^{-1})^2)$.
To see this, set
\[
 b:=\frac12\min\{\rho_{t_0},1,\vartheta^2\}>0,
 \qquad B_0:=2\max\{1,\vartheta^2\}.
\]
Since $\rho_t\in[\rho_{t_0},1]$ and
$v_t=(1-e^{-t})+\vartheta^2e^{-t}$ is a convex combination of $1$ and
$\vartheta^2$, both lie in $[b,B_0]$, an interval determined only by
$\Theta$.

Take an auxiliary accuracy $\xi_0\in(0,1/2)$.  Apply
$\operatorname{Exp}_{\max\{1,T\},\xi_0}$ to $t$ and
$\operatorname{Exp}_{\max\{1,T/2\},\xi_0}$ to $t/2$, and clip the outputs to
$[e^{-T},e^{-t_0}]$ and $[e^{-T/2},e^{-t_0/2}]$, respectively.
The resulting $\tilde u_t$ and $\tilde a_t$ approximate $e^{-t}$
and $a_t$ with error at most $\xi_0$.  Define
\[
 \tilde\rho_t=1-\tilde u_t,\qquad
 \tilde v_t=1+(\vartheta^2-1)\tilde u_t.
\]
They remain in $[b,B_0]$ and satisfy
$|\tilde\rho_t-\rho_t|\leq\xi_0$ and
$|\tilde v_t-v_t|\leq|\vartheta^2-1|\xi_0$.

Using the reciprocal module in Lemma~\ref{gbnn:lem-arithmetic}, set
\[
 \tilde J_{\rho,t}:=b^{-1}\operatorname{Rec}_{B_0/b,b\xi_0}
                                      (\tilde\rho_t/b),\qquad
 \tilde J_{v,t}:=b^{-1}\operatorname{Rec}_{B_0/b,b\xi_0}
                                      (\tilde v_t/b).
\]
The inputs lie in $[1,B_0/b]$, and $b\xi_0<1/2$ because $b\leq1/2$.
These outputs lie in $[0,b^{-1}]$ and approximate
$\tilde\rho_t^{-1}$ and $\tilde v_t^{-1}$ with error at most $\xi_0$.
To include the input error, use the $b^{-2}$-Lipschitz property of the
reciprocal on $[b,B_0]$.  For example,
\[
\begin{aligned}
 |\tilde J_{\rho,t}-\rho_t^{-1}|
 &\leq |\tilde J_{\rho,t}-\tilde\rho_t^{-1}|
          +|\tilde\rho_t^{-1}-\rho_t^{-1}|\\
 &\leq\xi_0+b^{-2}|\tilde\rho_t-\rho_t|
 \leq(1+b^{-2})\xi_0.
\end{aligned}
\]
The estimate for $\tilde J_{v,t}$ follows in the same way.

We now form the four coefficient approximations before final clipping:
\[
\begin{aligned}
 c_{\xi_0}(t)&=\tfrac d2\bigl[
       \operatorname{Log}_{b,B_0,\xi_0}(\tilde v_t)
       -\operatorname{Log}_{b,B_0,\xi_0}(\tilde\rho_t)\bigr]-\log M,\\
 \lambda_{\xi_0}(t)&=\tilde J_{\rho,t}-\tilde J_{v,t},\\
 \gamma_{\xi_0}(t)&=\operatorname{Mult}_{b^{-1},\xi_0}
                            (\tilde a_t,\tilde J_{\rho,t}),\\
 \beta_{\xi_0}(t)&=\tfrac12
           \operatorname{Mult}_{b^{-1},\xi_0}(\tilde u_t,\tilde J_{\rho,t}).
\end{aligned}
\]
The multiplication inputs lie in $[0,b^{-1}]$, since
$\tilde a_t,\tilde u_t\in[0,1]$ and $b^{-1}\geq2$.
The logarithm is $b^{-1}$-Lipschitz on $[b,B_0]$, so its input errors
are controlled in the same way.  The remaining operations are affine
maps and products of bounded scalars.  Thus each of the four errors is at most
$C_{\mathrm{time}}\xi_0$ for a constant $C_{\mathrm{time}}\geq1$ depending
only on $\Theta$.  Choosing $\xi_0=\xi/C_{\mathrm{time}}$ and clipping
each output to $[-K,K]$ proves \eqref{gbnn:time-errors}.
There are a fixed number of scalar modules, so the arithmetic lemma
also gives the claimed depth and sparsity.

\medskip
\noindent\emph{Step 2: Approximate the ratios $R_{i,t}$ on the cube.}
For each coordinate, the square network constructed in the proof of
Lemma~\ref{gbnn:lem-arithmetic} gives
$R^2\operatorname{Sq}_{\xi/(dR^2)}(|x_j|/R)$.
This output lies in $[0,R^2]$ and approximates $x_j^2$ with error at
most $\xi/d$ on the cube.  Summing over $j$, we obtain $\tilde Z(x)$
with
\begin{equation}
 |\tilde Z(x)-\|x\|_2^2|\leq\xi,
 \qquad 0\leq\tilde Z(x)\leq dR^2\leq Q^2.
\label{gbnn:square-sum}
\end{equation}
For each $i$, the inner product $x^\top X_i$ is an exact affine
function of $x$, and $\|X_i\|_2^2$ is a fixed sample-dependent
coefficient.  Both $|x^\top X_i|$ and $\|X_i\|_2^2$ are at most $Q^2$.
Unless stated otherwise, all subsequent products use
$\operatorname{Mult}_{B_{\mathrm{op}},\xi}$.
Write $\ell_i(x,t):=\log R_{i,t}(x)$ and approximate it by
\begin{equation}
 \tilde\ell_i(x,t):=\tilde c_t
      -\frac12\operatorname{Mult}_{B_{\mathrm{op}},\xi}
                    (\tilde\lambda_t,\tilde Z(x))
      +\operatorname{Mult}_{B_{\mathrm{op}},\xi}
                    (\tilde\gamma_t,x^\top X_i)
      -\tilde\beta_t\|X_i\|_2^2.
\label{gbnn:approx-logit}
\end{equation}
The last term is exact multiplication by a fixed coefficient and hence
an affine operation.
Equations \eqref{gbnn:time-errors} and \eqref{gbnn:square-sum} give
\begin{equation}
 |\tilde\ell_i-\ell_i|
 \leq\xi+\tfrac12(1+K+Q^2)\xi
                +(1+Q^2)\xi+Q^2\xi
 \leq6KQ^2\xi.
\label{gbnn:logit-error}
\end{equation}
Also, $\ell_i\leq0$ by \eqref{gbnn:ratio-bounds}, while
\eqref{gbnn:log-ratio} gives $|\ell_i|\leq4KQ^2$.
Set $H:=4KQ^2$ and let
\[
 \tilde\ell_i^{\,\mathrm{cl}}
       :=\operatorname{clip}_{[-H,0]}(\tilde\ell_i),
 \qquad
 \tilde R_i:=\operatorname{Exp}_{H,\xi}(-\tilde\ell_i^{\,\mathrm{cl}}).
\]
Clipping does not increase the error in \eqref{gbnn:logit-error}.
Since $-\tilde\ell_i^{\,\mathrm{cl}},-\ell_i\in[0,H]$ and
$z\mapsto e^{-z}$ is $1$-Lipschitz on $[0,\infty)$, adding the
exponential-network error $\xi$ gives
\begin{equation}
 0\leq\tilde R_i\leq1,
 \qquad |\tilde R_i-R_{i,t}(x)|\leq\xi+6KQ^2\xi\leq e_R,
 \qquad e_R:=7KQ^2\xi.
\label{gbnn:ratio-error}
\end{equation}
The last inequality uses $K,Q\geq1$.
All $n$ ratio modules are evaluated in parallel.

\medskip
\noindent\emph{Step 3: Assemble the score correction and bound its error.}
For this step, suppress $(x,t)$ in the notation and write
$U=U_{n,t}(x)$ and $V=V_{n,t}(x)$.  Form the exact affine aggregates
\[
 \tilde U=\sum_{i=1}^n\tilde R_i,
 \qquad \tilde V=\sum_{i=1}^n\tilde R_iX_i.
\]
Then
\begin{equation}
\begin{gathered}
 0\leq\tilde U\leq n,\qquad
 \|\tilde V\|_2\leq nD,\\
 |\tilde U-U|\leq ne_R,\qquad
 \|\tilde V-V\|_2\leq nDe_R.
\end{gathered}
\label{gbnn:aggregate-errors}
\end{equation}
Let $N$ be the exact numerator and $\overline N$ its approximation
using exact products:
\[
 N:=\gamma_tV-\lambda_txU,
 \qquad
 \overline N:=\tilde\gamma_t\tilde V
                 -\tilde\lambda_t x\tilde U.
\]
Using $U\leq n$, $\|V\|_2\leq nD$ and $D+\|x\|_2\leq Q$, we have
\begin{align}
 \|\overline N-N\|_2
 &\leq K\|\tilde V-V\|_2+\xi\|V\|_2
       +K\|x\|_2\,|\tilde U-U|+\xi\|x\|_2U\notag\\*
 &\leq nQ(Ke_R+\xi)
 \leq8nK^2Q^3\xi.
\label{gbnn:formal-numerator-error}
\end{align}
Replacing these products by multiplication networks gives the
implemented numerator $\tilde N$:
\begin{equation}
\begin{gathered}
 \tilde w
       =\operatorname{Mult}_{B_{\mathrm{op}},\xi}
                        (\tilde\lambda_t,\tilde U),\\
 \tilde N_j
       =\operatorname{Mult}_{B_{\mathrm{op}},\xi}
                        (\tilde\gamma_t,\tilde V_j)
         -\operatorname{Mult}_{B_{\mathrm{op}},\xi}(\tilde w,x_j),
 \qquad 1\leq j\leq d.
\end{gathered}
\label{gbnn:neural-numerator}
\end{equation}
All multiplication inputs are in range: $K,n,nD,Q\leq B_{\mathrm{op}}$,
and the first product gives $|\tilde w|\leq Kn+\xi\leq B_{\mathrm{op}}$.

The two final products contribute at most $2\xi$, and the error in
$\tilde w$ contributes at most $|x_j|\xi$.  Hence
\[
\begin{gathered}
 |\tilde N_j-\overline N_j|
 \leq\xi+\xi+|x_j|\xi=(2+|x_j|)\xi,\\
 \|\tilde N-\overline N\|_2
       \leq(2\sqrt d+\|x\|_2)\xi\leq3Q\xi.
\end{gathered}
\]
Combining this with \eqref{gbnn:formal-numerator-error} gives
\begin{equation}
 \|\tilde N-N\|_2\leq11nK^2Q^3\xi,
 \qquad \|\tilde N\|_2\leq4KnQ.
\label{gbnn:numerator-bounds}
\end{equation}
For the last bound, use $\|\overline N\|_2\leq KnQ$ and $\xi<1$.

The denominator also improves the bound on the exact correction.
Since $\|V\|_2\leq DU$, we have
\[
 \|G_{n,t}(x)\|_2=\frac{\|N\|_2}{1+U}
 \leq K(D+\|x\|_2)\frac{U}{1+U}\leq KQ.
\]
To compare the two quotients, we use the identity
\[
 \frac{\tilde N}{1+\tilde U}-G_{n,t}(x)
 =\frac{\tilde N-N+(U-\tilde U)G_{n,t}(x)}
        {1+\tilde U}.
\]
The denominator error is thus multiplied by a correction bounded by
$KQ$, avoiding an additional factor of $n$.
Since $1+\tilde U\geq1$, equations
\eqref{gbnn:aggregate-errors} and \eqref{gbnn:numerator-bounds} give
\begin{equation}
\begin{aligned}
 \left\|\frac{\tilde N}{1+\tilde U}-G_{n,t}(x)\right\|_2
 &\leq\|\tilde N-N\|_2+KQ|\tilde U-U|\\
 &\leq11nK^2Q^3\xi+KQne_R
 \leq18nK^2Q^3\xi.
\end{aligned}
\label{gbnn:denominator-error}
\end{equation}

To implement the division, use
\[
 \tilde J:=\operatorname{Rec}_{n+1,\xi}(1+\tilde U),
 \qquad
 \Psi_j(x,t):=\operatorname{Mult}_{B_{\mathrm{op}},\xi}
                          (\tilde J,\tilde N_j),
 \quad 1\leq j\leq d.
\]
These modules are evaluated within their approximation domains:
$1+\tilde U\in[1,n+1]$, $\tilde J\in[0,1]$, and
$|\tilde N_j|\leq4KnQ\leq B_{\mathrm{op}}$.  In particular,
$|\tilde J-(1+\tilde U)^{-1}|\leq\xi$, so
\[
 \left\|\Psi(x,t)-\frac{\tilde N}{1+\tilde U}\right\|_2
 \leq\sqrt d\,\xi+\|\tilde N\|_2\xi
 \leq Q\xi+4KnQ\xi.
\]
Combining this with \eqref{gbnn:denominator-error} yields
\begin{equation}
\begin{aligned}
 \|\Psi(x,t)-G_{n,t}(x)\|_2
 &\leq18nK^2Q^3\xi+Q\xi+4KnQ\xi\\
 &\leq24nK^2Q^3\xi
   =\frac{\eta n}{4(n+1)}\leq\frac{\eta}{4}.
\end{aligned}
\label{gbnn:local-error}
\end{equation}
Every approximation above is applied directly to the explicit score
correction; no approximation of a density is differentiated.

\medskip
\noindent\emph{Step 4: Extend to the whole space and verify complexity and measurability.}
Let $\mathcal C_R$ clip each spatial coordinate to $[-R,R]$ and set
$\Phi_{X_{1:n},\eta}(x,t):=\Psi(\mathcal C_R(x),t)$.
On the cube, \eqref{gbnn:local-error} gives error at most $\eta/4$.
Outside the cube, one clipped coordinate has magnitude $R$, so both
$\|x\|_2$ and $\|\mathcal C_R(x)\|_2$ are at least $R$.
By \eqref{gbnn:localized-tail},
\begin{equation}
\begin{aligned}
 \|\Phi_{X_{1:n},\eta}(x,t)-G_{n,t}(x)\|_2
 &\leq\|\Psi(\mathcal C_R(x),t)-G_{n,t}(\mathcal C_R(x))\|_2\\
 &\quad+\|G_{n,t}(\mathcal C_R(x))\|_2+\|G_{n,t}(x)\|_2
 \leq\frac{3\eta}{4}.
\end{aligned}
\label{gbnn:local-to-global}
\end{equation}
Thus \eqref{gbnn:uniform-approximation} holds on the whole space.

We now bound the width, depth, and sparsity.  From
\eqref{gbnn:localized-tail} and \eqref{gbnn:internal-parameters},
\begin{equation}
 Q\leq C_\Theta\sqrt{\ell_{n,\eta}},
 \qquad
 1+\log H+\log B_{\mathrm{op}}+\log\xi^{-1}
          \leq C_\Theta\ell_{n,\eta}.
\label{gbnn:log-budget}
\end{equation}
For the square modules, the underlying accuracy $\xi/(dR^2)$ gives
$1+\log(dR^2/\xi)\leq C_\Theta\ell_{n,\eta}$.
The square, log-ratio, and final multiplication modules therefore have
depth $O_\Theta(\ell_{n,\eta})$, while the time-coefficient, exponential,
and reciprocal modules have depth $O_\Theta(\ell_{n,\eta}^2)$.
Each scalar arithmetic module has constant width, so its sparsity is
bounded by a constant times its depth.  With $O(n+d)$ parallel branches and only a
fixed number of serial stages, these bounds give the width and depth
estimates in \eqref{gbnn:complexity}.

For sparsity, the two sample-dependent affine transformations
\[
 x\longmapsto(x^\top X_i)_{i=1}^n,
 \qquad (r_i)_{i=1}^n\longmapsto\sum_{i=1}^n r_iX_i
\]
use $O(nd)$ nonzero parameters in total.  The coefficients $\|X_i\|_2^2$
in \eqref{gbnn:approx-logit} require another $O(n)$ parameters.
To connect modules, store a signed scalar $z$ as $(z_+,(-z)_+)$ and
recover it by subtraction.  Passing $m$ such scalars through $k$ layers
costs $O(mk)$ parameters and $O(m)$ width, without multiplying consecutive
dense affine matrices.  The arithmetic modules and the channels needed
to store intermediate values and align branch depths thus use
$O_\Theta((n+d)\ell_{n,\eta}^2)$ parameters in total.
The coordinatewise input clipping adds $O(d)$ parameters at constant
depth.  Consequently,
\[
 S(\Phi)\leq Cnd+C_\Theta(n+d)\ell_{n,\eta}^2,
\]
which completes \eqref{gbnn:complexity}.

For fixed $\Theta,n,\eta$, the ranges, clipping thresholds, accuracies,
and integer depths are independent of the sample locations.  Hence the
architecture and allowed parameter positions can be fixed across samples.
The parameters depend continuously on $X_{1:n}$ through the entries of
$X_i$ and $\|X_i\|_2^2$, so the network output is jointly continuous in
the sample and its inputs.  The same holds for $G_{n,t}$ by the smoothness
and positivity of the underlying Gaussian mixture.  Hence their difference
is jointly continuous, and continuity in $x$ gives
\[
 \sup_{x\in\mathbb R^d}
   \|\Phi_{X_{1:n},\eta}(x,t)-G_{n,t}(x)\|_2
 =\sup_{x\in\mathbb Q^d}
   \|\Phi_{X_{1:n},\eta}(x,t)-G_{n,t}(x)\|_2.
\]
The right-hand side is a countable supremum of continuous functions of
the sample and time, and is therefore measurable.
\end{proofof}

\subsection{Transfer to Robust Score Risk}
\label{gbnn:subsec-defect-proof}

We now relate the neural approximation guarantee above to the robust score
risk. The following result shows that the robustness guarantee for an
OU-compatible score is stable under a uniformly controlled approximation.
\begin{corollary}
\label{cor:ou-score-defect}
Under the assumptions of Proposition~\ref{prop:ou-score-oracle}, consider
a measurable, possibly random estimator
\[
  s_\theta(x,t)=\nabla\log\hat p_t(x)+r_\theta(x,t),
  \qquad \hat p_t=P_t\hat\mu_0.
\]
Suppose that, for some $\delta\ge0$,
\begin{equation}
  \mathbb E\left[
    \frac{1}{T-t_0}\int_{t_0}^T
      \|r_\theta(\cdot,t)\|_{L^\infty(\mathbb R^d)}^2\,\rmd t
  \right]\le\delta,
  \label{eq:integrated-uniform-score-defect}
\end{equation}
Then, for every
$\varepsilon\ge0$,
\begin{equation}
  \mathcal R_{\varepsilon,p_0}(s_\theta)
  \le \frac{4\varepsilon^2}{(T-t_0)\tau_0}
    +\frac{4}{T-t_0}\,
      \mathbb E\left[D_2\,\left(p_{t_0}\middle\|\hat p_{t_0}\right)\right]
    +2\delta.
\label{eq:ou-score-defect-risk}
\end{equation}
\end{corollary}

\begin{proof}
Fix $\varepsilon\ge0$ and an arbitrary deterministic
$q_0\in\mathcal Q_\varepsilon(p_0)$, and write $q_t=P_tq_0$ and
$\hat s(x,t)=\nabla\log\hat p_t(x)$. The fields $\hat s$ and $r_\theta$
may depend on the same training sample; all expectations below are over
their joint randomness.

For each estimator realization and almost every $t\in[t_0,T]$, the
density $q_t$ has unit mass, so
\[
 \int_{\mathbb R^d}\|r_\theta(x,t)\|_2^2q_t(x)\,\rmd x
 \le \|r_\theta(\cdot,t)\|_{L^\infty(\mathbb R^d)}^2.
\]
Consequently, Tonelli's theorem and
\eqref{eq:integrated-uniform-score-defect} give
\begin{equation}
 \mathbb E\left[
   \frac{1}{T-t_0}\int_{t_0}^T\int_{\mathbb R^d}
   \|r_\theta(x,t)\|_2^2q_t(x)\,\rmd x\,\rmd t
 \right]\le
 \mathbb E\left[
   \frac{1}{T-t_0}\int_{t_0}^T
   \|r_\theta(\cdot,t)\|_{L^\infty(\mathbb R^d)}^2\,\rmd t
 \right]\le\delta.
\label{gbnn:target-defect-bound}
\end{equation}
This bound is uniform over the choice of $q_0$.

Next, the decomposition $s_\theta=\hat s+r_\theta$ and the inequality
$\|u+v\|_2^2\le2\|u\|_2^2+2\|v\|_2^2$ imply, for each realization,
\[
 \|s_\theta(x,t)-\nabla\log q_t(x)\|_2^2
 \le 2\|\hat s(x,t)-\nabla\log q_t(x)\|_2^2
      +2\|r_\theta(x,t)\|_2^2.
\]
Integrating against $q_t(x)\,\rmd x\,\rmd t/(T-t_0)$, taking
expectation, and using \eqref{gbnn:target-defect-bound} yields
\[
 \mathbb E[\mathcal L_{q_0}(s_\theta)]
 \le 2\mathbb E[\mathcal L_{q_0}(\hat s)]+2\delta.
\]
Taking the supremum over $q_0\in\mathcal Q_\varepsilon(p_0)$, outside
the expectation as in \eqref{eq:robust-shifted-score-risk}, therefore gives
\[
 \mathcal R_{\varepsilon,p_0}(s_\theta)
 \le 2\mathcal R_{\varepsilon,p_0}(\hat s)+2\delta.
\]
Finally, the oracle bound \eqref{eq:generic-ou-score-oracle} and
$\tau_0=(1-e^{-t_0})/e^{-t_0}$ give
\[
 \mathcal R_{\varepsilon,p_0}(s_\theta)
 \le \frac{4\varepsilon^2}{(T-t_0)\tau_0}
     +\frac{4}{T-t_0}\,
       \mathbb E\left[D_2\left(p_{t_0}\middle\|\hat p_{t_0}\right)\right]
     +2\delta,
\]
which is \eqref{eq:ou-score-defect-risk}.
\end{proof}

We next apply Corollary~\ref{cor:ou-score-defect} to the neural approximation
constructed in Proposition~\ref{gbnn:prop-approximation}.
Let $X_1,\ldots,X_n\overset{\mathrm{i.i.d.}}{\sim}p_0$, with $p_0$
satisfying Assumption~\ref{assum:support}. Recall that the
Gaussian-blanket score can be written as
\[
\hat s_n^\vartheta(x,t)
=
-\frac{x}{v_t}
+
G_{n,t}(x),
\]
where the Gaussian term $-x/v_t$ is retained exactly and only the
sample-dependent correction $G_{n,t}$ is approximated by a ReLU network.

For every $\eta\in(0,1/2)$,
Proposition~\ref{gbnn:prop-approximation} constructs a sample-dependent ReLU
network $\Phi_{X_{1:n},\eta}$ satisfying
\[
\sup_{x\in\mathbb R^d,\;t\in[t_0,T]}
\left\|
\Phi_{X_{1:n},\eta}(x,t)-G_{n,t}(x)
\right\|_2
\le
\eta.
\]
Define the corresponding neural score estimator by
\begin{equation}
\hat s_{n,\eta}^{\mathrm{NN}}(x,t)
:=
-\frac{x}{v_t}
+
\Phi_{X_{1:n},\eta}(x,t).
\label{eq:neural-gaussian-blanket-score}
\end{equation}
Then
\[
\sup_{x\in\mathbb R^d,\;t\in[t_0,T]}
\left\|
\hat s_{n,\eta}^{\mathrm{NN}}(x,t)
-
\hat s_n^\vartheta(x,t)
\right\|_2
\le
\eta.
\]
Hence \eqref{eq:integrated-uniform-score-defect} holds with
$\delta=\eta^2$, and the proof of Corollary~\ref{cor:ou-score-defect} gives
\begin{equation}
\mathcal R_{\varepsilon,p_0}
\left(
\hat s_{n,\eta}^{\mathrm{NN}}
\right)
\le
2\mathcal R_{\varepsilon,p_0}
\left(
\hat s_n^\vartheta
\right)
+
2\eta^2.
\label{eq:neural-risk-transfer-basic}
\end{equation}

Applying Minkowski's inequality before taking the supremum over $q_0$
gives the sharper bound
\begin{equation}
\sqrt{
\mathcal R_{\varepsilon,p_0}
\left(
\hat s_{n,\eta}^{\mathrm{NN}}
\right)
}
\le
\sqrt{
\mathcal R_{\varepsilon,p_0}
\left(
\hat s_n^\vartheta
\right)
}
+
\eta.
\label{eq:neural-risk-transfer-minkowski}
\end{equation}
With $\vartheta=D/\sqrt d$ and $\tau_0\le1$,
Theorem~\ref{thm:finite-sample-robust-kde} yields
\begin{equation}
\sqrt{
\mathcal R_{\varepsilon,p_0}
\left(
\hat s_{n,\eta}^{\mathrm{NN}}
\right)
}
\le
\left[
\frac{C_{d,D}}
{(n+1)(T-t_0)\tau_0^{d/2}}
+
\frac{2\varepsilon^2}
{(T-t_0)\tau_0}
\right]^{1/2}
+
\eta.
\label{eq:neural-robust-rate}
\end{equation}
Thus, the neural approximation preserves the Gaussian-blanket robustness rate,
up to a constant factor, provided that $\eta^2$ is bounded by the sum of
the statistical and distribution-shift terms.

\paragraph{Scope of the neural representation.}
The approximation result is constructive: the network parameters are specified
from the observed sample, and the width, depth, and sparsity are controlled
explicitly in Proposition~\ref{gbnn:prop-approximation}. The result is a
realizability statement and does not imply that a particular score-matching
training procedure recovers the constructed parameters. Moreover, the
approximation bounds are stated for fixed $0<t_0<T<\infty$; their constants
may deteriorate as $t_0\to 0$ or $T\to\infty$.

\section{Numerical Illustration}
\label{app:three-dimensional-experiment}

We provide a controlled numerical study illustrating the finite-sample and
distribution-shift effects predicted by our theory. We consider a full-rank
three-dimensional example, which permits accurate numerical evaluation of both
the time-averaged score risk and the positive-time KL error while retaining a
nontrivial distribution shift setting.

\paragraph{Source distribution and shift family.}
We set $d=3$ and $D=1.5$, and consider the four support points
\[
\mu_1=\frac{D}{\sqrt3}(1,1,1),\qquad
\mu_2=\frac{D}{\sqrt3}(1,-1,-1),
\]
\[
\mu_3=\frac{D}{\sqrt3}(-1,1,-1),\qquad
\mu_4=\frac{D}{\sqrt3}(-1,-1,1).
\]
The reference distribution is
\[
p_0
=
0.45\delta_{\mu_1}
+0.30\delta_{\mu_2}
+0.15\delta_{\mu_3}
+0.10\delta_{\mu_4}.
\]
The support points are noncoplanar and lie in $\mathcal B(0,D)$, so $p_0$
has full-rank covariance.

To generate distribution shifts, let $\phi=(1+\sqrt5)/2$ and let
$v_1,\ldots,v_{12}$ be the normalized vectors obtained from
\[
(0,\pm1,\pm\phi),\qquad
(\pm1,\pm\phi,0),\qquad
(\pm\phi,0,\pm1).
\]
For $k=1,\ldots,12$, define the translated distributions
\[
q_{0,k}^{\varepsilon}
:=
(x\mapsto x+\varepsilon v_k)_\#p_0,
\qquad
q_{t,k}^{\varepsilon}:=P_tq_{0,k}^{\varepsilon}.
\]
The translation coupling gives
$\mathsf W_2(p_0,q_{0,k}^{\varepsilon})\le\varepsilon$, while the difference
of means gives the reverse inequality; hence
\[
\mathsf W_2(p_0,q_{0,k}^{\varepsilon})=\varepsilon.
\]

We evaluate the worst-case error over these twelve translation directions:
\begin{align}
R_{12}^{\rm tr}(n,\varepsilon)
&:=
\max_{1\le k\le12}
\mathbb E_{X_{1:n}}
\mathcal L_{q_{0,k}^{\varepsilon}}
(\hat s_n^\vartheta),\\
K_{12}(n,\varepsilon)
&:=
\max_{1\le k\le12}
\mathbb E_{X_{1:n}}
\mathrm{KL}\left(
q_{t_0,k}^{\varepsilon}
\middle\|
\hat p_{n,t_0}^{\vartheta}
\right),
\end{align}
where the expectation over the source samples is taken before maximizing over
the translation directions.

\paragraph{Estimator and theoretical bounds.}
We take
\[
t_0=0.05,\qquad T=1,
\]
and consider 23 logarithmically spaced shift radii in $[10^{-5},0.2]$.
The Gaussian-blanket estimator follows Section~\ref{sec:finite-sample}, with
\[
\vartheta=\frac{D}{\sqrt d},
\qquad
M=M_{\vartheta,t_0}
=
\exp\left(\frac{D^2}{2\vartheta^2}\right)
\left(1+\frac{\vartheta^2}{\tau_0}\right)^{d/2}.
\]
For each source-sample realization, we generate
\[
(N_1,\ldots,N_4)
\sim
\operatorname{Multinomial}
(n;(0.45,0.30,0.15,0.10)),
\]
which exactly represents the empirical measure for the finite-support source.
The resulting density estimator is
\[
\hat p_{n,t}^{\vartheta}(x)
=
\frac{
\sum_{j=1}^4N_j\varphi_{\rho_t}(x-a_t\mu_j)
+
M\varphi_{\rho_t+a_t^2\vartheta^2}(x)
}
{n+M}.
\]

The theoretical comparison curves are
\begin{equation}
U_K(n,\varepsilon)
:=
\log\left(1+\frac{M-1}{n+1}\right)
+\frac{\varepsilon^2}{\tau_0},
\label{eq:experiment-explicit-bounds}
\end{equation}
and
\[
U_S(n,\varepsilon)
:=
\frac{2}{T-t_0}U_K(n,\varepsilon),
\]
as implied by Propositions~\ref{prop:ou-score-oracle}
and~\ref{prop:reverse-chi-square}.

\paragraph{Numerical evaluation.}
For the main experiment, we use
\[
n\in\{5\times10^3,\,5\times10^5,\,5\times10^7\}.
\]
For each $n$, we average over 12 independent source-sample replicates, reusing
the same replicates across all shift radii and translation directions.

Since the Gaussian-blanket score is OU-compatible, the time-averaged score
risk is evaluated through the endpoint identity
\[
\mathcal L_{q_{0,k}^{\varepsilon}}
(\hat s_n^\vartheta)
=
\frac{2}{T-t_0}
\left[
\mathrm{KL}\left(
q_{t_0,k}^{\varepsilon}
\middle\|
\hat p_{n,t_0}^{\vartheta}
\right)
-
\mathrm{KL}\left(
q_{T,k}^{\varepsilon}
\middle\|
\hat p_{n,T}^{\vartheta}
\right)
\right],
\]
which follows from Lemma~\ref{lem:relative-de-bruijn-ou}. The KL terms at
$t_0$ and $T$ are computed using component-stratified order-32 tensor
Gauss--Hermite quadrature. Under self-consistent initialization, the reverse
sampler output at forward time $t_0$ is exactly
$\hat p_{n,t_0}^{\vartheta}$, allowing the KL error to be evaluated directly
without simulating the reverse SDE.

Figure~\ref{fig:three-dimensional-score-and-kl} reports the resulting
time-averaged score risk and positive-time KL error. Increasing the source
sample size moves both quantities toward their known-source counterparts,
reflecting the reduction in finite-sample estimation error. As the shift
radius increases, the distribution-shift contribution becomes dominant, with
the observed behavior consistent with the quadratic dependence on
$\varepsilon$ predicted by
Theorem~\ref{thm:finite-sample-robust-kde} and
\eqref{eq:self-consistent-kl-sampling-bound}. Across the displayed
configurations, the empirical errors remain below the corresponding
theoretical upper bounds.

\begin{figure}[htbp]
  \centering
  \includegraphics[width=\textwidth]
  {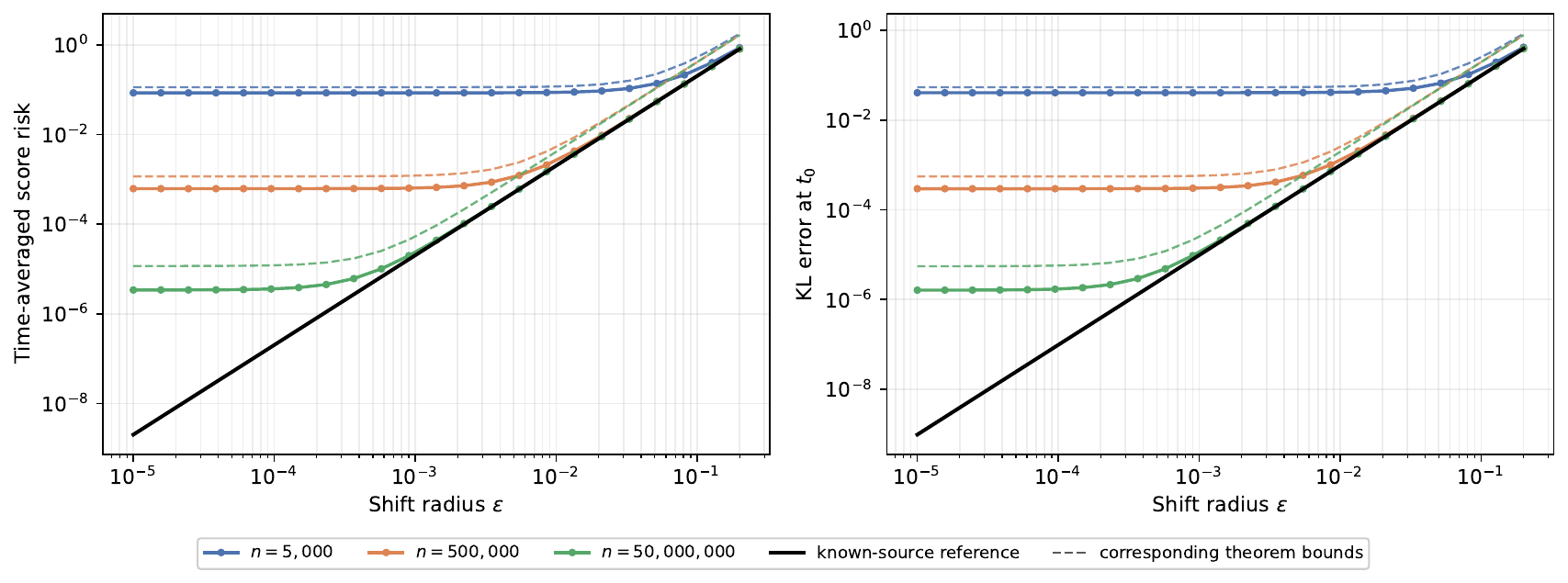}
  \caption{
  Time-averaged score risk (left) and positive-time KL error (right) in the
  controlled three-dimensional experiment. Increasing the source sample size
  reduces the finite-sample contribution, while the distribution-shift
  contribution dominates at larger shift radii. Dashed curves show the
  corresponding theoretical upper bounds.
  }
  \label{fig:three-dimensional-score-and-kl}
\end{figure}

The black curves denote known-source references, using
$\nabla\log p_t$ for the score risk and $p_{t_0}$ for the KL error.
The shaded regions are pointwise 95\% paired-bootstrap intervals over the
12 source-sample replicates, with the maximization over translation directions
recomputed in each of 4,000 bootstrap resamples.

\paragraph{Reference-domain finite-sample behavior.}
To isolate the finite-sample contribution, Figure~
\ref{fig:three-dimensional-blanket-sampling-floor} reports the score risk at
$\varepsilon=0$ for seven sample sizes between $5\times10^3$ and
$5\times10^9$. The risk is evaluated using the same endpoint identity and
shows approximately $n^{-1}$ decay over the displayed range. This experiment
is intended only as a finite-sample illustration and does not probe the
worst-case $\tau_0^{-d/2}$ dependence or the full minimax rate.

\begin{figure}[htbp]
  \centering
  \includegraphics[width=0.82\textwidth]
  {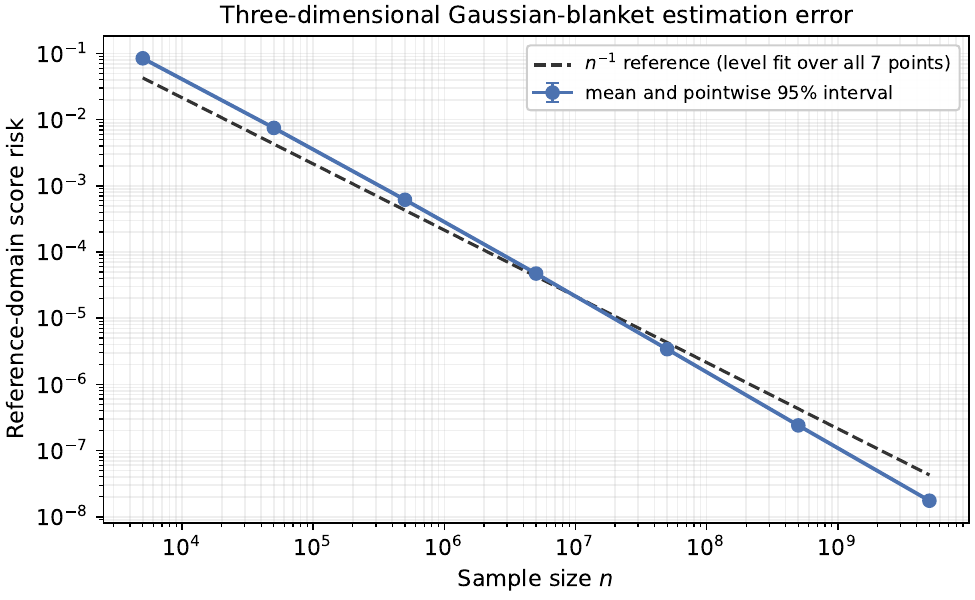}
  \caption{
  Reference-domain Gaussian-blanket score risk as a function of source sample
  size. Points are averaged over 256 independent multinomial-count replicates;
  error bars indicate $\pm1.96$ standard errors, and the dashed line shows an
  $n^{-1}$ reference.
  }
  \label{fig:three-dimensional-blanket-sampling-floor}
\end{figure}

\paragraph{Quadrature validation.}
As a numerical stability check, increasing the endpoint quadrature order from
32 to 48 changes the mean score risks by at most $0.11\%$ over the
configurations examined.

\end{document}